\documentclass[reqno,11pt]{amsart}
\usepackage{setspace,tikz,xcolor,mathrsfs,listings,multicol,amssymb,amsfonts}
\usepackage{pgfplots}
\usepackage{rotating}
\usepackage[vcentermath]{youngtab}
\usepackage{tkz-base}
\usepackage{tkz-euclide}
\usepackage[margin=1in,includefoot,footskip=30pt]{geometry}
\usepackage{enumerate}
\usepackage{booktabs}
\usepackage[centertableaux]{ytableau}
\usepackage[all,cmtip]{xy}
\usetikzlibrary{arrows,matrix,math}
\tikzset{tab/.style={matrix of math nodes,column sep=-.35, row sep=-.35,text height=7pt,text width=7pt,align=center,inner sep=2,font=\footnotesize}}

\newif\iftikz
\tikztrue

\newif\iflimitshapes
\limitshapestrue  

\usepackage[colorlinks=true, pdfstartview=FitV, linkcolor=blue, citecolor=blue, urlcolor=blue]{hyperref}

\newcommand{\oeis}[1]{\href{https://oeis.org/#1}{OEIS \texttt{#1}}}

\newcommand{\g}{\mathfrak{g}}
\newcommand{\fsl}{\mathfrak{sl}}
\newcommand{\gl}{\mathfrak{gl}}
\newcommand{\so}{\mathfrak{so}}
\newcommand{\spn}{\mathfrak{sp}}

\newcommand{\fw}{\Lambda} 
\newcommand{\coroot}{\alpha^{\vee}} 

\newcommand{\inner}[2]{\langle #1, #2 \rangle}
\newcommand{\iso}{\cong}

\newcommand{\qbinom}[3]{\genfrac{[}{]}{0pt}{}{#1}{#2}_{#3}}
\newcommand{\abs}[1]{\lvert #1 \rvert}
\newcommand{\Abs}[1]{\lVert #1 \rVert}
\newcommand{\absval}[1]{\left\lvert #1 \right\rvert}

\newcommand{\cp}{{\color{blue}+}}
\newcommand{\cm}{{\color{darkred}-}}
\newcommand{\zero}{\mathbf{0}}
\newcommand{\bigO}{\mathcal{O}} 
\newcommand{\branch}{\downarrow}  
\renewcommand{\Re}{\operatorname{Re}}
\renewcommand{\Im}{\operatorname{Im}}

\newcommand{\ds}{\; \mathrm{d}s}
\newcommand{\dt}{\; \mathrm{d}t}
\newcommand{\dx}{\; \mathrm{d}x}
\newcommand{\dy}{\; \mathrm{d}y}
\newcommand{\dz}{\; \mathrm{d}z}

\newcommand{\te}{\widetilde{e}}
\newcommand{\tf}{\widetilde{f}}

\DeclareMathOperator{\supp}{supp} 
\DeclareMathOperator{\wt}{wt} 
\DeclareMathOperator{\ch}{ch} 
\DeclareMathOperator{\sig}{sig} 
\DeclareMathOperator{\pv}{p.v.} 
\DeclareMathOperator{\Cliff}{Cliff} 
\DeclareMathOperator{\End}{End} 
\newcommand{\Spin}{\operatorname{Spin}} 
\newcommand{\Pin}{\operatorname{Pin}} 
\newcommand{\Sp}{Sp}
\newcommand{\SO}{SO}
\newcommand{\Or}{O}  
\newcommand{\SU}{SU}
\newcommand{\SL}{SL}
\newcommand{\GL}{GL}

\newcommand{\mcB}{\mathcal{B}}
\newcommand{\mcD}{\mathcal{D}}

\newcommand{\mcH}{\mathcal{H}}

\newcommand{\mcN}{\mathcal{N}}

\newcommand{\pp}{\mathbf{p}}
\newcommand{\uu}{\mathbf{u}}
\newcommand{\vv}{\mathbf{v}}
\newcommand{\xx}{\mathbf{x}}
\newcommand{\yy}{\mathbf{y}}

\newcommand{\ZZ}{\mathbb{Z}}

\newcommand{\RR}{\mathbb{R}}
\newcommand{\CC}{\mathbb{C}}

\newcommand{\bon}{\overline{1}}
\newcommand{\btw}{\overline{2}}

\newcommand{\bn}{\overline{n}}

\definecolor{darkred}{rgb}{0.7,0,0} 
\newcommand{\defn}[1]{{\color{darkred}\emph{#1}}} 

\definecolor{UQgold}{RGB}{196, 158, 54} 
\definecolor{UQpurple}{RGB}{73, 7, 94} 

\usepackage{listings}
\lstdefinelanguage{Sage}[]{Python}
{morekeywords={False,sage,True},sensitive=true}
\definecolor{dblackcolor}{rgb}{0.0,0.0,0.0}
\definecolor{dbluecolor}{rgb}{0.01,0.02,0.7}
\definecolor{dgreencolor}{rgb}{0.2,0.4,0.0}
\definecolor{dgraycolor}{rgb}{0.30,0.3,0.30}

\theoremstyle{plain}
\newtheorem{thm}{Theorem}[section]
\newtheorem{lemma}[thm]{Lemma}
\newtheorem{conj}[thm]{Conjecture}
\newtheorem{prop}[thm]{Proposition}
\newtheorem{cor}[thm]{Corollary}
\theoremstyle{definition}

\newtheorem{ex}[thm]{Example}
\newtheorem{remark}[thm]{Remark}
\newtheorem{problem}[thm]{Problem}
\numberwithin{equation}{section}

\usepackage[colorinlistoftodos]{todonotes}

\begin{document}
\title[Skew Howe duality and limit shapes]{Skew Howe duality and limit shapes of Young diagrams}

\author[A.~Nazarov]{Anton Nazarov}
\address[A.~Nazarov]{Department of High Energy and Elementary Particle Physics, St.\ Petersburg State University, Ulyanovskaya 1, St.\ Petersburg, Russia, 198504}
\address[A.~Nazarov]{Yanqi Lake Beijing Institute of Mathematical Sciences and Applications, No.~544, Hefangkou Village, Huaibei Town, Huairou District, Beijing 101408, China}
\email{antonnaz@gmail.com}
\urladdr{http://hep.spbu.ru/index.php/en/1-nazarov}

\author[O.~Postnova]{Olga Postnova}
\address[O.~Postnova]{Laboratory of Mathematical Problems of Physics, St.\ Petersburg Department of Steklov Mathematical Institute of Russian Academy of Sciences, 27 Fontanka, St.\ Petersburg, Russia, 191023}
\email{postnova.olga@gmail.com}

\author[T.~Scrimshaw]{Travis Scrimshaw}
\address[T.~Scrimshaw]{School of Mathematics and Physics, The University of Queensland, St.\ Lucia, QLD 4072, Australia}
\email{tcscrims@gmail.com}
\curraddr{Department of Mathematics, Hokkaido University, 5 Ch\=ome Kita 8 J\=onishi, Kita Ward, Sapporo, Hokkaid\=o 060-0808, Japan}
\urladdr{https://tscrim.github.io/}


\begin{abstract}
We consider the skew Howe duality for the action of certain dual pairs of Lie groups $(G_1, G_2)$ on the exterior algebra $\bigwedge(\CC^{n} \otimes \CC^{k})$ as a probability measure on Young diagrams by the decomposition into the sum of irreducible representations.
We prove a combinatorial version of this skew Howe duality for the pairs $(\GL_{n}, \GL_{k})$, $(\SO_{2n+1}, \Pin_{2k})$, $(\Sp_{2n}, \Sp_{2k})$, and $(\Or_{2n}, \SO_{k})$ using crystal bases, which allows us to interpret the skew Howe duality as a natural consequence of lattice paths on lozenge tilings of certain partial hexagonal domains. 
The $G_1$-representation multiplicity is given as a determinant formula using the Lindstr\"om--Gessel--Viennot lemma and as a product formula.
These admit natural $q$-analogs that we show equals the $q$-dimension of a $G_2$-representation (up to an overall factor of $q$), giving a refined version of the combinatorial skew Howe duality.
Using these product formulas (at $q =1$), we take the infinite rank limit and prove the diagrams converge uniformly to the limit shape.
\end{abstract}

\maketitle

\tableofcontents

\section{Introduction}
\label{sec:introduction}

The study of multiplicity-free actions of reductive dual pairs of groups has been very fruitful and is now usually called Howe duality~\cite{howe1995perspectives}.
The most well-known of such dualities is the $(\GL_{n},\GL_{k})$ duality from the action of $\GL_{n}\times \GL_{k}$ on the symmetric space $S(\CC^{n}\otimes \CC^{k})$.
Howe duality generalizes Schur--Weyl duality (see, \textit{e.g.},~\cite[Sec.~2.4]{howe1995perspectives}), and it is described combinatorially by the Robinson--Schensted--Knuth (RSK) algorithm, which bijectively maps a multiset with elements in $\{1, \dotsc, n\} \times \{1, \dotsc, k\}$ to a pair of semistandard (Young) tableaux of the same shape $\lambda$.
Furthermore, by using Kashiwara's crystal bases~\cite{K90,K91}, we can encode the $\GL_n \times \GL_k$ action on the multisets and on the irreducible representations using semistandard tableaux with RSK being the crystal isomorphism (see also~\cite{Lascoux03,vanLeeuwen06} and the relation with coplactic operators~\cite[Ch.~5]{Lothaire02}).

This duality is related to the most famous result in asymptotic representation theory, the Vershik--Kerov--Logan--Shepp limit shape~\cite{logan1977variational,vershik1977asymptotics}.
We can embed the regular representation of $S_k \subseteq \GL_k$ into $S^k(\CC^{k}\otimes \CC^{k})$ by using two-line representation of the permutations: For a permutation $\sigma$, we have $\prod_{i=1}^k e_i \otimes e_{\sigma(i)}$.
By restriction, RSK then bijectively maps a permutation of $k$ elements to a pair of \emph{standard} Young tableaux of the same shape $\lambda$.
Since there are more permutations of $k$ than partitions of $k$, the image of uniform random permutations under RSK defines the famous Plancherel probability measure on partitions of $k$.
This measure has the probability of $\lambda$ given by the ratio of the square of $f^{\lambda}$, the number of standard Young tableaux of shape $\lambda$, and $k!$.
We can reinterpret this using representation theory (over $\CC$) as the regular representation has dimension $k!$ and decomposes into all of its irreducible representations $S^{\lambda}$ with multiplicity equal to $\dim S^{\lambda}$.
For $S_k$, irreducible modules $S^{\lambda}$ are the Specht modules, where $\lambda$ ranges over all partitions of $k$ with $\dim S^{\lambda} = f^{\lambda}$.
Hence, the Plancherel measure is given by
\[
\mu^P_{k}(\lambda) = \frac{(f^{\lambda})^2}{k!}= \frac{(\dim S^{\lambda})^{2}}{k!}.
\]
In the limit $k\to\infty$, the Plancherel measure is concentrated on the Vershik--Kerov--Logan--Shepp limit shape computed in~\cite{logan1977variational,vershik1977asymptotics}.

Next, Schur--Weyl duality is described as the decomposition of commuting actions of $S_k$ and $\GL_n$ on $(\CC^n)^{\otimes k}$.
In the paper~\cite{kerov1986asymptotic}, S.V.~Kerov used Schur--Weyl duality\footnote{
In~\cite{kerov1986asymptotic}, the group $\SU_n$ was considered, but the complexification is $\SL_n$ and this does not change the dimension of the irreducible representations, which are also equal to those of $\GL_n$.
} to construct a similar measure on Young diagrams $\lambda$ of size $k$ as
\[
\mu^{SW}_{n,k}(\lambda) = \frac{\dim S^{\lambda} \cdot \dim V_{\GL_n}(\lambda)}{n^{k}}.
\]
We can see this formula through RSK by embedding $(\CC^n)^{\otimes k}$ into $S^k(\CC^n \otimes \CC^k)$ by $v_1 \otimes \cdots \otimes v_k \mapsto (v_1 \otimes e_1) \dotsm (v_k \otimes e_k)$, where $\{e_1, \dotsc, e_k\}$ is the standard basis of $\CC^k$.
In the limit $n,k\to\infty$ such that $k/n \to \mathrm{const}$, the Vershik--Kerov--Logan--Shepp limit shape is recovered. 
This measure was also studied by Biane~\cite{Biane01}, but under the limit $\sqrt{k}/n \to \mathrm{const}$.
Under this scaling, a family of curves depending on this constant is obtained, and the Vershik--Kerov--Logan--Shepp limit shape is when this constant is $0$.

Returning back to the general $(\GL_n, \GL_k)$ duality, we note that the symmetric space is infinite dimensional, so it does not allow an immediate measure on partitions. If we restrict to $S^m(\CC^n \otimes \CC^k)$ or so that the degree of $e_1 \in \CC^n$ to be at most $m$, we can define a probability measures on Young diagrams of size $m$ or on those with $\lambda_1 \leq m$ (or $\lambda$ is contained in a $\min(n, k) \times m$ rectangle), respectively, as
\[
\mu^{(m)}_{n,k}(\lambda) = \frac{\dim V_{\GL_n}(\lambda) \dim V_{\GL_k}(\lambda)}{\displaystyle \binom{nk+m-1}{m}},
\qquad\qquad
\mu^{\square m}_{n,k}(\lambda) = \frac{\dim V_{\GL_n}(\lambda) \dim V_{\GL_k}(\lambda)}{\displaystyle \prod_{a=1}^n \prod_{b=1}^m \prod_{c=1}^k \frac{a+b+c-1}{a+b+c-2}},
\]
respectively.
The measure $\mu^{\square m}_{n,k}$ is related to the arctic circle limit shape of lozenge tilings of a hexagon~\cite{BKMM03,CLP98,Gorin08,Petrov14,Petrov15} (see also~\cite{Gorin21}) by applying RSK, taking the corresponding pair of Gelfand--Tsetlin (GT) patterns, and joining them together to form a plane partition inside of an $n \times m \times k$ box (see, \textit{e.g.},~\cite[Ch.~7]{ECII}; the number of plane partitions in a box is due to MacMahon~\cite{MacMahon96,MacMahon15}) and projecting.
In order to get the full symmetric space, we can take the refined data of the characters instead of taking dimensions, and we obtain a well-defined probability measure by the Cauchy identity of Schur functions:
\begin{equation}
\label{eq:cauchy}
\sum_{\ell(\lambda) \leq \min(n,k)} s_{\lambda}(x_1, \dotsc, x_n) s_{\lambda}(y_1, \dotsc, y_k) = \prod_{i=1}^n \prod_{j=1}^k \frac{1}{1 - x_i y_j}.
\end{equation}
When $n, k \to \infty$, we obtain the famous Schur measure on partitions~\cite{Okounkov00,Okounkov01}.

Another variant of Howe duality is skew Howe duality~\cite[Thm.~4.1.1]{howe1995perspectives}, where there is a multiplicity-free action of a pair of Lie groups $(G_1, G_2)$ on the exterior algebra $\bigwedge\left(\CC^{n}\otimes (\CC^k)^*\right)$.
This is usually proven with the use of the Schur duality, and we have the multiplicity-free decomposition
\[
\bigwedge\left(\CC^{n}\otimes (\CC^k)^* \right) \iso \bigoplus_{\lambda \subseteq k^n} V_{G_{1}}(\lambda)\otimes V_{G_{2}}(\overline{\lambda}'),
\]
where $\overline{\lambda}'$ is the conjugate of the complement diagram of $\lambda$ inside an $n \times k$ rectangle.
One key advantage of the exterior algebra over the symmetric algebra is that it is finite dimensional, which allows us to introduce a probability measure on diagrams
\begin{equation}
\label{eq:skew_howe_measure}
\mu_{n,k}(\lambda) = \frac{\dim V_{G_{1}}(\lambda) \cdot \dim  V_{G_{2}}(\overline{\lambda}')}{2^{nk}}.
\end{equation}
The measure $\mu_{n,k}$ will be the main focus of this paper.
The exterior algebra can be also seen as a tensor power $\left(\bigwedge \CC^{n}\right)^{\otimes k}$, and thus skew Howe duality can be used to provide multiplicity formulas for a tensor power decomposition
\[
\left(\bigwedge \CC^{n}\right)^{\otimes k} \iso \bigoplus_{\lambda \subseteq k^n} M^k(\lambda) \cdot V_{G_{1}}(\lambda),
\]
where $M^k(\lambda) = \dim V_{G_2}(\overline{\lambda}')$.
Hence, if the multiplicity of $V(\lambda)$ in $V^{\otimes k}$ for some $G_1$-representation $V$ equals the dimension of the irreducible $G_2$-representation $V(\overline{\lambda}')$, then we call this \defn{combinatorial skew Howe duality}. 
(We note that this does not imply a (commuting) $G_2$ action, but it says there is a highest weight $G_2$ Kashiwara crystal structure on combinatorial objects counting the multiplicity.)
Moreover, the probability measure becomes $\mu_{n,k}(\lambda) = 2^{-nk} M^k(\lambda) \dim V(\lambda)$.
For skew Howe duality over other fields, see also~\cite{AB95,Moeglin89}.

We look at some known examples at the level of characters, all of which give rise to character measures on partitions.
We first consider the case of $(G_1, G_2) = (\GL_n, \GL_k)$, which can be proven using a variation of RSK~\cite{knuth1970permutations} (called dual RSK in~\cite[Ch.~7]{ECII}) showing pairs of semistandard Young tableaux of shape $\lambda$ and $\lambda'$ are in one-to-one correspondence with the $n \times k$ matrices of zeros and ones, yielding the dual Cauchy identity:
\[
\sum_{\lambda \subseteq k^n} s_{\lambda}(x_1, \dotsc, x_n) s_{\lambda'}(y_1, \dotsc, y_k) = \prod_{i=1}^n \prod_{j=1}^k (1 + x_i y_j).
\]
Here we used $V_{\GL_k}(\overline{\lambda}')^* \iso V_{\GL_k}(\lambda')$ up to a shift of the determinant representation (\textit{cf.}~\cite[Thm.~4.1.1]{howe1995perspectives}).
This has been applied to the random matrix theory with computing the correlations of characteristic polynomials of the unitary group~\cite{bump2006averages} with generalizations to other random matrix ensembles given in~\cite{JKM21}.
This was also studied by Gravner, Tracy, and Widom under the name oriented digital boiling~\cite{GTW01,GTW02,GTW02II}.
Panova and \'Sniady~\cite{panova2018skew} considered the analog of $\mu_{n,k}^{(m)}$, where they consider the exterior power $\bigwedge^{m}(\CC^{n}\otimes\CC^{k})$ with the corresponding probability measure
\[
\mu_{n,k}^{\langle m \rangle}(\lambda)=\frac{\dim V_{GL_{n}}(\lambda)\cdot\dim V_{GL_{k}}(\overline\lambda')}{\binom{nk}{m}}
\]
for the diagrams of $m$ boxes in the $n\times k$ rectangle.
They compute the limit shapes for the limit $n,k,m \to \infty$, $\frac{k}{n}\to\mathrm{const}$, $\frac{m}{nk}\to\mathrm{const}$ by reformulating the problem in terms of the representations of permutation group as the level lines of the limit shape for plane partitions presented in~\cite{pittel2007limit}.
The measure $\mu^{\langle m \rangle}_{n,k}$ has also appeared in~\cite[Sec.~2.1.3]{GTW01} in relation to Johansson's result~\cite{Johansson01} on the Krawtchouk ensemble and is a refinement of the measure $\mu_{n,k}$ from~\eqref{eq:skew_howe_measure} since
\begin{equation}
\label{eq:skew_refinements}
\mu_{n,k}(\lambda) = \sum_{m=0}^{nk} 2^{-nk} \binom{nk}{p} \mu_{n,k}^{\langle m \rangle}(\lambda),
\end{equation}

For $(\Sp_{2n}, \Sp_{2k})$, this yields the following character identity first due to King~\cite{King75} with later proofs due to Jimbo and Miwa~\cite{JM85} and Howe~\cite{howe1995perspectives}:
\begin{equation}
\label{eq:sp_char_prob}
\sum_{\lambda \subseteq k^n} \chi^{\Sp_{2n}}_{\lambda}(x_1, \dotsc, x_n) \chi^{\Sp_{2k}}_{\overline{\lambda}'}(y_1, \dotsc, y_k) = \prod_{i=1}^n \prod_{j=1}^k (x_i + x_i^{-1} + y_j + y_j^{-1}).
\end{equation}
This also has an RSK-like proof~\cite{berele1986schensted,Sundaram90III,Terada93} and has been applied to random matrix theory in~\cite{LO20}.
The case $(G_1, \Sp_{2k})$ was examined in Heo and Kwon~\cite{HK20}, which recovers~\eqref{eq:sp_char_prob} and other identities such as~\cite[Eq.~(1.4)]{HK20}.
Proctor~\cite{proctor1993reflection} also provides proofs of numerous character identities, including skew Howe dualities, using the reflection method.
An RSK-type algorithm has also been used for the orthogonal group by Sundaram~\cite{Sundaram90}.
Generalizations of some of these identities are known, such as using Macdonald polynomials~\cite[p.~329]{Macdonald98}, Koornwinder polynomials~\cite{Okounkov98}, and an extension of continuous $q$-Hermite polynomials~\cite{Nteka18}.

\begin{table}
\[
\begin{array}{cccccc}
\toprule
G_1 & \GL_n & \SO_{2n+1}, k \text{ even} & \SO_{2n+1}, k \text{ odd} & \Sp_{2n} & \Or_{2n}
\\ \midrule
G_2 & \GL_k & \Pin_k & \Sp_{k-1} & \Sp_{2k} & \SO_k
\\ \bottomrule
\end{array}
\]
\caption{The combinatorial skew Howe duality obtained for $V^{\otimes k}$.}
\label{table:results}
\end{table}

In the present paper, we first examine the pairs $(\GL_{n}, \GL_{k})$, $(\SO_{2n+1}, \Pin_{2k})$, $(\Sp_{2n}, \Sp_{2k})$, and $(\Or_{2n}, \SO_{k})$ and prove a natural $q$-analog of combinatorial skew Howe duality.
We begin by looking at the multiplicity $M^k(\lambda)$ of $V(\lambda)$ inside $V^{\otimes k}$, where $V$ is the following representation for the group~$G_1$:
\begin{itemize}
\item[$\GL_n$:] $V = \bigwedge \CC^n$, the exterior algebra of the natural representation;
\item[$\SO_{2n+1}$:] $V$ is the spinor representation;
\item[$\Sp_{2n}$:] $V = \bigwedge \CC^{2n}$, the exterior algebra of the natural representation;
\item[$\SO_{2n}$:] $V$ is the sum of the two nonisomorphic spinor representations, which is irreducible as an $\Or_{2n}$ representation.
\end{itemize}
Our proof uses the crystal basis and the nonintersecting lattice paths approach formulated in~\cite{OS19II} to write $M^k(\lambda)$ as certain determinants of binomial coefficients or Catalan triangle numbers using the Lindstr\"om--Gessel--Viennot (LGV) lemma~\cite{GV85,Lindstrom73}.
Next, we take a natural $q$-deformation of these determinants and use techniques from~\cite{Krat99} to transform the determinant formulas for the multiplicities into product formulas with $q$-integers, which when $q = 1$ is similar to those in the work of Kulish, Lyakhovsky, and Postnova~\cite{KLP12,KLP12II}.
Again using the LGV lemma, we show that the $q$-analogs of our determinant formulas give $\dim_q V(\overline{\lambda}')$, the $q$-dimension of the irreducible $G_2$-representation.
Taking $q = 1$, we obtain the combinatorial skew Howe duality.
We summarize our results in Table~\ref{table:results}.
We remark that the cases for $\SO_N$ when $k$ is an odd power is not a skew Howe duality in the sense we have described above as it does not come from a decomposition of an exterior algebra.
However, this can be described as a type of Howe duality and our product formulas do not depend on the parity of $k$.

While the $q$-analog of combinatorial skew Howe duality was previously known from specializing the aforementioned character formulas, our proofs are new with more of a direct representation theory application.
Furthermore, the $q$-analogs of the determinant formulas are generally new, even for the case $q = 1$, and the product formulas are entirely new except for $q=1$ for $G_1 = \SO_{2n+1}$ in~\cite{KLP12,KLP12II,kulish2012tensor}.
For $(\GL_{n}, \GL_{k})$, the determinant formula was previously obtained in~\cite{EG95} purely combinatorially as a number of certain lattice path, and the $q$-analog was independently shown by Cigler~\cite[Thm.~8]{Cigler21}.
In both of these cases, the connection to the representation theory was not established.
In~\cite{KLP12,KLP12II}, the case $(\SO_{2n+1},\Pin_{2k})$ was derived without noticing the importance of skew Howe duality.
Determinant multiplicity formulas for $(\SO_{2n+1}, \Pin_{2k})$ and $(\Sp_{2n}, \Sp_{2k})$ were shown in~\cite{OS19II} also without noticing the skew Howe duality.
In all of these cases, the $q$-analog of these formulas were not known.

Let us discuss the dependence on the parity of $k$ for the decomposition for $V(\fw_n)^{\otimes k}$ of the spin representation for $\SO_{2n+1}$.
We note that there is an alternating form on $V(\fw_n)$, which means the tensor power can embed in an orthogonal or symplectic space depending on the parity of $k$ by building a symmetric or alternating form, respectively.
Thus, we have an action of $\Pin_k$ or $\Sp_{k-1}$, respectively, since they preserve a symmetric or alternating form (see also~\cite{howe1995perspectives}).
There is also an RSK-type algorithm that recovers the corresponding character identities~\cite[$(\mathsf{D}_x^p\mathsf{B}_y)$, $(\mathsf{B}_x\mathsf{C}_y)$]{proctor1993reflection} due to Benkart and Stroomer~\cite{benkart1991tableaux}.
An analogous RSK-type algorithm for the $\Or_{2n}$ spinor was given by Okada~\cite{okada1993robinson}.
We also note that tensor powers of spin representations has been examined by Rowell and Wenzl~\cite[Lemma~2.1]{RW17}.

We also provide a natural interpretation of the appearance of lattice paths as they have an innate description with lozenge tilings of a certain half hexagonal domain.
Indeed, lozenge tilings of the half hexagon naturally correspond to GT patterns that arise to describe the representations of $\GL_k$, which also correspond to the lattice paths describing $\dim_q V(\overline{\lambda}')$.
By taking a different set of paths, we recover the lattice paths that we used to compute the multiplicity of $V(\lambda)$.
Joining this to be the full hexagon with side lengths alternating between $k$ and $n$ and a seam down the middle encoding $\lambda$, we recover our $\GL_n \times \GL_k$ probability measure (up to the normalization factor of $2^{nk}$).
The other dual pairs arise from imposing extra symmetries on the hexagon from the symmetries on GT patterns described by Proctor~\cite{Proctor94}, which have also been considered by Bufetov and Gorin~\cite[Sec.~3.2]{BG15}.
Similarly, many of the representations we consider can be seen as arising from $\bigwedge \CC^n$ from the branching rule from the inclusion $G_1 \to \GL_n$.
We are using a refined version of the skew Howe duality for $G_1 = \SO_{2n+p}$ arising from the relation $(1 + p) V^{\otimes 2} \iso \bigwedge \CC^{2n+p}$, where $m V$ for an integer $m$ means $V \oplus \cdots \oplus V$ with $V$ occurring $m$ times.

The second part of this paper is dedicated to our novel asymptotic results on the limit shapes of generalized Young diagrams.
We apply our product formulas at $q = 1$ to undertake the asymptotic analysis to compute the limit shapes for the probability measure $\mu_{n,k}(\lambda)$ introduced above in the limit $n,k\to\infty,n/k\to\mathrm{const}$.
This main asymptotic result is formulated as Theorem \ref{thm:limit_shape_gl}, where we derive the limit shapes for all the dual pairs $(G_{1},G_{2})$ mentioned above and prove the convergence of the diagrams to the limit shapes. 
This relies strongly on our product formulas, which are well-suited to this asymptotic analysis and describe the dimension of $V(\overline{\lambda}')$ in terms of $\lambda$ (as opposed to the Weyl dimension formula).
In particular, we use these formulas to produce an integral functional (given by~\eqref{eq:limit-shape-functional}) and solve a variational problem (Lemma~\ref{lemma:limit-shape-function}) in the spirit of the Vershik--Kerov--Logan--Schepp proof.

Since the exterior algebra can be seen as a tensor power, we obtain new results on the asymptotic analysis of the tensor power decomposition.
The asymptotic analysis of the tensor power multiplicities and corresponding probability measure was previously done for a fixed $n$ and $k$ going to infinity in~\cite{nazarov2018limit,tate2004lattice}.
The asymptotics of the probability measure for the tensor power $2k$ of spinor representation of $\SO_{2n+1}$ for both $n, k$ going to infinity was considered in~\cite{NNP20}, where the convergence of generalized Young diagrams to the limit shape was proven.
In the present paper we demonstrate that this result is a consequence of skew Howe duality for $(\SO_{2n+1}, \Pin_{2k})$.
We discover that the limit shapes of Young diagrams for the symplectic and orthogonal groups are ``halves'' of the limit shape of the general linear group.
This can be seen as a reflection of the fact that the branching rule from $\GL_N$ to $\Sp_{2n}$ and $\SO_N$ induces a symmetry in the combinatorics, such as the GT patterns (see, \textit{e.g.},~\cite{Proctor94}).

Our limit shapes for $(\GL_n, \GL_k)$ are already known as well.
This case was computed in~\cite{GTW01} (with their $p = 1/2$) using Toeplitz determinants and saddle point analysis.
It is also a consequence of the limit shapes of Panova and \'Sniady~\cite{panova2018skew} from Equation~\eqref{eq:skew_refinements}, where our limit shape is their limit at $m = \frac{nk}{2}$.
There is also a universality with the fluctuations along the left boundary point with~\cite{Biane01} (\textit{cf.}~\cite{GTW01}), even through the limit shapes are different as those in~\cite{Biane01} are unbounded to the right (with rescaling so the left boundary is at $-1$).
All of the other $(G_1, G_2)$ pairs in Table~\ref{table:results} are new results as far as the authors are aware.

Additionally, we demonstrate that the probability measure for the $(\GL_{n},\GL_{k})$ skew Howe duality is given by the Krawtchouk ensemble (\textit{cf.}~\cite[Sec.~5]{borodin2007asymptotics}; Johansson~\cite{Johansson01} attributes the first appearance of this ensemble to Sepp\"al\"ainen~\cite{Sepp98}).
The Krawtchouk ensemble is a specialization~\cite{BO06} of the $z$-measure~\cite{KOV93}.
Analogously, we show that the skew Howe dualities for the series $\SO_{2n+1},\Sp_{2n},\SO_{2n}$ is a specialization of the $BC$ $z$-measure recently introduced by Cuenca~\cite{Cuenca18} up to a sign and renormalization.
We also show that $\mu_{n,k}$ equals the spectral measure~\cite{BO05gamma,BK10,OO12} for a particular extremal weight.
We discuss these relationships more precisely in Section~\ref{sec:limit_shapes_poly_ensembles}. 

This paper is organized as follows.
In Section~\ref{sec:skew_howe}, we recall basic facts on skew Howe duality.
In Section~\ref{sec:background}, we provide a general background to the combinatorial methods that are employed in this paper.
In Section~\ref{sec:combinatorial_duality}, we derive the multiplicity formulas, prove combinatorial skew Howe duality, and establish the connection to lozenge tilings.
In Section~\ref{sec:limit-shapes-young}, we derive the limit shapes and prove the convergence of the diagrams to the limit shape. We discuss the relation of the limit shapes to the insertion algorithms.
In Section \ref{sec:open_problems}, we list some open problems.

\subsection*{Acknowledgements}

The authors thank Pavel Etingof and Nicolai Reshetikhin for useful conversions.
The authors thank Grigory Olshanski and Evgeny Feigin for pointing out the relation of this work to the skew Howe duality.
The authors thank Pavel Nikitin for pointing out the connection to Krawtchouk ensemble. 
The authors thank Alexei Borodin for the references~\cite{BO06,Cuenca18,OO12} and useful discussions on the $z$-measure, the spectral measure, and the relationship to the Krawtchouk and Meixner ensembles.
The authors thank Daniel Bump for describing how the decomposition of the spinor representation of $\SO_{2n+1}$ depends on the parity of $k$.
The authors thank Christian Krattenthaler for simplifications of our proofs of the determinant-to-product results and noting the $q$-determinant product formulas in~\cite{BKW16} in terms of the conjugate shapes.
The authors thank the anonymous referee for many useful comments.
This work benefited from computations using \textsc{SageMath}~\cite{sage,combinat}.

The work of A.~Nazarov and O.~Postnova is supported by the Russian Science Foundation under grant No.~21-11-00141 and T.~Scrimshaw was partially supported by Grant-in-Aid for JSPS Fellows 21F51028.
This work was partly supported by Osaka City University Advanced Mathematical Institute (MEXT Joint Usage/Research Center on Mathematics and Theoretical Physics JPMXP0619217849).

\section{Classical groups and skew Howe duality}
\label{sec:skew_howe}

\subsection{Clifford algebras and orthogonal groups}
\label{sec:clifford}

To study the action of Lie groups on exterior algebras, we will first recall basic facts about Clifford algebra from~\cite{FH91}. 
The Clifford algebra $C(Q)=\Cliff(V,Q)$ associated to a finite-dimensional, complex, positive-definite inner product space $(V,Q)$ is defined as the quotient of the tensor algebra $T(V) = \bigoplus_{k=0}^{\infty} V^{\otimes k}$ by the two-sided ideal  of $T(V)$ generated by the elements of the form $v\otimes v + 2Q(v,v) \cdot 1_{T(V)}$.
The natural $\ZZ_2$-grading of $T(V)$ into even and odd tensors induces a $\ZZ_2$-grading of the Clifford algebra $C(Q) = C^{even} \oplus C^{odd}$.
The space $V$ is also a subspace of $C(Q)$.
We let $\gl(V) := \End(V)$ denote the Lie algebra of all linear endomorphisms of $V$.

Let $N = \dim V$.
We have the special orthogonal Lie algebra $\so_N(\CC) = \so_N(Q) = C(Q)^{[2]}$, where $C(Q)^{[2]}$ is the (homogeneous) degree $2$ elements of $C(Q)$ from the $\ZZ$-filtration induced from the natural $\ZZ$-grading on the tensor algebra.
Denote by $\Pin_N$ the subgroup of the group of all invertible elements of $C(Q)$ generated by the elements $v\in V$ such that $v^2 = 1$ (equivalently $Q(v,v) = 1$).
The group $\Pin_N$ is a two-fold cover of $\Or_N$, where $\Or_N$ is the orthogonal group of invertible linear maps of $V$ that preserve $Q$.
We will denote by $\Spin_N$ the preimage of $\SO_N$ under natural projection $\Pin_N \to \Or_N$, which is also equal to $\Pin_N \cap C(Q)^{even}$.
Note that the Lie algebra of $\Spin_N$ and $\SO_N$ is isomorphic to $\so_N$, and $V$ is the natural representation of $\so_N$.

Below we will consider when $V$ is even and odd dimensional separately.

\subsubsection{The even dimensional case}

Let $V := \CC^{2n}$, and we write $V = V_+ \oplus V_-$, where $V_+$ has a basis $\{e_1, \dotsc, e_n\}$ and $V_-$ has a basis $\{e_{-n}, \dotsc, e_{-1}\}$.
Furthermore, we choose $V_+$ and $V_-$ to be maximal isotropic subspaces for $Q$.
 
We define $S = \bigwedge V_-$.
The standard basis of $S$ consists of the elements $e_{i_1} \wedge \cdots \wedge e_{i_n}$ with $i_1 < \cdots < i_n$.
There is a unique way, up to isomorphism, to make $S$ into a simple $C(Q)$-module.
The decomposition $V=V_+\oplus V_-$ determines an isomorphism of algebras~\cite{FH91}:
 \[
 C(Q) \iso \End(S).
 \]
Moreover, there is an isomorphism
\[
C(Q)^{even} \iso \End(\bigwedge\nolimits^{even} V_-) \oplus \End(\bigwedge\nolimits^{odd} V_-)
\]
that leads to an embedding of Lie algebras $\so_{2n}(\CC) \subseteq C(Q)^{even} \iso \gl(\bigwedge^{even} V_-) \oplus \gl(\bigwedge^{odd} V_-)$.
Hence, there are two representations of $\so_{2n}$, which we denote by
\[
S^{+} = \bigwedge\nolimits^{even} V_-\;\; \text{and}\;\; S^{-} = \bigwedge\nolimits^{odd} V_-.
\]
These representations are the half-spin representations of $\so_{2n}$ and their highest weights are the fundamental weights $\fw_n$ and $\fw_{n-1}$:
\begin{align*}
\text{for even $n$:} & \;\;S^+=V_{\so_{2n}}(\fw_{n-1})\;\;\text{and}\;\;S^{-} = V_{\so_{2n}}(\fw_{n}),
\\
\text{for odd $n$:} & \;\;S^-=V_{\so_{2n}}(\fw_{n-1})\;\;\text{and} \;\; S^{+} = V_{\so_{2n}}(\fw_{n}).
\end{align*}
Their sum $\bigwedge V_-=S^+\oplus S^-$ is called the spin representation of $\so_{2n}$.
The vector space $S$ when regarded as $\Pin_{2n}$-module is called the \defn{spinor $\Pin_{2n}$-module}.

\subsubsection{The odd dimensional case}

Let $V = \CC^{2n+1}$, which we can decompose as $V = V_+ \oplus V_0 \oplus V_-$, where we take $V_+$ and $V_-$ to be maximal isotropic subspaces as before.
Thus, we have $\dim V_0 = 1$, which can be described as the orthogonal complement, under the inner product defined by $Q$, of $V_+ \oplus V_-$.
There is  a unique up to isomorphism structure of simple $C(Q)^{even}$-module on $S = \bigwedge V_-$.
The decomposition $V=V_+\oplus V_0\oplus V_-$ determines an isomorphism of algebras~\cite{FH91}: 
\[
C(Q) \iso \End\left( \bigwedge V_- \right) \oplus \End\left( \bigwedge V_+ \right).
\]
Moreover, there is an isomorphism
\[
C(Q)^{even} \iso \End\left(\bigwedge\nolimits^{even} V_- \right)
\]
that leads to an embedding of Lie algebras $\so_{2n+1} \subseteq C(Q)^{even} \iso \gl\left(\bigwedge V_- \right) = \gl(S)$.
The representation $S = \bigwedge V_-$ is the irreducible representation of $\so_{2n+1}$ with highest weight $\fw_n$:
\[
S = \bigwedge V_- = V_{\so_{2n+1}}(\fw_{n}).
\]

\subsection{Skew Howe duality}

In the paper~\cite{howe1989remarks}, Roger Howe gives dual pairs of Lie groups and what are now known as Howe correspondences.
We will be interested in the following cases.
Let $n, k$ be nonnegative integers and let $(G_1,G_2)$ be one of the following pairs of classical groups:
\[
(\GL_n, \GL_k),
\qquad
(\Sp_{2n}, \Sp_{2k}),
\qquad
( \Or_{2n}, \SO_{2k}),
\qquad
(\SO_{2n+1}, \Pin_{2k}).
\]
The skew Howe duality for the pairs $(G_1, G_2)$ of classical groups above is given in~\cite{Adamovich96}, where the corresponding $G_1\times G_2$-module is constructed explicitly.
We will follow notations from~\cite{Adamovich96}.

Denote by $V$ the natural $G_1$-module and by $W$ the natural $G_2$-module.  Below we will consider the above mentioned pairs of groups separately.
We will denote by $V_{G_1}(\lambda)$ the simple $G_1$-module and by $V_{G_2}(\overline{\lambda}')$ the simple $G_2$-module. 

We begin by considering the $(\GL_{n}, \GL_{k})$-case.
For $G_1 = \GL_{n}$ the natural module is  $V=\CC^{n}$.
Similarly, for $G_2 = \GL_{k}$ the natural module is $W = \CC^{k}$. 

Firstly, recall that skew Howe duality in $(\GL_{n}, \GL_{k})$-case
\begin{equation}
  \label{eq:gl_gl_skew_Howe}
  \bigwedge\left(\CC^{n}\otimes (\CC^{k})^* \right) = \bigwedge (V \otimes W^*) \iso \bigoplus_{\lambda} V_{GL_{n}}(\lambda)\otimes V_{GL_{k}}(\overline{\lambda}'),
\end{equation}
where $V_{\GL_{n}}(\lambda)$ and $V_{\GL_{k}}(\overline{\lambda}')$ are irreducible modules of $\GL_{n}$ and $\GL_{k}$ correspondingly and $\overline{\lambda}'$ is the conjugate of the complement diagram of $\lambda$ in the $n \times k$ rectangle.
The skew Howe duality decomposition~\eqref{eq:gl_gl_skew_Howe} could also be viewed as the decomposition of a $\GL_{n}$-module into irreducible submodules
\begin{equation}
\label{eq:skew_Howe_A}
\left(\bigwedge V\right)^{\otimes k} \iso \bigoplus_{\lambda} \dim\bigl( V_{\GL_{k}}(\overline{\lambda}') \bigr) V_{\GL_{n}}(\lambda).
\end{equation}
Thus, the dimension of $\GL_n$-module that corresponds to the complement diagram $\overline{\lambda}'$ can be seen as a tensor product decomposition multiplicity
\[
M^k(\lambda) = \dim  V_{\GL_k}(\overline{\lambda}').
\]

Consider $(\Sp_{2n}, \Sp_{2k})$ case.
We have
\[
V=\CC^{2k} = V_+ \oplus V_-,
\qquad\qquad
W=\CC^{2n} = W_+ \oplus W_-,
\]
such that $\dim V_{\pm} = n$ and $V_{\pm}$ are isotropic with respect to the preserved skew-symmetric bilinear form, and similarly for $W$.
The skew Howe duality implies multiplicity free decomposition 
\begin{equation}
\label{eq:spsp}
\bigwedge(\CC^{2n} \otimes (\CC^{k})^*) = \bigwedge(W \otimes (\CC^{k})^*) \iso \bigoplus_{\lambda} V_{\Sp_{2n}}(\lambda)\otimes V_{\Sp_{2k}}(\overline{\lambda}')
\end{equation}
and could be viewed as the decomposition of a $\Sp_{2n}$-module into irreducible submodules
\[
\left(\bigwedge W\right)^{\otimes k} \iso \bigoplus_{\lambda} \dim\bigl( V_{\Sp_{2k}}(\overline{\lambda}') \bigr) V_{\Sp_{2n}}(\lambda).
\]

For the other two pairs of groups, we can simplify the decomposition by expressing the exterior algebra of standard representation in terms of fundamental representations.
We will use the decomposition and notation given in Section~\ref{sec:clifford}.

We consider the $(\SO_{2n+1}, \Pin_{2k})$ case, which is
\[
V = \CC^{2n+1} = V_+ \oplus V_0 \oplus V_-,
\qquad\qquad
W = \CC^{2k} = W_+\oplus W_-,
\]
The skew Howe duality implies multiplicity free decomposition 
\begin{equation}
\label{eq:so2np1}
\bigwedge(\CC^{2n+1}\otimes \CC^{k}) = \bigwedge(V\otimes \CC^{k}) \iso \bigoplus_{\lambda} V_{\SO_{2n+1}}(\lambda)\otimes V_{\Pin_{2k}}(\overline{\lambda}').
\end{equation}
It could be viewed as the decomposition of a $\SO_{2n+1}$-module into irreducible submodules
\[
\left(\bigwedge V \right)^{\otimes k} \iso \bigoplus_{\lambda} \dim\bigl( V_{\Pin_{2k}}(\overline{\lambda}') \bigr) V_{\SO_{2n+1}}(\lambda).
\]
Let us look closely at the left hand side of this decomposition.
There exists an isomorphism
\[
\bigwedge V \iso \bigwedge V_- \otimes  \bigwedge V_0 \otimes \bigwedge V_+ \iso 2 \left(V_{\SO_{2n+1}}(\Lambda_n)\right)^{\otimes 2}
\]
due to the fact that $\bigwedge V_0$ is two dimensional (recall $\dim V_0 = 1$) and
\[
V_{\SO_{2n+1}}(\Lambda_n) = \bigwedge V_- \iso\bigwedge V_+
\]
is a spinor $\SO_{2n+1}$-module. 
On the other hand, recall that the group $\Pin_{2k}$ is a two-fold cover of the group $\Or_{2k}$.
Due to~\cite[Thm.~4.9]{King92}, if $\lambda$ has exactly $k$ rows then the $\Or_{2k}$-module is decomposable on restriction to $\SO_{2k}$ into the direct sum of two inequivalent irreducible $\SO_{2k}$-modules, the dimension of each being half that of  original $\Or_{2k}$-module:
\[
\dim\bigl( V_{\Pin_{2k}}(\overline{\lambda}') \bigr) =\dim\bigl( V_{O_{2k}}(\overline{\lambda}') \bigr)=2\dim\bigl( V_{\SO_{2k}}(\overline{\lambda}') \bigr).
\]
Therefore, this skew Howe duality implies a decomposition of an $\SO_{2n+1}$-module into irreducible submodules
\begin{equation}
\label{spinorrr}
2^k \left(V_{\SO_{2n+1}}(\Lambda_n)\right)^{\otimes 2k} \iso \bigoplus_{\lambda} 2 \dim\bigl( V_{\SO_{2k}}(\overline{\lambda}') \bigr) V_{\SO_{2n+1}}(\lambda).
\end{equation}


Finally, consider the $(\SO_{2n},O_{2k})$ case, where
\[
V = \CC^{2k} = V_+ \oplus V_-,
\qquad\qquad
W = \CC^{2n} = W_+ \oplus W_-,
\]
The skew Howe duality~\cite{Adamovich96} implies multiplicity free decomposition in $(\SO_{2n},\Pin_{2k})$ case:
\begin{equation}
\label{eq:so2nso2l}
\bigwedge(\CC^{2n} \otimes \CC^{k})
 =
   \bigwedge(W\otimes \CC^{k}) \iso \bigoplus_{\lambda}
   V_{SO_{2n}}(\lambda)\otimes V_{O_{2k}}(\overline{\lambda}'),
 \end{equation}
 It could be viewed as the decomposition of a $\SO_{2n}$-module into irreducible submodules
\[
\left(\bigwedge W\right)^{\otimes k} \iso \bigoplus_{\lambda} 2\dim(V_{\SO_{2k}}(\overline{\lambda}'))V_{\SO_{2n}}(\lambda).
\]
The exterior algebra of the standard representation of $\SO_{2n}$ decomposes as
\[
\bigwedge W \iso \bigwedge W_-\otimes  \bigwedge W_+ \iso \bigwedge W_-^{\otimes 2}.
\]
The spin module $\bigwedge W_-$ decomposes into even and odd parts:
\[
\bigwedge W_- \iso \bigwedge\nolimits^{even} W_- \oplus \bigwedge\nolimits^{odd}W_- = V_{\SO_{2n}}(\Lambda_{n-1}) \oplus V_{\SO_{2n}}(\Lambda_{n}).
\]
Therefore, the skew Howe duality implies a decomposition of a sum of half spin $\SO_{2n}$ fundamental modules into irreducible submodules
\[
 \bigl( V_{\SO_{2n}}(\Lambda_{n-1}) \oplus V_{\SO_{2n}}(\Lambda_{n}) \bigr)^{\otimes 2k} \iso \bigoplus_{\lambda} \dim\bigl( V_{\SO_{2k}}(\overline{\lambda}') \bigr) V_{\SO_{2n}}(\lambda).
\]


\section{Combinatorics}
\label{sec:background}

We give the necessary background on partitions, tableaux, highest weight representations, crystals, and the Lindstr\"om--Gessel--Viennot (LGV) lemma.
Fix a positive integer $n$.
Denote $[n] := \{1, 2, \dotsc, n\}$.
Let $\g$ denote a finite-dimensional simple Lie algebra of classical type (\textit{i.e.}, Cartan type ABCD) with indexing set $I$, simple roots $\{\alpha_i\}_{i \in i}$, fundamental weights $\{\fw_i\}_{i \in i}$, weight lattice $P$, simple coroots $\{\coroot_i\}_{i \in I}$, and inner product $\inner{\alpha_i}{\coroot_j} = C_{ij}$ with $[C_{ij}]_{i,j \in I}$ the Cartan matrix.
Let $\{\epsilon_i\}_{i=1}^n$ denote the standard basis of $(\frac{1}{2}\ZZ)^n$ with the standard embedding of $P$.

Note that the finite dimensional highest weight representation of the Lie groups is equivalent to that of its Lie algebra $\g$, which are more natural to discuss for Kashiwara crystals via its (Drinfel'd--Jimbo) quantum group.
Thus, we will freely use the Lie algebra in place of the Lie group when discussing the representation theory.

A \defn{partition} $\lambda$ is a weakly decreasing finite sequence of positive integers, and we draw the Young diagram of $\lambda$ using English convention.
We use the standard identification of partitions with elements in the dominant weight lattice $P^+$.
We denote
\[
\abs{\lambda} = \sum_{i=1}^{\ell} \lambda_i,
\qquad\qquad
\Abs{\lambda} = \sum_{i=1}^{\ell} (i-1) \lambda_i.
\]
the size and weighted size, respectively.

The $q$-analogs of numbers, factorials, and binomials are the standard
\[
[k]_q = 1 + q + \cdots + q^{k-1},
\qquad\qquad
[k]_q! = \prod_{m=1}^k [m]_q,
\qquad\qquad
\qbinom{k}{m}{q} = \frac{[k]_q!}{[m]_q! [k-m]_q!}.
\]
Following~\cite{OS19II}, we will also require the natural (Mahonian) $q$-analog of the triangle Catalan number given by
\[
\mathcal{C}_{n,k}(q) = \frac{[n+k]_q! [n-k+1]_q}{[k]_q! [n+1]_q!} = \frac{[n-k+1]_q}{[n+1]_q} \qbinom{n+k}{k}{q}
\]
for all $n \geq 0$ and $0 \leq k \leq n$.
We consider $\mathcal{C}_{n,k}(q) = 0$ if $n < 0$, $k < 0$, or $k > n$.

\subsection{Crystals}

An \defn{crystal} is a set $\mcB$ with \defn{crystal operators} $\te_i, \tf_i \colon \mcB \to \mcB \sqcup \{ \zero \}$, for $i \in I$, such that for the functions
\[
\varepsilon_i(b) := \max \{k \mid \te_i^k b \neq \zero\},
\qquad\qquad
\varphi_i(b) := \max \{k \mid \tf_i^k b \neq \zero\},
\qquad\qquad
\wt \colon \mcB \to P,
\]
the relations
\[
\te_i b = b' \quad \Longleftrightarrow \quad b = \tf_i b',
\qquad\qquad
\inner{\wt(b)}{\alpha_i} + \varepsilon_i(b) = \varphi_i(b)
\]
hold for all $i \in I$ and $b,b' \in \mcB$ and forms the crystal basis as defined by Kashiwara~\cite{K90,K91} of a  Drinfel'd--Jimbo quantum group $U_q(\g)$-module.
Our definition is what is called a regular or seminormal crystal in the literature (see, \textit{e.g.},~\cite{BS17} for additional information on crystals).
For the crystals considered here, we can encode them as edge $I$-colored (weighted) directed graphs, where $b \xrightarrow{i} b'$ means $\tf_i b = b'$.
We call an element $b \in \mcB$ \defn{highest weight} if $\te_i b = \zero$ for all $i \in I$.
For any $\lambda \in P^+$, there exists a unique crystal $B(\lambda)$ with a unique highest weight element $u_{\lambda}$ of weight $\lambda$ corresponding to the highest weight irreducible representation $V(\lambda)$~\cite{K90,K91}.

We can construct the tensor product of crystals $\mcB_1, \dotsc, \mcB_L$ as follows.
Let $\mcB = \mcB_L \otimes \cdots \otimes \mcB_1$ be the set $\mcB_L \times \cdots \times \mcB_1$.
We define the crystal operators using the \defn{signature rule}.
Let $b = b_L \otimes \cdots \otimes b_2 \otimes b_1 \in \mcB$, and for $i \in I$, we write
\[
\underbrace{\cm\cdots\cm}_{\varphi_i(b_L)}\
\underbrace{\cp\cdots\cp}_{\varepsilon_i(b_L)}\
\cdots\
\underbrace{\cm\cdots\cm}_{\varphi_i(b_1)}\
\underbrace{\cp\cdots\cp}_{\varepsilon_i(b_1)}\ .
\]
Then by successively deleting any $(\cp\cm)$-pairs (in that order) in the above sequence, we obtain a sequence
\[
\sig_i(b) :=
\underbrace{\cm\cdots\cm}_{\varphi_i(b)}\
\underbrace{\cp\cdots\cp}_{\varepsilon_i(b)}
\]
called the \defn{reduced signature}.
Suppose $1 \leq j_{\cm}, j_{\cp} \leq L$ are such that $b_{j_{\cm}}$ contributes the rightmost $\cm$ in $\sig_i(b)$ and $b_{j_{\cp}}$ contributes the leftmost $\cp$ in $\sig_i(b)$.
Then, we have
\begin{align*}
\te_i b &= b_L \otimes \cdots \otimes b_{j_{\cp}+1} \otimes \te_ib_{j_{\cp}} \otimes b_{j_{\cp}-1} \otimes \cdots \otimes b_1, \\
\tf_i b &= b_L \otimes \cdots \otimes b_{j_{\cm}+1} \otimes \tf_ib_{j_{\cm}} \otimes b_{j_{\cm}-1} \otimes \cdots \otimes b_1.
\end{align*}
If one of the factors in a tensor product is $\zero$, then we consider the entire element to be $\zero$.
For type A, the highest weight condition is the classical Yamanouchi condition (see, \textit{e.g.},~\cite{ECII}).

\begin{remark}
Our tensor product convention follows~\cite{BS17}, which is opposite of the tensor product rule used by Kashiwara~\cite{K90,K91}.
\end{remark}

For two crystals $\mcB_1$ and $\mcB_2$, a \defn{crystal morphism} $\psi \colon \mcB_1 \to \mcB_2$ is a map $\mcB_1 \sqcup \{\zero\} \to \mcB_2 \sqcup \{\zero\}$ with $\psi(\zero) = \zero$ such that the following properties hold for all $b \in \mcB_1$ and $i \in I$:
\begin{itemize}
\item[(1)] If $\psi(b) \in \mcB_2$, then $\wt\bigl(\psi(b)\bigr) = \wt(b)$, $\varepsilon_i\bigl(\psi(b)\bigr) = \varepsilon_i(b)$, and $\varphi_i\bigl(\psi(b)\bigr) = \varphi_i(b)$.
\item[(2)] We have $\psi(\te_i b) = \te_i \psi(b)$ if $\psi(\te_i b) \neq \zero$ and $\te_i \psi(b) \neq \zero$.
\item[(3)] We have $\psi(\tf_i b) = \tf_i \psi(b)$ if $\psi(\tf_i b) \neq \zero$ and $\tf_i \psi(b) \neq \zero$.
\end{itemize}
An \defn{embedding} (resp.~\defn{isomorphism}) is a crystal morphism such that the induced map $\mcB_1 \sqcup \{\zero\} \to \mcB_2 \sqcup \{\zero\}$ is an embedding (resp.~bijection).

Next, we consider types B and D.
Here we recall a specific realization of the crystals for the spinor representations due to Kashiwara and Nakashima~\cite{KN94} that is called the \defn{spinor crystal}.
This is $B(\fw_n)$ in type $B_n$ and $B(\fw_{n-1})$ or $B(\fw_n)$ in type $D_n$, which has an underlying set $\{+, -\}^n$ with the additional condition in type $D_n$ that for $(s_1, \dotsc, s_n) \in B(\fw_k)$ we require $\prod_{i=1}^n s_i = -,+$ if $k = n-1,n$ respectively.
The crystal operators are defined by
\begin{align*}
\te_i(s_1, \dotsc, s_n) & = \begin{cases}
(\dotsc, s_{i-1}, +, -, s_{i+2}, \dotsc) & \text{if $i < n$ and } (s_i, s_{i+1}) = (-, +), \\
(\dotsc, s_{n-1}, +) & \text{if $i = n$, type $B_n$ and } s_n = -, \\
(\dotsc, s_{n-2}, +, +) & \text{if $i = n$, type $D_n$ and } (s_{n-1}, s_n) = (-, -), \\
0 & \text{otherwise},
\end{cases}
\\ \tf_i(s_1, \dotsc, s_n) & = \begin{cases}
(\dotsc, s_{i-1}, -, +, s_{i+2}, \dotsc) & \text{if $i < n$ and } (s_i, s_{i+1}) = (+, -), \\
(\dotsc, s_{n-1}, -) & \text{if $i = n$, type $B_n$ and } s_n = +, \\
(\dotsc, s_{n-2}, -, -) & \text{if $i = n$, type $D_n$ and } (s_{n-1}, s_n) = (+, +), \\
0 & \text{otherwise}.
\end{cases}
\\ \wt(s_1, \dotsc, s_n) & = \frac{1}{2} \left( s_1 \epsilon_1 + s_2 \epsilon_2 + \cdots + s_n \epsilon_n \right),
\end{align*}
We remark that these are distinct from the (reduced) signature described above.
An element $(s_1, \dotsc, s_n)$ will be written as tableaux whose shape is a half-width column of height $n$.
For $B(\fw_{n-1})$ in type $D_n$ we consider the box at height $n$ as being a negative half-width box.
This is consistent with the identification of $P^+$ with partitions, and following English convention for tableaux, the entry in the $i$-th row counted from the top in the tableau is $s_i$.

Some examples of the basic crystals of this paper can be found in Figure~\ref{fig:vec_repr} and Figure~\ref{fig:spinor_repr}.

\begin{figure}
\[
\begin{array}{rl}\toprule
A_n: &
\begin{tikzpicture}[xscale=1.9,baseline=-4]
\node (1) at (0,0) {$\ytableaushort{1}$};
\node (2) at (1.5,0) {$\ytableaushort{2}$};
\node (d) at (3.0,0) {$\ytableaushort{3}$};
\node (n-1) at (4.5,0) {$\cdots$};
\node (n) at (6,0) {$\ytableaushort{n}$};
\draw[->,darkred] (1) to node[above]{\tiny$1$} (2);
\draw[->,UQpurple] (2) to node[above]{\tiny$2$} (d);
\draw[->,brown] (d) to node[above]{\tiny$3$} (n-1);
\draw[->,dgreencolor] (n-1) to node[above]{\tiny$n-1$} (n);
\end{tikzpicture}\\
C_n: &
\begin{tikzpicture}[xscale=1.3,baseline=-4]
\node (1) at (0,0) {$\ytableaushort{1}$};
\node (d1) at (1.8,0) {$\cdots$};
\node (n) at (3.6,0) {$\ytableaushort{n}$};
\node (bn) at (5.4,0) {$\ytableaushort{{\bn}}$};
\node (d2) at (7.2,0) {$\cdots$};
\node (b1) at (9,0) {$\ytableaushort{{\bon}}$};
\draw[->,darkred] (1) to node[above]{\tiny$1$} (d1);
\draw[->,dgreencolor] (d1) to node[above]{\tiny$n-1$} (n);
\draw[->,blue] (n) to node[above]{\tiny$n$} (bn);
\draw[->,dgreencolor] (bn) to node[above]{\tiny$n-1$} (d2);
\draw[->,darkred] (d2) to node[above]{\tiny$1$} (b1);
\end{tikzpicture}
\\\bottomrule
\end{array}
\]
\caption{Crystals of the natural representation $B(\fw_1)$ of types $A_n$ and $C_n$.}
\label{fig:vec_repr}
\end{figure}
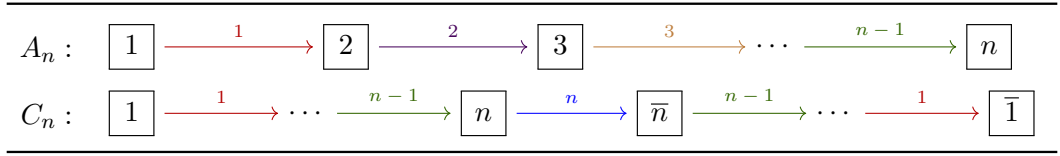

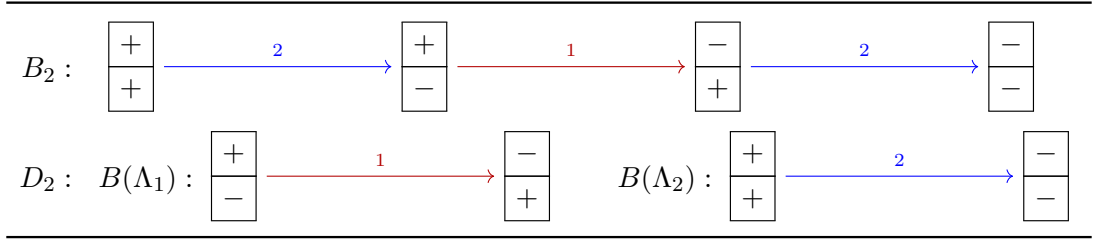
\begin{figure}
\[
\begin{array}{rl}\toprule
B_2: &
\begin{tikzpicture}[xscale=1.3,baseline=-4]
\node (1) at (0,0) {$\ytableaushort{+,+}$};
\node (2) at (3,0) {$\ytableaushort{+,-}$};
\node (3) at (6,0) {$\ytableaushort{-,+}$};
\node (4) at (9,0) {$\ytableaushort{-,-}$};
\draw[->,blue] (1) to node[above]{\tiny$2$} (2);
\draw[->,darkred] (2) to node[above]{\tiny$1$} (3);
\draw[->,blue] (3) to node[above]{\tiny$2$} (4);
\end{tikzpicture}\\
D_2: & B(\fw_1):
\begin{tikzpicture}[xscale=1.3,baseline=-4]
\node (1) at (0,0) {$\ytableaushort{+,-}$};
\node (2) at (3,0) {$\ytableaushort{-,+}$};
\draw[->,darkred] (1) to node[above]{\tiny$1$} (2);
\end{tikzpicture}
\qquad
B(\fw_2):
\begin{tikzpicture}[xscale=1.3,baseline=-4]
\node (1) at (0,0) {$\ytableaushort{+,+}$};
\node (2) at (3,0) {$\ytableaushort{-,-}$};
\draw[->,blue] (1) to node[above]{\tiny$2$} (2);
\end{tikzpicture}
\\\bottomrule
\end{array}
\]
\caption{Crystals of the spinor representations $B(\fw_2)$ for type $B_2$ and $B(\fw_1)$, $B(\fw_2)$ for type $D_2$.}
\label{fig:spinor_repr}
\end{figure}

\subsection{The \texorpdfstring{Lindstr\"om}{Lindstrom}--Gessel--Viennot lemma}

A useful tool for changing combinatorial information into a determinant formula is the \defn{Lindstr\"om--Gessel--Viennot (LGV) Lemma}~\cite{GV85,Lindstrom73}.
Let $\Gamma$ denote an edge-weighted directed graph with weight function $\wt \colon E(\Gamma) \to R$, for some commutative ring $R$.
Let $\uu = (u_1, u_2, \dotsc, u_k)$ and $\vv = (v_1, v_2, \dotsc, v_k)$ be tuples of vertices of $\Gamma$ for some fixed positive integer $k$.
A \defn{family of nonintersecting lattice paths (NILP)} from $\uu$ to $\vv$ is a tuple $(p_1, p_2, \dotsc, p_k)$ of (directed) paths in $\Gamma$, where $p_i$ is a path from $u_i$ to $v_i$ such that no two paths have a common vertex.
Let $N(\uu, \vv)$ denote the set of all NILPs from $\uu$ to $\vv$.
Define the \defn{weight} of a path $p = (\eta_1, \eta_2, \dotsc, \eta_{\ell})$, where $\eta_i \in E(\Gamma)$, and NILP $\pp = (p_1, p_2, \dotsc, p_k)$ to be
\[
\wt(p) = \prod_{i=1}^{\ell} \wt(\eta_i),
\qquad\qquad\qquad
\wt(\pp) = \prod_{i=1}^k \wt(p_i).
\]

\begin{lemma}[{LGV lemma~\cite{GV85,Lindstrom73}}]
We have
\[
\det \left[ \sum_{\pp \in N(u_i, v_j)} \wt(\pp) \right]_{i,j=1}^k = \sum_{\pp \in N(\uu, \vv)} \wt(\pp).
\]
\end{lemma}

Two applications of the LGV lemma are used to compute the multiplicity of $V(\lambda)$ inside of $V(\fw_n)^{\otimes 2k}$ in type $B_n$~\cite[Thm.~4.4]{OS19II} and $V^{\otimes k}$, where $V = \bigwedge V(\fw_1)$, in type $C_n$~\cite[Thm.~4.12]{OS19II}.
In both of these constructions, we are working on a square grid with two types of steps
\begin{itemize}
\item $E \colon (i,j) \mapsto (i+1,j)$,
\item $N \colon (i,j) \mapsto (i,j+1)$.
\end{itemize}
The highest weight condition on the corresponding tensor product of crystals is the nonintersecting condition on a family of lattice paths and that the paths stay strictly below the antidiagonal.

\subsection{Tableaux and patterns}
\label{sec:tableaux_patterns}

A \defn{semistandard tableau} of shape $\lambda$ is a filling of the Young diagram of $\lambda$ with positive integers such that rows are weakly increasing and columns are strictly increasing.
It is a classical fact that the set of all semistandard tableaux of shape $\lambda$ with entries in $\{1, \dotsc, n\}$ parameterize a basis for the irreducible highest weight $\gl_n$-representation $V(\lambda)$.
This can be shown by using the branching rule $\gl_n \branch \gl_{n-1}$, which gives rise to Gelfand--Tsetlin (GT) patterns~\cite{GT1950}, which are triangular arrays such that the top row is the partition $\lambda$ and satisfy the local conditions
\[
\begin{array}{ccc}
a && c \\ & b
\end{array}
\qquad\qquad
a \geq b \geq c.
\]
Furthermore, there is a natural crystal structure on semistandard tableaux by reading columns bottom-to-top from left-to-right and applying the signature rule on the reading word realized as a tensor product of $B(\fw_1)$~\cite{KN94}.
The bijection between GT patterns and semistandard tableaux is given by the $i$-th row of the GT pattern is the shape of the tableau restricted to entries at most $i$.

For $\spn_{2n}$, a basis for $V(\lambda)$ is indexed by the set of \defn{King tableaux} of shape $\lambda$~\cite{King76}, a semistandard tableau in the alphabet $\{1 \prec \bon \prec 2 \prec \btw \prec \cdots \prec n \prec \bn\}$ such that smallest entry in the $i$-th row is at least $i$.
When $\lambda$ is a single column (\textit{i.e.}, $\lambda = \fw_h$ for some $h \in I$), then the King tableaux agree with the Kashiwara--Nakashima tableaux~\cite{KN94} (which has a crystal structure from the reading word) by reordering the column, but this does not hold for general shapes.
This was described in terms of branching rules by Proctor~\cite[Thm~4.2]{Proctor94} using a version of GT patterns for $\gl_{2n+1}$ first given by Kirillov~\cite{Kirillov88}.
These GT patterns satisfy the symmetry that when reflected over the middle, we obtain the negative pattern.
So it becomes sufficient to consider only a half GT pattern (and forgetting the middle column forced to be $0$) as given in~\cite[Thm~4.2]{Proctor94}, which we call a \defn{type $C_n$ Proctor pattern}.
We remark that this half pattern description was first given by \v{Z}elobenko~\cite{Zelobenko62}.
Furthermore, we obtain a King tableau from a type $C_n$ Proctor pattern analogous to the $\gl_n$ case.

Next, we look at the analog of GT patterns for $\so_N$, again following~\cite{Proctor94}.
We can index the basis of $V(\lambda)$ by symmetric (in the sense above) $\gl_{N-1}$ patterns except the middle column no longer has to be its own negative.
Hence, we obtain half patterns as before except the rightmost entries now can be positive or negative, but other satisfy the inequalities with respect to their absolute value.
For $N = 2n+1$, that is we are in type $B_n$, these near symmetric GT patterns are in bijection with type $C_n$ Proctor patterns except we can now allow the rightmost entry to be in $\frac{1}{2}\ZZ_{\geq 0}$ by having an entry $a < 0$ going to $-a - \frac{1}{2}$.
We call such a half pattern a \defn{type $B_n$ Proctor pattern}.
These are in bijection with \defn{Sundaram tableaux}~\cite{Sundaram90}, which are King tableaux with an extra symbol $\infty$ that can only appear at most once in any single row.
A half pattern with the sign for $N = 2n$ will be called a \defn{type $D_n$ Proctor pattern}.

Some examples can be found in Figure~\ref{fig:tableau_patterns}, as well as below in the examples of Section~\ref{sec:lozenge}.

\begin{figure}
\[
\begin{array}{rccc}
\toprule
\text{King tableau} &
\ytableaushort{11{\bon},2{\btw}} &
\begin{array}{ccccc}
  3 && 2 \\
  & 3 && 1 \\
  && 3 \\
  &&&  2
\end{array} &
\begin{array}{ccccc@{\;}c@{\;}c@{\;}c@{\;}c@{\;}c}
  3 && 2 && 0 && -2 && -3\\
  & 3 && 1 && -1 && -3\\
  && 3 && 0 && -3\\
  &&&  2 && -2 \\
  &&&& 0
\end{array}
\\
\midrule
\text{Sundaram tableau} &
\ytableaushort{11{\bon},2{\infty}} &
\begin{array}{ccccc}
  3 && 2 \\
  & 3 && \frac{1}{2} \\
  && 3 \\
  &&&  2
\end{array} &
\begin{array}{cccc@{\;}c@{\;}c@{\;}c@{\;}c}
  3 && 2 && -2 && -3\\
  & 3 && -1 && -3\\
  && 3 && -3\\
  &&&  2
\end{array}
\\\bottomrule
\end{array}
\]
\caption{A King (resp.\ Sundaram) tableau, the corresponding type $C_2$ (resp.\ $B_2$) Proctor pattern, and the (near) symmetric GT pattern.}
\label{fig:tableau_patterns}
\end{figure}

\section{Combinatorial skew Howe duality}
\label{sec:combinatorial_duality}

In this section, we will prove a combinatorial version of skew Howe duality.
Recall that this means that we show that the multiplicity of the representation $V(\lambda)$ inside of $V^{\otimes k}$ for some $G_1$-representation $V$ equals the dimension of another representation $V(\mu)$ for some other classical Lie group $G_2$.
Our proofs use combinatorial identities involving crystal bases and NILPs, which we can then express as a determinant.
Therefore, we express our results in terms of the corresponding Lie algebras.
We can then describe this duality in terms of lozenge tilings, where we are taking paths along two different directions.
Additionally, we give a $q$-deformation of the combinatorial skew Howe duality in a number of cases, where we relate a natural $q$-deformation of our formula with the $q$-dimension of $V(\mu)$.

The $q$-dimension of a highest weight irreducible $\g$-representation $V(\lambda)$ is given by
\[
\dim_q V(\lambda) = \dim_q(\lambda) := \prod_{\alpha \in \Phi^+} \frac{1 - q^{\inner{\lambda+\rho}{\alpha^{\vee}}}}{1 - q^{\inner{\rho}{\alpha^{\vee}}}},
\]
where $\Phi^+$ denotes the set of positive roots of $\g$ and $\rho = \sum_{i \in I} \fw_i$ is the Weyl vector.
We can also compute it using the principal grading (see, \textit{e.g.},~\cite[\S10.10]{kac90}) on $\g$ by
\[
\dim_q(\lambda) = \sum_{\uu \in \ZZ_{\geq 0}^{\abs{I}}} q^{\sum_{i \in I} \uu_i} \dim V(\lambda)_{\lambda-\sum_{i\in I} \uu_i \alpha_i}.
\]

\subsection{Multiplicity in type A}
\label{sec:mult_type_A}

We begin with $\g = \gl_n$ with taking the exterior algebra of the natural representation $V = \bigwedge V(\fw_1)$ and compute the multiplicity of $V(\lambda)$ inside of $V^{\otimes k}$.
To obtain the multiplicities for $\fsl_n$, we need to take the projection of the $\ZZ^{n+1}$ ambient space along the vector $(1,1,\dotsc,1)$.

\begin{prop}
\label{prop:mult_det_A}
Let $\g = \gl_n$, and let $V = \bigwedge V(\fw_1)$.
Then the multiplicity of $V(\lambda)$ in $V^{\otimes k}$ is
\begin{equation}
\label{eq:mult_det_A}
\det \left[ \binom{k+i}{k+i-j-\lambda_{n-j}} \right]_{i,j=0}^{n-1} = \det \left[ \binom{k+i}{j+\lambda_{n-j}} \right]_{i,j=0}^{n-1}.
\end{equation}
\end{prop}

\begin{proof}
We show that the multiplicity is equal to the number of NILPs on a square grid with the initial points $s_i = (0, -i)$ and the terminal points at $t_j = (j+\lambda_{n-j}, k - j - \lambda_{n-j})$.
We build a bijection as follows.
Consider the $(m+i)$-th step on the path $p_i$.
If the step is a horizontal step $E$, then there is an $n-i$ appearing in the $m$-th tensor factor from the right.
Thus, a vertical step $N$ does not contribute anything to the $m$-th factor.
It is straightforward to see that the nonintersecting condition corresponds to the highest weight condition.
Hence, the image is a highest weight element, and the inverse map is clear. 
\end{proof}

\begin{ex}
\label{ex:LGV_mult_type_A_bij}
Consider $\gl_5$, $k = 6$, and $\lambda = (5,4,4,2,1)$.
One such lattice path collection and the corresponding highest weight element in $V^{\otimes 6}$ is
\[
\begin{tikzpicture}[baseline=10, scale=0.5]
\draw[gray!40, very thin] (0,-4) grid (10,6);
\draw[very thick, blue, line join=round] (0,0) -- (0,4) -- (1,4) -- (1,5);
\draw[very thick, blue, line join=round] (1,-1) -- (1,3) -- (3,3);
\draw[very thick, blue, line join=round] (2,-2) -- (2,-1) -- (4,-1) -- (4,0) -- (6,0);
\draw[very thick, blue, line join=round] (3,-3) -- (6,-3) -- (6,-2) -- (7,-2) -- (7,-1);
\draw[very thick, blue, line join=round] (4,-4) -- (7,-4) -- (7,-3) -- (9,-3);
\foreach \i in {1,2,3,4}
  \draw[very thick, red] (0, -\i) -- +(\i, 0);
\draw[thick,dashed] (0,0) -- +(4, -4);
\draw[thick,dashed] (0,6) -- +(10, -10);
\foreach \i in {0,1,2,3,4} {
  \draw[fill=red, color=red] (0, -\i) circle (0.12);
  \draw (0, -\i) node[anchor=east] {$s_{\i}$};
}
\foreach \x/\i in {0/1,1/2,2/4,3/4,4/5} {
  \draw[fill=purple, color=purple] (\i+\x, 6-\i-\x) circle (0.12);
  \draw (\i+\x, 6-\i-\x) node[anchor=south west] {$t_{\x}$};
}
\foreach \i in {1,2,3,4}
  \draw[fill=purple, color=purple] (0+\i, 0-\i) circle (0.12);
\end{tikzpicture}
\quad \longmapsto \quad
\ytableaushort{1,3,4} \otimes \ytableaushort{1,2,3,4,5} \otimes \emptyset \otimes \ytableaushort{1,2,3} \otimes \ytableaushort{1,2,3} \otimes \ytableaushort{1,2}.
\]
Indeed, the entry $1$ appears in all factors except the fourth, so the path $s_4 \to t_4$ has its only vertical step as its fourth step in the region contained between the dashed diagonal lines.
\end{ex}

\begin{cor}
\label{cor:mult_flagged_A}
Let $\overline{\lambda}$ denote the complement of $\lambda$ inside of an $n \times k$ rectangle.
The multiplicity of $V(\lambda)$ in $V^{\otimes k}$ is equal to the number of semistandard tableaux of shape $\overline{\lambda}$ flagged by $(f_0, \dotsc, f_{n-1})$, where $f_i = i + 1 + \lambda_{n-i}$.
\end{cor}

\begin{proof}
The claim follows from the standard bijection between NILPs in a square grid (\textit{e.g.}, rotated by $\pi/2$ counterclockwise from~\cite[Thm.~7.16.1]{ECII}) and semistandard tableaux.
\end{proof}

\begin{ex}
Consider the NILP from Example~\ref{ex:LGV_mult_type_A_bij}.
Corollary~\ref{cor:mult_flagged_A} yields the semistandard tableau
\[
\ytableaushort{11112,2222,35,78,8}\,,
\]
which satisfies the flagging $(2,4,7,8,10)$.
\end{ex}

\begin{cor}
\label{cor:symmetry_A}
Let $\g = \gl_n$, and let $V = \bigwedge V(\fw_1)$.
Let $\overline{\lambda}$ denote the complement of $\lambda$ inside of an $n \times k$ rectangle.
The multiplicity of $V(\lambda)$ and $V(\overline{\lambda})$ in $V^{\otimes k}$ are equal.
\end{cor}

\begin{proof}
This follows by interchanging the roles of the horizontal and vertical steps and noting that we get the same set of NILPs if we instead have the starting points be $s_i = (n-i,-n)$ (that is, being along the bottom boundary instead of the left boundary).
\end{proof}

Another determinant formula for the multiplicity was given by Essam and Goodman in~\cite[Eq.~(53)]{EG95} by applying the LGV lemma but instead only considering the portion of the paths that are not fixed.
Indeed, NILPs are precisely the vicious walkers in~\cite{EG95}.
In Example~\ref{ex:LGV_mult_type_A_bij}, this is the portion of the NILP that is between the dashed lines.

We will prove that the natural $q$-analog of our determinant formula gives a product formula of $q$-integers.
We note that the product formula from the determinant is known to experts, and the product as a $q$-dimension can be seen from~\cite[Lemma~3.2]{BKW16} from a straightforward computation.

\begin{thm}
\label{thm:q_mult_dim_A}
Let $\g = \gl_n$, and let $V = \bigwedge V(\fw_1)$.
For a partition $\lambda$ contained in an $n \times k$ rectangle, define
\[
M^A_q(\lambda) = \det \left[ \qbinom{k+i}{j+\lambda_{n-j}}{q} \right]_{i,j=0}^{n-1}.
\]
Let $a_i = \lambda_i + n - i$.
Then we have
\[
M^A_q(\lambda) = q^{\Abs{\overline{\lambda}}} \frac{\displaystyle \prod_{m=0}^{n-1} [k+m]_q! \times \prod_{1 \leq i < j \leq n} [a_i - a_j]_q}{\displaystyle \prod_{i=1}^n [a_i]_q! [k + n - 1 - a_i]_q!}
 = q^{\Abs{\overline{\lambda}}} \dim_q(\overline{\lambda}') = q^{\Abs{\overline{\lambda}}} \dim_q(\lambda') \in \ZZ_{\geq 0}[q],
\]
where $\dim_q(\nu)$ is the $q$-dimension of $V(\nu)$ for $\gl_k$. Moreover, $M_1^A(\lambda)$ is equal to the multiplicity of $V(\lambda)$ in $V^{\otimes k}$.
\end{thm}

\begin{proof}
We thank Christian Krattenthaler for the following simple proof evaluating the determinant in the first equality.
We first substitute $j \mapsto n-j$.
We factor out $[k+i]_q!$ from the $i$-th row and $1 / ([k+n-1-a_j]_q! [a_j]_q!)$ from the $j$-th column for all $i$ and $j$.
What remains to compute is (after reindexing by $i \mapsto i+1$)
\[
\det\left[ \prod_{m=i}^{n-1} [k+m-a_j]_q \right]_{i,j=1}^n = [a_i - a_j]_q
\]
by~\cite[Prop.~1]{Krat99} with noting that the $(i,j)$-th entry is a polynomial in $q^{-a_j}$ of degree $n - i$.

For the second equality, we use the LGV lemma in two different ways to build a bijection $\phi$, which we then will show is weight preserving up to the shift by $\Abs{\overline{\lambda}}$.
We first note that the determinant is equal to the sum over NILPs as in the proof of Proposition~\ref{prop:mult_det_A}, but we can weight the $m$-th vertical edges from the left by $q^m$ starting with $m = 0$.
Indeed, this gives the $q$-binomial coefficient by using the well-known description of
\[
\qbinom{a}{b}{q} = \sum_{\nu \subseteq a^b} q^{\abs{\nu}}.
\]
To construct a tableau corresponding to a term in $\dim_q(\mu)$, we will construct a NILP using the horizontal steps, but instead of vertical steps, we will use diagonal steps.

The NILP for $\dim_q(\mu)$ is constructed by setting initial points $\overline{s}_i = (-i, i)$ and terminal points $\overline{t}_j = (k - j + \mu_j, j - \mu_j)$, where $1 \leq i,j \leq k$.
Note that since $\mu = \overline{\lambda}'$, the points $\overline{t}_j$ and $t_{j'}$ correspond to all points along the diagonal from $(0,k)$ to $(n+k, -n)$.
For every NILP $\pp$ from $(\mathbf{s}, \mathbf{t})$, we build $\overline{\pp} = (\overline{p}_1, \dotsc, \overline{p}_k)$ from $(\overline{\mathbf{s}}, \overline{\mathbf{t}})$ by having a diagonal step in $\overline{\pp}$ for each vertical step in $\pp$ and connecting the result.
Indeed, the $j$-th diagonal step in $\overline{p}_i$ corresponds to the $i$-th vertical step in $p_{j-1}$, where we take the choice to have the end of the diagonal step be the start of the vertical step (so the top points touch).\footnote{The other choice would be to have the start of the diagonal step matching with the ends of the vertical steps. This would make it so the initial step of each path from $\overline{s}_i \to \overline{t}_i$ would be a horizontal step rather than the last step.}
From the position of the starting points and terminal points, we see that this map $\phi$ is a bijection.
This is the case as the NILP $\overline{\pp}$ is just a semistandard tableaux from the usual LGV lemma proof of the Jacobi--Trudi formula, where the $j$-th diagonal step of $\overline{p}_i$ on the $m$-th diagonal $y = -x + m$ corresponds to the $(i,j)$-th entry being $m$ in a tableau of shape $\mu$.

To show that $\phi$ is weight preserving, we note that $u_{\overline{\lambda}} \mapsto u_{\mu}$ under this bijection given above, which maps the weight $\Abs{\overline{\lambda}}$ to weight $0$.
Next, we note that every time we shift a vertical step right by $1$, we move a diagonal step up by $1$.
Therefore, under these shifts, the bijection is weight preserving (up to the shift by $\Abs{\overline{\lambda}}$).
Since every semistandard tableau in $V(\mu)$ can by obtained from $u_{\mu}$ by a sequence of shifts (which simply changes some $i \mapsto i+1$ in the tableau),  we see this bijection is weight preserving.

We can show the determinant equals $\dim_q(\lambda')$ similarly by instead taking the starting points $\overline{s}_i = (n - 1 + i, -i)$ and terminal points $\overline{t}_j = (j - 1, j - 1 + n - \lambda'_j)$.
The rest of the proof is similar with the $j$-th diagonal step in $\overline{p}_i$ corresponds to the $(j+i)$-th horizontal step in $p_{n-j}$.
\end{proof}

\begin{ex}
\label{ex:LGV_Howe_dual_A}
We consider the NILP from Example~\ref{ex:LGV_mult_type_A_bij}.
We have $\mu = \overline{\lambda}' = (5,4,2,2,1,0)$.
Under the bijection $\phi$ used to prove the second equality in Theorem~\ref{thm:q_mult_dim_A}, we have
\[
\iftikz
\begin{tikzpicture}[baseline=10, scale=0.5]
\draw[gray!40, very thin] (-6,-4) grid (10,6);
\draw[very thick, blue!50, line join=round] (0,0) -- (0,4) -- (1,4) -- (1,5);
\draw[very thick, blue!50, line join=round] (1,-1) -- (1,3) -- (3,3);
\draw[very thick, blue!50, line join=round] (2,-2) -- (2,-1) -- (4,-1) -- (4,0) -- (6,0);
\draw[very thick, blue!50, line join=round] (3,-3) -- (6,-3) -- (6,-2) -- (7,-2) -- (7,-1);
\draw[very thick, blue!50, line join=round] (4,-4) -- (7,-4) -- (7,-3) -- (9,-3);
\foreach \i in {1,2,3,4}
  \draw[very thick, red!50] (0, -\i) -- +(\i, 0);
\draw[thick,dashed] (-6,6) -- +(10, -10);
\draw[thick,dashed] (0,6) -- +(10, -10);
\foreach \i in {0,1,2,3,4} {
  \draw[fill=red, color=red!50] (0, -\i) circle (0.12);
  \draw (0, -\i) node[color=black!50,anchor=north east] {$s_{\i}$};
}
\foreach \x/\i in {0/1,1/2,2/4,3/4,4/5} {
  \draw[fill=purple, color=purple!50] (\i+\x, 6-\i-\x) circle (0.12);
  \draw (\i+\x, 6-\i-\x) node[color=black!50,anchor=south west] {$t_{\x}$};
}
\foreach \i in {1,2,3,4}
  \draw[fill=purple, color=purple!50] (0+\i, 0-\i) circle (0.12);
\draw[very thick, UQpurple, line join=round] (-6,6) -- (0,6);
\draw[very thick, UQpurple, line join=round] (-5,5) -- (0,5) -- (1,4) -- (2,4);
\draw[very thick, UQpurple, line join=round] (-4,4) -- (-1,4) -- (1,2) -- (4,2);
\draw[very thick, UQpurple, line join=round] (-3,3) -- (-1,3) -- (1,1) -- (5,1);
\draw[very thick, UQpurple, line join=round] (-2,2) -- (-1,2) -- (1,0) -- (3,0) -- (4,-1) -- (6,-1) -- (7,-2) -- (8,-2);
\draw[very thick, UQpurple, line join=round] (-1,1) -- (2,-2) -- (5,-2) -- (7,-4) -- (10,-4);
\foreach \i in {1,2,...,6} {
  \draw[fill=red, color=dgreencolor] (-7+\i, 7-\i) circle (0.12);
  \draw (-\i, \i) node[anchor=north east] {$\overline{s}_{\i}$};
}
\foreach \i/\x in {1/10,2/8,3/5,4/4,5/2,6/0} {
  \draw[fill=purple, color=dgreencolor] (\x, 6-\x) circle (0.12);
  \draw (\x, 6-\x) node[anchor=south west] {$\overline{t}_{\i}$};
}
\end{tikzpicture}
\fi
\quad
\longleftrightarrow
\quad
\ytableaushort{11144,2246,33,44,6}\,.
\]
\end{ex}

We can describe the bijection $\phi$ used in the proof of Theorem~\ref{thm:q_mult_dim_A} explicitly in terms of the semistandard tableaux.
For a semistandard tableau $T$, let $\psi(T)$, the entry $m$ in cell $(i, j)$ goes to $m+i-j$ in the cell $(j, i)$.
Clearly $\psi^2$ is the identity map (defined on the set of all semistandard tableaux).
It is a straightforward computation to see that the set given by Corollary~\ref{cor:mult_flagged_A} goes to semistandard tableaux of shape $\overline{\lambda}'$ with the largest entry being $k$.
Furthermore, we have that $\psi$ is $\phi$ translated to semistandard tableaux.

We remark that~\cite[Eq.~(55)]{EG95} is the $q = 1$ version of our product formula for the multiplicity in Theorem~\ref{thm:q_mult_dim_A}.
An alternative proof at $q = 1$ was given in~\cite[Thm.~1]{KGV00}, which could also be extended to the general $q$ case by taking the principal specialization.
Additionally, the equality $q^{\Abs{\overline{\lambda}}} \dim_q(\overline{\lambda}') = q^{\Abs{\overline{\lambda}}} \dim_q(\lambda')$ is the $q$-analog of Corollary~\ref{cor:symmetry_A}.

We describe another connection with a more classical enumeration problem attributed to Verner Hoggatt by Fielder and Alford~\cite{FA89}.
The \defn{$n$-Hoggatt triangle} is the array of integers $(H_{km})_{0 \leq m \leq k}$ given by
\[
H_{km} = \frac{b_n(k)}{b_n(m) b_n(k-m)},
\qquad \text{ where } \qquad
b_n(k) = \prod_{j=1}^k \binom{j+n-1}{n}.
\]
As first proven by Qiaochu Yuan (see~\cite[Sec.~3]{Cigler21}), $H_{km}$ equals the number of semistandard Young tableaux with max entry $k$ with the shape of an $n \times m$ rectangle by the hook-content formula.
By Theorem~\ref{thm:q_mult_dim_A} at $q = 1$, we have the following.

\begin{cor}
\label{cor:Hoggatt_triangle}
Let $\delta = \epsilon_1 + \cdots + \epsilon_n \in P$ for $\gl_n$.
The multiplicity of $V(m\delta)$ in $V^{\otimes k}$ equals $H_{k,k-m}$.
\end{cor}

Corollary~\ref{cor:symmetry_A} then yields the symmetry $H_{km} = H_{k,k-m}$~\cite[Eq.~(3)]{Cigler21}.
Furthermore, Theorem~\ref{thm:q_mult_dim_A} gives a natural $q$-analog of the Hoggatt triangles, along with determinant formulas for the entries and a connection with representation theory from another perspective.
That is, define the $q$-analog of the $n$-Hoggatt triangle by
\[
H_{km}(q) := \dim_q\bigl( \overline{(k-m)\delta}' \bigr) = \dim_q(n^m)
\]
for $\gl_k$.
Consequently, from this perspective, Theorem~\ref{thm:q_mult_dim_A} was recently proven independently in the case $\lambda = m\delta$ by Johann Cigler~\cite[Thm.~8]{Cigler21}.
For alternative proof, we can manipulate~\cite[Eq.~(3.13)]{BKW16} for $\lambda = m \delta$ to obtain the $q$-analog of Hoggatt's triangle given by~\cite[Eq.~(22)]{Cigler21}.

Next we consider the projection to $\fsl_n$, where all of the weights $m\delta \mapsto 0$.
Here, the multiplicity of $V(0)$ equals the \defn{$n$-Hoggatt sums}: the rows sums of the $n$-Hoggatt triangle.
Alternatively, these are the diagonals of the generalized Catalan number triangle as described in \oeis{A116925}~\cite{OEIS}.
In particular, the case of $n=3$ yields a correspondence with Baxter permutations.
The multiplicities for $\fsl_n$ can also be described by the (generalized) hypergeometric series evaluation $_nF_{n-1}(-n+1-k, \dotsc, -k; 2, \dotsc, n; (-1)^n)$.
For the $q$-analog, it is equal to
\[
_n\phi_{n-1}\left[ \begin{array}{l@{\;\;}l@{\;\;}l@{\;\;}l@{}} q^{-n+1-k} & \cdots & q^{-k-1} & q^{-k} \\ q^2 & \cdots & q^n \end{array}; q, (-1)^n \right]
= {}_2\phi_1 \left[ \begin{array}{l@{\;\;}l@{}} q^n & q^{-(n+k-1)} \\ q \end{array}; q^n, (-1)^n \right].
\]
We thank Ole Warnaar for the simplification to using $_2\phi_1$.

\subsection{Multiplicity in types BC}
\label{sec:mult_type_BC}

In this section, we consider the power of the spinor representation $V(\fw_n)^{\otimes K}$ for $\SO_{2n+1}$.
We give a determinant formula and closed product formula for the $q$-analog of the multiplicity of $V(\lambda)$ similar to the case for $\GL_n$ in the previous section.
When $K = 2k+1$, this also equals the multiplicity of $V(\lambda)$ inside of $V^{\otimes k}$ for $\Sp_{2n}$, where $V = \bigwedge V(\fw_1)$.
We recover the formula from~\cite[Eq.~(15)]{NNP20} and unify~\cite[Thm.~4.4, Thm.~4.12]{OS19II}.

We start by considering the natural $q$-analog of the determinant formula from~\cite[Thm.~4.4]{OS19II} for the multiplicity of $V(\lambda + p \Lambda_n)$ inside of $V(\fw_n)^{\otimes 2k+p}$, where $p = 0,1$, in type $B_n$:
\begin{equation}
\label{eq:q_mult_det_B}
M^{BC}_q(\lambda + p\Lambda_n) := \det \left[ \mathcal{C}_{a(i,j), b(i,j)}(q) \right]_{i,j=1}^n,
\end{equation}
where
\[
a(i,j) = 2n - i - j + k + p + \lambda_j,
\qquad\qquad
b(i,j) = j - i + k - \lambda_j.
\]
We remark that the $p = 1$ is a straightforward extension of the $p = 0$ case.

\begin{thm}
\label{thm:q_mult_dim_B}
Let $\lambda$ be a partition inside an $n \times k$ rectangle.
\begin{itemize}
\item[$p=0$:]
We have
\[
M^{BC}_q(\lambda) \prod_{a=1}^{k-1} (q^a + 1) = q^{\Abs{\overline{\lambda}}} \dim_q(\overline{\lambda}' + \omega_k),
\]
where $\omega_k = \frac{1}{2}(\epsilon_1 + \cdots + \epsilon_k)$ for type $D_k$ and $\dim_q(\overline{\lambda}' + \omega_k)$ is the $q$-dimension of $V(\overline{\lambda}' + \omega_k)$ in type $D_k$.
Furthermore, $M^{BC}_1(\lambda)$ equals the multiplicity of $V(\lambda)$ in $V(\fw_n)^{\otimes 2k}$ for type $B_n$ and $M^{BC}_q(\lambda) \in \ZZ_{\geq 0}[q]$.

\item[$p=1$:]
We have
\[
M^{BC}_q(\lambda + \fw_n) = q^{\Abs{\overline{\lambda}}} \dim_q(\overline{\lambda}') \in \ZZ_{\geq 0}[q],
\]
where $\dim_q(\overline{\lambda}')$ is the $q$-dimension of $V(\overline{\lambda}')$ in type $C_k$.
Furthermore, $M^{BC}_1(\lambda + \Lambda_n)$ equals the multiplicity of $V(\lambda + \Lambda_n)$ in $V(\fw_n)^{\otimes 2k+1}$ for type $B_n$ and $V(\lambda)$ in $V^{\otimes k}$ for $V = \bigwedge V(\fw_1)$ in type~$C_n$.
\end{itemize}
\end{thm}

\begin{proof}
We first prove the $p = 0$ case.
The multiplicity claim was proven in~\cite[Thm.~4.4]{OS19II}.
From~\cite[Thm.~3.8, Thm.~4.4]{OS19II}, we have a bijection between the NILPs and King tableaux.
In~\cite[Thm.~8.1]{Proctor94}, the character of $V(\overline{\lambda}' + \omega_k)$ is given by a pair of a King tableau and a $\pm$-vector of length $k-1$.
When we look at the $q$-dimension, this $\pm$-vector contributes a factor of
\begin{equation}
\label{eq:q_dim_spinor}
\dim_q(\omega_k) = \prod_{a=1}^{k-1} (q^a + 1).
\end{equation}
Hence it remains to show that the bijection between NILPs and King tableaux is weight preserving up to a shift by $q^{\Abs{\overline{\lambda}}}$.
This follows from noting that the $q$-dimension for the King tableau is formula~\cite[Thm.~3.8]{OS19II} with a slight modification to compute the $q$-dimension by the $i$-th diagonal having weight $q^i$ and the shift by $q^{\Abs{\overline{\lambda}}}$.

Now we consider the case when $p = 1$.
As previously mentioned, the multiplicity claim for type $B_n$ is a straightforward extension of~\cite[Thm.~4.4]{OS19II}, and type $C_n$ is proven in~\cite[Thm.~4.12]{OS19II}.
Since the rightmost tensor factor has to be $u_{\fw_n}$, the NILPs for the multiplicities in $V(\fw_n)^{\otimes 2k+1}$ are exactly those used to compute the multiplicities in type $C_n$.
Moreover, we note that these are precisely the same NILPs used to compute the dimension of $V(\overline{\lambda}')$ of type $C_k$ by~\cite[Thm.~3.8]{OS19II}.
This also holds for the natural $q$-deformation and the result follows.
\end{proof}

We note that Equation~\eqref{eq:q_dim_spinor} is almost the $q$-analog of $2^k$ that comes from the factor of $2$ difference in dimension between $\bigwedge V(\fw_1)$ and $V(\fw_n) \otimes V(\fw_n)$.
The extra factor of $2$, which would become the $a = 0$ factor in Equation~\eqref{eq:q_dim_spinor}, comes from the order $2$ symmetry of type $D_n$, which replaces the coefficient of $\epsilon_k \leftrightarrow -\epsilon_k$.
This can also be seen as coming from the fact we are using $SO_{2n}$ rather than $\Or_{2n}$ to describe the crystals.
Hence, this is the $q$-combinatorial version of the Howe duality of $\left( \bigwedge V(\fw_1) \right)^{\otimes k}$ with $\SO_{2k}$.

For the next proof, we use the Dodgson condensation method that is based off the Desnanot--Jacobi identity (see~\cite[Sec.~2.3]{Krat99}) to give proof using induction on $n$.
We note that removing the initial vertex $u_n$ increases $k$ by $1$ and removing the terminal vertex $v_n$ increases $\lambda_i$ by $1$ for all $i$, yet the values $\{a_i\}_{i=1}^{n+1}$ do not change in each of the minors.
The base case of $n = 1$ is trivial.
We leave the details to the reader.

\begin{thm}
\label{thm:q_prod_B_even}
Fix positive integers $k$ and $n$.
Let $\lambda$ be a partition contained inside of an $n \times k$ rectangle.
Let $a_i = \lambda_i + (n-i) + \frac{p+1}{2}$.
Then we have
\[
M^{BC}_q(\lambda + p \fw_n) = q^{\Abs{\overline{\lambda}}} \frac{\displaystyle \prod_{i=1}^n [2k+p+2i-2]_q! \, [2a_i]_q \times \prod_{1 \leq i < j \leq n} [a_i - a_j]_q [a_i + a_j]_q}{\displaystyle  \prod_{i=1}^n \left[k+n-a_i + \frac{p-1}{2} \right]_q! \left[ k+n+a_i + \frac{p-1}{2} \right]_q!}.
\]
\end{thm}


Theorem~\ref{thm:q_prod_B_even} for $q = 1$ was proven using different techniques in~\cite[Thm.~6]{KGV00}, where our $a_i$ is their $e_i + 1$.
Their proof can be modified for the general $q$ case by taking the principal specialization.
The product formula for the $q$-dimension is also given by~\cite[Lemma~3.3, Lemma~3.5]{BKW16}.

\subsection{Multiplicity in type D}
\label{sec:mult_type_D}

Now we consider the case for representations of $\SO_{2n}$.
Let $V = V(\fw_{n-1}) \oplus V(\fw_n)$.
The goal of this section is to compute the multiplicity of $V(\lambda)$ inside of $V^{\otimes K}$ as a determinant and a product formula for the natural $q$-analog.
We continue our approach of giving a determinant formula via crystal bases and the LGV lemma, but a little more care is needed because of the two types of spinors involved.
We relate our formula when $K = 2k + 1$ to a $q$-dimension and conjecture that it is a $q$-dimension up to a simple ratio for $K = 2k$, which is precisely the dimension of a representation when $q = 1$.

For an element $(s_1, \dotsc, s_n)$ in a spinor crystal in type $D_n$, we note that the last sign $s_n$ is uniquely determined by the product of the first $n-1$ signs $s_1 \dotsm s_{n-1}$.
However, since $V$ is the direct sum of both spinors, we can freely choose $s_n$, which uniquely determines which of the two summands the element belongs to.
Therefore, we can use the same identification as in type $B_n$ to identify elements in $V^{\otimes 2k}$ with lattice paths in a square grid.
However, for the highest weight condition, we are not allowed to freely choose the sign for $s_n$ nor is it as simple as keeping the paths below the antidiagonal.
We still want the nonintersecting condition to be the translation of the highest weight condition for all $i < n - 1$, so we can restrict ourselves to the rank $2$ case with $i = n-1, n$.

In this case, we have four elements for $(s_{n-1}, s_n)$
\[
  \ytableausetup{boxsize=1.5em}
\ytableaushort{+,+}\,,
\qquad\qquad
\ytableaushort{+,-}\,,
\qquad\qquad
\ytableaushort{-,+}\,,
\qquad\qquad
\ytableaushort{-,-}\,.
\]
The first and second cases are highest weight elements, which pair with the fourth and third cases respectively.
Hence, we no longer require the path $p_n$ to stay strictly below the antidiagonal, but there is some influence from $p_{n-1}$.
We fix the path $p_{n-1}$ and we then mirror the path from $u_{n-1}$ across the antidiagonal; \textit{i.e.}, we swap $N \leftrightarrow E$ steps.

\begin{lemma}
\label{lemma:rank2_case}
Let $B$ be the corresponding crystal of $V$ in type $D_2$.
All of the highest weight elements in $B^{\otimes 2k}$ of weight $\lambda = c_{n-1} \epsilon_{n-1} + c_n \epsilon_n \in P^+$ are in bijection with NILPs $(p_{n-1}, p_n)$ in the grid with
\begin{itemize}
\item starting vertices $u_n = (0,0)$ and $u_{n-1} = (-1,-1)$,
\item terminal vertices $v_n = (k+c_n, k-c_n)$ and $v_{n-1} = (k+1+c_{n-1}, k-1-c_{n-1})$, and
\item $p_n$ does not intersect $p_{n-1}$ reflected across the antidiagonal.
\end{itemize}
\end{lemma}

\begin{proof}
We note that for any $(s_{n-1}, s_n) \in V$ there is a corresponding element obtained by sending $s_n \leftrightarrow -s_n$, which applied to every factor translates to reflecting the path $p_n$.
Therefore, if there is an intersection, it is sufficient to consider the case when the first intersection point is below the antidiagonal.
The bijection between the paths $(p_{n-1}, p_n)$ and highest weight elements is the same as the type $B_2$ case in~\cite[Thm.~4.4]{OS19II}.
This reflection is also a manifestation of the Dynkin diagram symmetry that sends $n-1 \leftrightarrow n$, so it is only sufficient to consider $e_{n-1}$.
Hence, the highest weight condition corresponds to the nonintersecting condition and is proven similar to~\cite[Thm.~4.4]{OS19II}.
\end{proof}

Next, we want to convert this to an honest NILP, which means we need to remove the symmetric path $p_{n-1}$.
We do this by noting that every time the path touches the antidiagonal, we have two choices.
Therefore, we ``fold'' the path $p_n$ to stay below the antidiagonal but we still need to retain the fact that we have two choices, which we encode by having two edges $N_{\pm}$.
We have $N_+$ correspond to the case when the previous step in the path was below the antidiagonal and $N_-$ when it was above antidiagonal.
We note that reflecting part of a path over the antidiagonal corresponds to interchanging $E \leftrightarrow N$.
Hence, the number of these paths are in bijection with paths from $(0,0)$ to $(x,y)$ on a square grid.

To make this precise in terms of the LGV lemma, let $\mcD$ denote the (infinite) ``grid'' (directed graph) that consists of $E \colon (i, j) \mapsto (i+1, j)$ and $N \colon (i, j) \mapsto (i, j+1)$ steps that do not have an endpoint on the antidiagonal (that is, either endpoint has coordinates $(i,i)$) and two steps $N_{\pm} \colon (i,i-1) \mapsto (i, i)$ that end on the antidiagonal.
As a consequence, the directed paths must lie weakly below the antidiagonal line of $y = x$.

\begin{ex}
We demonstrate Lemma~\ref{lemma:rank2_case} when $\lambda = 0$ with a pair of lattice paths $(p_1, p_2)$ in the symmetric and folded versions and the corresponding highest weight element:
\[
\begin{tikzpicture}[baseline=32, scale=0.5]
\draw[-, very thin, gray] (0,0) grid (5,5);
\draw[-, very thick, UQpurple, line join=round] (0,0) -- (3,0) -- (3,1) -- (5,1) -- (5,3);
\draw[-, very thick, UQpurple, line join=round, dashed] (0,0) -- (0,3) -- (1,3) -- (1,5) -- (3,5);
\draw[-, dashed] (2,0) -- (0,2);
\draw[-, dashed] (3,5) -- (5,3);
\draw[fill=UQpurple, color=UQpurple] (0, 0) circle (0.12);
\draw[fill=UQpurple, color=UQpurple] (5, 3) circle (0.12);
\draw[fill=UQpurple, color=UQpurple] (3, 5) circle (0.12);
\draw[-, very thick, blue, line join=round] (1,1) -- (2,1) -- (2,3) -- (3,3) -- (3,4) -- (4,4);
\draw[fill=blue, color=blue] (1, 1) circle (0.12);
\draw[fill=blue, color=blue] (4, 4) circle (0.12);
\end{tikzpicture}
\quad \longleftrightarrow \quad
\begin{tikzpicture}[baseline=32, scale=0.5]
\draw[-, thick] (0,0) -- (5,5);
\foreach \i in {1,...,5} {
  \draw[-, very thin, gray] (\i,\i) -- (5,\i);
  \draw[-, very thin, gray] (\i,\i-1) to[out=70,in=290] (\i,\i);
  \draw[-, very thin, gray] (\i,\i-1) to[out=110,in=250] (\i,\i);
}
\foreach \i in {1,...,5} {
  \draw[-, very thin, gray] (\i,0) -- (\i,\i-1);
}
\draw[-, very thin, gray] (0,0) -- (5,0);
\draw[-, dashed] (2,0) -- (1,1);
\draw[-, dashed] (4,4) -- (5,3);
\draw[-, very thick, UQpurple, line join=round] (0,0) -- (3,0) -- (3,1) -- (5,1) -- (5,3);
\draw[fill=UQpurple, color=UQpurple] (0, 0) circle (0.12);
\draw[fill=UQpurple, color=UQpurple] (5, 3) circle (0.12);
\draw[-, very thick, blue, line join=round] (1,1) -- (2,1) to[out=70,in=290] (2,2) -- (3,2) to[out=110,in=250] (3,3) -- (4,3) to[out=110,in=250] (4,4);
\draw[fill=blue, color=blue] (1, 1) circle (0.12);
\draw[fill=blue, color=blue] (4, 4) circle (0.12);
\end{tikzpicture}
\quad \longleftrightarrow \quad
\ytableaushort{-,+} \otimes \ytableaushort{-,-} \otimes \ytableaushort{+,+} \otimes \ytableaushort{+,-} \otimes \ytableaushort{-,-} \otimes \ytableaushort{+,+}\,.
\]
\end{ex}

\begin{lemma}
\label{lemma:path_counting_D}
The number of paths from $(0,0)$ to $(x,y)$ on the grid $\mcD$, where necessarily $x \geq y$, is
\[
\binom{x+y}{y}.
\]
\end{lemma}

\begin{proof}
We unfold the path and the underlying graph to be on a square grid.
\end{proof}

Now we can prove a determinant formula and product formula for the multiplicity.

\begin{thm}
\label{thm:mult_type_D}
Let $\g = \so_{2n}$ and let $V = V(\fw_{n-1}) \oplus V(\fw_n)$.
Let $p = 0, 1$.
Define
\[
M^D_q(\lambda + p\Lambda_n) := \det \left[ \qbinom{2(k+i) + p}{k+i-j-\absval{\lambda_{n-j}}}{q} \right]_{i,j=0}^{n-1}.
\]
Then the multiplicity of $V(\lambda + p \fw_{n-1})$ and $V(\lambda + p \fw_n)$ in $V^{\otimes 2k+p}$ is $M^D_1(\lambda + p \Lambda_n)$.
Furthermore, we have
\begin{equation}
M^D_q(\lambda + p\fw_n) = q^{\Abs{\overline{\lambda}}} \frac{\displaystyle \prod_{i=1}^n [2k+2n-2i+p]_q! \times \prod_{1 \leq i < j \leq n} [a_i - a_j]_q [a_i + a_j]_q}{\displaystyle \prod_{i=1}^n \left[ k+n-1-a_i + \frac{p}{2} \right]_q! \left[k+n-1+a_i + \frac{p}{2} \right]_q!} \in \ZZ_{\geq 0}[q] \label{eq:D_q_prod},
\end{equation}
where $a_i = \lambda_i + n - i + \frac{p}{2}$.
\end{thm}

\begin{proof}
The proof of the first claim is similar for the type $B_n$ case from~\cite[Thm.~4.4]{OS19II}.
Indeed, we note that we have the starting vertices $u_i = (-i, -i)$ and the terminal vertices $v_j = (k+j+\absval{\lambda_{n-j}}, k-j-\absval{\lambda_{n-j}})$.
From Lemma~\ref{lemma:path_counting_D} and the LGV lemma, we see that the number of paths in this graph is equal to the determinant~\eqref{thm:mult_type_D} at $q = 1$.
From the definition of the crystal operators and the folded version of Lemma~\ref{lemma:rank2_case}, a NILP corresponds to a highest weight element read along diagonals, where in the $k$-th path, the $k$-th entry is given by $E \mapsto +$ and $N \mapsto -$.

We can show the product formula for $M^D_q(\lambda)$ by using Dodgson condensation similar to the proof of Theorem~\ref{thm:q_prod_B_even}. 
To show $M^D_q(\lambda) \in \ZZ_{\geq 0}[q]$, we apply the LGV lemma.
\end{proof}

\begin{figure}
\[
\iftikz
\begin{tikzpicture}[baseline=10, scale=0.5]
\draw[-, thick] (-4,-4) -- (5,5);
\draw[-, very thin, gray] (-4,-4) -- (11,-4);
\foreach \i in {-3,...,5} {
  \draw[-, very thin, gray] (\i,\i) -- (11,\i);
  \draw[-, very thin, gray] (\i,\i-1) to[out=70,in=290] (\i,\i);
  \draw[-, very thin, gray] (\i,\i-1) to[out=110,in=250] (\i,\i);
}
\foreach \i in {-3,-2,...,5} {
  \draw[-, very thin, gray] (\i,-4) -- (\i,\i-1);
}
\foreach \i in {6,7,...,11}
  \draw[-, very thin, gray] (\i,-4) -- (\i,5);
\draw[-,dashed] (5,5) -- +(6,-6);
\draw[-,dashed] (-1,-1) -- +(3,-3);
\draw[very thick, blue, line join=round] (2,-4) -- (5,-4) -- (5,-3) -- (9,-3) -- (9, -2) -- (11,-2) -- (11,-1);
\draw[very thick, blue, line join=round] (1,-3) -- (3,-3) -- (3,-2) -- (6,-2) -- (6,0) -- (7,0) -- (7,1) -- (9,1);
\draw[very thick, blue, line join=round] (0,-2) -- (2,-2) -- (2,-1) -- (3,-1) -- (3,0) -- (5,0) -- (5,1) -- (6,1) -- (6,2) -- (8,2);
\draw[very thick, blue, line join=round] (-1,-1) -- (0,-1) to[out=110,in=250] (0,0) -- (2,0) -- (2,1) to[out=110,in=250] (2,2) -- (4,2) -- (4,3) to[out=70,in=290] (4,4) -- (5,4) to[out=110,in=250] (5,5);
\foreach \i in {2,3,4}
  \draw[very thick, red] (-\i, -\i) -- (-2+\i, -\i);
\foreach \i in {1,2,3,4} {
  \draw[fill=red, color=red] (-5+\i, -5+\i) circle (0.12);
  \draw (-5+\i, -5+\i) node[anchor=east] {$s_{\i}$};
}
\foreach \x/\i in {1/0,2/2,3/3,4/6} {
  \draw[fill=purple, color=purple] (11-\i, -1+\i) circle (0.12);
  \draw (11-\i, -1+\i) node[anchor=south west] {$t_{\x}$};
}
\foreach \i in {2,3,4}
  \draw[fill=purple, color=purple] (-2+\i, -\i) circle (0.12);
\end{tikzpicture}
\fi
\]
\caption{An example of an NILP on the grid $\mcD$ for type $D_4$ and $k = 6$.}
\label{fig:type_D_ex}
\end{figure}
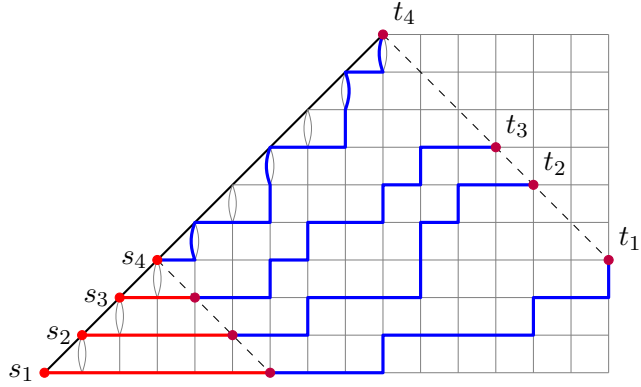

\begin{ex}
Consider the NILP from Figure~\ref{fig:type_D_ex}.
The corresponding highest weight element is
\[
\ytableaushort{-,+,+,+}\otimes
\ytableaushort{+,+,+,-} \otimes
\ytableaushort{+,-,-,-} \otimes
\ytableaushort{-,+,+,-} \otimes
\ytableaushort{+,-,-,+} \otimes
\ytableaushort{+,-,+,+} \otimes
\ytableaushort{+,+,+,+} \otimes
\ytableaushort{+,+,-,+} \otimes
\ytableaushort{-,+,+,-} \otimes
\ytableaushort{+,-,-,-} \otimes
\ytableaushort{+,+,+,+} \otimes
\ytableaushort{+,+,+,-} 
\]
\end{ex}

An alternative formulation of Theorem~\ref{thm:mult_type_D} was given by Grabiner and Magyar~\cite[Eq.~(47)]{GM93}.
For the case of $V^{\otimes 2k+1}$, we can show the result is equal to a $q$-dimension.

\begin{remark}
Christian Krattenthaler has provided an alternative proof of Theorem~\ref{thm:q_prod_B_even} and Theorem~\ref{thm:mult_type_D} by using~\cite[Thm.~27]{Krat99} with $A = -2k - p + 1, L_i = -a_i + \frac{p+1}{2} + k - 1$ and $A = -2k - 1, L_i = -a_i + \frac{p}{2} + k - 1$, respectively, after reversing the rows (and columns for Theorem~\ref{thm:mult_type_D} and an index shift), transposing the matrix, and applying the identity
$
\qbinom{a}{b}{q} = (-1)^b q^{b(2a-b+1)/2} \qbinom{b-a-1}{b}{q}.
$
\end{remark}

\begin{thm}
Fix positive integers $k$ and $n$.
Let $\lambda$ be a partition contained inside of an $n \times k$ rectangle.
Then we have
\[
M^D_q(\lambda + \Lambda_n) = q^{\Abs{\overline{\lambda}}} \dim_q(\overline{\lambda}'),
\]
where $\dim_q(\overline{\lambda}')$ be the $q$-dimension of $V(\overline{\lambda}')$ in type $B_k$.
Moreover, $M^D_1(\lambda + \Lambda_n)$ equals the multiplicity of $V(\lambda + \Lambda_{n'})$ for $n' = n-1, n$ in $V^{\otimes 2k + 1}$.
\end{thm}

\begin{proof}
The first claim follows from the analogous bijection between Sundaram tableaux and NILPs as in the type $B_n$ case with King tableaux with the $\infty$ entries in the Sundaram tableaux being one of the choices $N_{\pm}$ for the vertical steps.
The fact that this equals the multiplicity is analogous to the proof of Theorem~\ref{thm:mult_type_D}.
\end{proof}

When we look at an even power of $V$, we have the following simple ratio relating the two formulas.
This can be shown by comparing the product formula~\eqref{eq:D_q_prod} and~\cite[Lemma~3.6]{BKW16}.

\begin{cor}
\label{cor:q_mult_D_even}
Fix positive integers $k$ and $n$.
Let $\lambda$ be a partition contained inside of an $n \times k$ rectangle.
Then we have
\begin{equation}
M^D_q(\lambda) = q^{\Abs{\overline{\lambda}}} \prod_{i=1}^k \frac{q^{\overline{\lambda}'_i+k-i} + 1}{q^{k-i} + 1} \dim_q(\overline{\lambda}'), \label{eq:D_q_dim}
\end{equation}
where $\dim_q(\overline{\lambda}')$ be the $q$-dimension of $V(\overline{\lambda}')$ in type $D_k$.
\end{cor}

We note that the ratio in Equation~\eqref{eq:D_q_dim} is $2^k / 2^k = 1$ at $q = 1$.
In the sequel, we will provide a combinatorial proof that Corollary~\ref{cor:q_mult_D_even} holds at $q = 1$.

\subsection{Interpretation with lozenge tilings}
\label{sec:lozenge}

We give a unified description for all of the combinatorial skew Howe duality discussed in this section using lozenge tilings.
We realize the combinatorial skew Howe duality by using the fact that lozenge tilings of a half hexagon have three different sets of paths that uniquely specify the tiling.
The lozenge tilings are naturally in bijection with GT patterns, so they can be used to describe irreducible representations of $\gl_n$.
This also accounts for the natural symmetry of Corollary~\ref{cor:symmetry_A}.
The remaining cases then are based on imposing additional symmetries to the lozenge tilings coming from~\cite{Proctor94} on GT patterns.
These can be seen as manifestations of the branching rule from the natural embedding of $\g$ inside of $\gl_N$.
These symmetries on lozenge tilings were discussed in~\cite[Sec.~3.2]{BG15}.

We begin by describing the lozenge tilings on a half hexagon.
Fix a partition $\mu$, and let $a_i = \mu_i + k + 1 - i$.
We will be using the lozenge tiles $R, G, B$
\[
\begin{tikzpicture}
\draw[fill=red!30] (0,0) -- ++(0,1) -- ++(30:1) -- ++(0,-1) -- cycle;
\draw[fill=green!30] (3,0)  -- ++(0,1) -- ++(150:1) -- ++ (0,-1) -- cycle;
\draw[fill=blue!30] (5,0.2) -- ++(30:1) -- ++(150:1) -- ++ (30:-1) -- cycle;
\end{tikzpicture}
\]
respectively.
The region we will be tiling is a half hexagon $\mcH_{\lambda}$ with the top and bottom sides having length $n$, the left side having length $k$, which means the right side will have length $n + k$.
We will place the $B$ tiles along the right boundary at heights $a_i$, where they protrude outside of the region. Sometimes we consider these $B$ tiles to actually be triangles, so the result lies perfectly inside the half hexagon.
Lozenge tilings of $\mcH_{\lambda}$ are in bijection with GT patterns, where for row $\mu^{(j)}$ in a GT pattern $(\mu^{(j)})_{j=1}^k$, we place $B$ tiles at heights $\mu^{(j)}_i + k + 1 - i$ in the $j$-th column from the right.
It is a classical fact that this uniquely determines a lozenge tiling as the $B$ tiles can be seen as the tops of cubes stacked in a corner, known as a plane partition in combinatorics (see, \textit{e.g.},~\cite{ECII}).

Now we can give an explanation of the two NILPs that appear in the proof of Theorem~\ref{thm:q_mult_dim_A}.
From an NILP that contributes to the multiplicity of $V(\lambda)$, we obtain a lozenge tiling by considering $p_i$ to be a path starting from the $i$-th position along left boundary from the bottom with every horizontal (resp.\ vertical) step corresponding to an $R$ (resp.~$G$) tile.
This also uniquely determines the lozenge tiling as the only tiles missing must be $B$ tiles.
Next, we can take paths in this lozenge tiling that avoid the $R$ tiles, and this will correspond to the paths $\overline{p}_i$ with $\overline{s}_i$ being on the left at the $i$-th position from the top with $B$ (resp.~$G$ tiles) translating to horizontal (resp.~diagonal) steps.
This yields the semistandard tableaux that gives the dimension of $V(\overline{\lambda}')$.
Finally, by taking the paths in the lozenge tiling avoiding the $G$ tiles, we have the NILPs for the semistandard tableaux for $V(\lambda')$.

\begin{ex}
\label{ex:half_hexagon}
We consider the semistandard tableau from Example~\ref{ex:LGV_Howe_dual_A}, which recall that $n = 5$ and $k = 6$.
We see that it has the corresponding GT pattern and lozenge tiling of
\[
\begin{array}{ccccccccccc}
5 && 4 && 2 && 2 && 1 && 0 \\
& 5 && 3 && 2 && 2 && 0 \\
&& 5 && 3 && 2 && 2 \\
&&& 3 && 2 && 2 \\
&&&& 3 && 2 \\
&&&&& 3
\end{array}
\qquad \longleftrightarrow \qquad
\iftikz
\begin{tikzpicture}[scale=0.5,baseline=70]
\draw (0,0) -- ++(150:6) -- ++(0,5) -- ++(30:6);
\foreach \y in {0,2,4,5,8,10}
  \draw[fill=blue!30] (0,\y) -- ++(30:1) -- ++(150:1) -- ++ (30:-1) -- cycle;
\foreach \y in {0,3,4,6,9}
  \draw[fill=blue!30] (150:1) ++ (0,\y) -- ++(30:1) -- ++(150:1) -- ++ (30:-1) -- cycle;
\foreach \y in {2,3,5,8}
  \draw[fill=blue!30] (150:2) ++ (0,\y) -- ++(30:1) -- ++(150:1) -- ++ (30:-1) -- cycle;
\foreach \y in {2,3,5}
  \draw[fill=blue!30] (150:3) ++ (0,\y) -- ++(30:1) -- ++(150:1) -- ++ (30:-1) -- cycle;
\foreach \y in {2,4}
  \draw[fill=blue!30] (150:4) ++ (0,\y) -- ++(30:1) -- ++(150:1) -- ++ (30:-1) -- cycle;
\draw[fill=blue!30] (150:5) ++ (0,3) -- ++(30:1) -- ++(150:1) -- ++ (30:-1) -- cycle;
\foreach \y in {1,7}
  \draw[fill=green!30] (0,\y)  -- ++(0,1) -- ++(150:1) -- ++ (0,-1) -- cycle;
\foreach \y in {0,1,4,6,7}
  \draw[fill=green!30] (150:2) ++ (0,\y)  -- ++(0,1) -- ++(150:1) -- ++ (0,-1) -- cycle;
\foreach \x in {3,4,5} {
  \foreach \y in {0,1}
    \draw[fill=green!30] (150:\x) ++ (0,\y)  -- ++(0,1) -- ++(150:1) -- ++ (0,-1) -- cycle;
}
\draw[fill=green!30] (150:5) ++ (0,2)  -- ++(0,1) -- ++(150:1) -- ++ (0,-1) -- cycle;
\foreach \y in {3,6,9}
  \draw[fill=red!30] (0,\y) -- ++(0,1) -- ++ (210:1) -- ++(0,-1) -- cycle;
\foreach \y in {1,2,5,7,8}
  \draw[fill=red!30] (150:1) ++ (0,\y) -- ++(0,1) -- ++ (210:1) -- ++(0,-1) -- cycle;
\foreach \x in {3,4,5} {
  \foreach \y in {9-\x,10-\x}
    \draw[fill=red!30] (150:\x) ++ (0,\y) -- ++(0,1) -- ++ (210:1) -- ++(0,-1) -- cycle;
}
\foreach \y in {3,4}
  \draw[fill=red!30] (150:\y) ++ (0,7-\y) -- ++(0,1) -- ++ (210:1) -- ++(0,-1) -- cycle;
\draw[very thick, color=blue] (150:6) ++ (0,0.5) -- ++(150:-4) -- ++(30:1) -- ++(150:-1);
\draw[very thick, color=blue] (150:6) ++ (0,1.5) -- ++(150:-4) -- ++(30:2);
\draw[very thick, color=blue] (150:6) ++ (0,2.5) -- ++(150:-1) -- ++(30:2) -- ++(150:-1) -- ++(30:2);
\draw[very thick, color=blue] (150:6) ++ (0,3.5) -- ++(30:3) -- ++(150:-1) -- ++(30:1) -- ++(150:-1);
\draw[very thick, color=blue] (150:6) ++ (0,4.5) -- ++(30:3) -- ++(150:-1) -- ++(30:2);
\foreach \i/\y in {0/1,1/3,2/6,3/7,4/9} {
  \draw[fill=purple, color=purple] (0, \y+0.5) circle (0.12);
  \draw[fill=purple, color=purple] (150:6) ++ (0, \i+0.5) circle (0.12);
  \draw (150:6) ++ (0, \i+0.5) node[anchor=east] {$\widetilde{s}_{\i}$};
  \draw (0, \y+0.5) node[anchor=west] {$t_{\i}$};
}
\end{tikzpicture}
\fi
\]
Note that the path from $\widetilde{s}_i \to t_i$ in the lozenge tiling that avoid the $B$ tiles is the path from $s_i \to t_i$ in Example~\ref{ex:LGV_mult_type_A_bij} between the dashed lines.
Furthermore, the paths from $\overline{s}_i \to \overline{t}_i$ in Example~\ref{ex:LGV_Howe_dual_A} correspond to the paths in the lozenge tiling that avoids the $R$ tile.
The NILP that would correspond to a semistandard tableau of shape $\lambda'$ come from taking the paths that avoid the $G$ tiles.
\end{ex}

We can realize the skew Howe duality in terms of lozenge tilings in a different way.
We allow ourselves to take only triangles along the middle of a hexagon, so no $B$ tile can cross the middle, with side lengths alternating between $m$ and $n$.
Therefore, one side corresponds to the multiplicity of $V(\lambda)$ in $V^{\otimes k}$ for $\gl_n$ as before.
The other side we will consider as the representation $V(\lambda')$ for $\gl_k$ from the GT pattern.
Thus, this gives a natural combinatorial description of Equation~\eqref{eq:skew_Howe_A} and the skew Howe duality in~\eqref{eq:gl_gl_skew_Howe}.

\begin{ex}
One tiling of a hexagon for $n = 5$ and $k = 6$ with $\lambda = 44421$ representing the combinatorial skew Howe duality and their corresponding pair of GT patterns is
\[
\iftikz
\begin{tikzpicture}[scale=0.5,baseline=75]
\draw (0,0) -- ++(150:6) -- ++(0,5) -- ++(30:6) -- ++(-30:5) -- ++(0,-6) -- ++(-150:5);
\foreach \y in {0,2,4,5,9,10}
  \draw[fill=black!50] (0,\y) -- ++(150:1) -- ++(30:1) -- cycle;
\foreach \y in {0,3,4,7,9}
  \draw[fill=blue!30] (150:1) ++ (0,\y) -- ++(30:1) -- ++(150:1) -- ++ (30:-1) -- cycle;
\foreach \y in {1,3,5,8}
  \draw[fill=blue!30] (150:2) ++ (0,\y) -- ++(30:1) -- ++(150:1) -- ++ (30:-1) -- cycle;
\foreach \y in {2,4,6}
  \draw[fill=blue!30] (150:3) ++ (0,\y) -- ++(30:1) -- ++(150:1) -- ++ (30:-1) -- cycle;
\foreach \y in {2,5}
  \draw[fill=blue!30] (150:4) ++ (0,\y) -- ++(30:1) -- ++(150:1) -- ++ (30:-1) -- cycle;
\draw[fill=blue!30] (150:5) ++ (0,4) -- ++(30:1) -- ++(150:1) -- ++ (30:-1) -- cycle;
\foreach \y in {1,8}
  \draw[fill=green!30] (0,\y)  -- ++(0,1) -- ++(150:1) -- ++ (0,-1) -- cycle;
\foreach \x in {3,4,5} {
  \foreach \y in {0,1}
    \draw[fill=green!30] (150:\x) ++ (0,\y)  -- ++(0,1) -- ++(150:1) -- ++ (0,-1) -- cycle;
}
\foreach \y in {2,3}
  \draw[fill=green!30] (150:5) ++ (0,\y)  -- ++(0,1) -- ++(150:1) -- ++ (0,-1) -- cycle;
\draw[fill=green!30] (150:2) ++ (0,7)  -- ++(0,1) -- ++(150:1) -- ++ (0,-1) -- cycle;
\draw[fill=green!30] (150:1) ++ (0,6)  -- ++(0,1) -- ++(150:1) -- ++ (0,-1) -- cycle;
\draw[fill=green!30] (150:3) ++ (0,3)  -- ++(0,1) -- ++(150:1) -- ++ (0,-1) -- cycle;
\draw[fill=green!30] (150:2) ++ (0,0)  -- ++(0,1) -- ++(150:1) -- ++ (0,-1) -- cycle;
\draw[fill=green!30] (150:1) ++ (0,2)  -- ++(0,1) -- ++(150:1) -- ++ (0,-1) -- cycle;
\foreach \y in {3,6,7}
  \draw[fill=red!30] (0,\y) -- ++(0,1) -- ++ (210:1) -- ++(0,-1) -- cycle;
\foreach \y in {1,5,8}
  \draw[fill=red!30] (150:1) ++ (0,\y) -- ++(0,1) -- ++ (210:1) -- ++(0,-1) -- cycle;
\foreach \x in {3,4,5} {
  \draw[fill=red!30] (150:\x) ++ (0,10-\x) -- ++(0,1) -- ++ (210:1) -- ++(0,-1) -- cycle;
  \draw[fill=red!30] (150:\x-1) ++ (0,9-\x) -- ++(0,1) -- ++ (210:1) -- ++(0,-1) -- cycle;
}
\draw[fill=red!30] (150:4) ++ (0,3) -- ++(0,1) -- ++ (210:1) -- ++(0,-1) -- cycle;
\draw[fill=red!30] (150:2) ++ (0,4) -- ++(0,1) -- ++ (210:1) -- ++(0,-1) -- cycle;
\draw[fill=red!30] (150:2) ++ (0,2) -- ++(0,1) -- ++ (210:1) -- ++(0,-1) -- cycle;
\foreach \y in {1,3,6,7,8}
  \draw[fill=black!50] (0,\y) -- ++(30:1) -- ++(150:1) -- cycle;
\foreach \y in {1,3,6,7}
  \draw[fill=blue!30] (30:1) ++ (0,\y) -- ++(30:1) -- ++(150:1) -- ++ (30:-1) -- cycle;
\foreach \y in {2,4,6}
  \draw[fill=blue!30] (30:2) ++ (0,\y) -- ++(30:1) -- ++(150:1) -- ++ (30:-1) -- cycle;
\foreach \y in {2,4}
  \draw[fill=blue!30] (30:3) ++ (0,\y) -- ++(30:1) -- ++(150:1) -- ++ (30:-1) -- cycle;
\draw[fill=blue!30] (30:4) ++ (0,3) -- ++(30:1) -- ++(150:1) -- ++ (30:-1) -- cycle;
\foreach \x in {1,2,3,4,5} {
  \foreach \y in {9,10}
    \draw[fill=green!30] (30:\x) ++ (0,\y-\x)  -- ++(0,1) -- ++(150:1) -- ++ (0,-1) -- cycle;
}
\draw[fill=green!30] (30:2) ++ (0,1)  -- ++(0,1) -- ++(150:1) -- ++ (0,-1) -- cycle;
\draw[fill=green!30] (30:2) ++ (0,3)  -- ++(0,1) -- ++(150:1) -- ++ (0,-1) -- cycle;
\draw[fill=green!30] (30:4) ++ (0,2)  -- ++(0,1) -- ++(150:1) -- ++ (0,-1) -- cycle;
\draw[fill=green!30] (30:4) ++ (0,4)  -- ++(0,1) -- ++(150:1) -- ++ (0,-1) -- cycle;
\draw[fill=green!30] (30:5) ++ (0,3)  -- ++(0,1) -- ++(150:1) -- ++ (0,-1) -- cycle;
\foreach \y in {0,2,4,5}
  \draw[fill=red!30] (30:1) ++ (0,\y) -- ++(0,1) -- ++ (210:1) -- ++(0,-1) -- cycle;
\foreach \y in {0,5}
  \draw[fill=red!30] (30:2) ++ (0,\y) -- ++(0,1) -- ++ (210:1) -- ++(0,-1) -- cycle;
\foreach \x in {3,4,5} {
  \foreach \y in {0,1}
    \draw[fill=red!30] (30:\x) ++ (0,\y) -- ++(0,1) -- ++ (210:1) -- ++(0,-1) -- cycle;
}
\foreach \y in {3,5}
  \draw[fill=red!30] (30:3) ++ (0,\y) -- ++(0,1) -- ++ (210:1) -- ++(0,-1) -- cycle;
\draw[fill=red!30] (30:5) ++ (0,2) -- ++(0,1) -- ++ (210:1) -- ++(0,-1) -- cycle;
\end{tikzpicture}
\fi
\qquad \longleftrightarrow \qquad
\begin{gathered}
\begin{array}{ccccccccccc}
5 && 5 && 2 && 2 && 1 && 0 \\
& 5 && 3 && 2 && 2 && 0 \\
&& 5 && 3 && 2 && 1 \\
&&& 4 && 3 && 2 \\
&&&& 4 && 2 \\
&&&&& 4
\end{array}
\\[10pt]
\begin{array}{ccccccccc}
4 && 4 && 4 && 2 && 1 \\
& 4 && 4 && 2 && 1 \\
&& 4 && 3 && 2 \\
&&& 3 && 2 \\
&&&& 3
\end{array}
\end{gathered}
\]
\end{ex}

Next, we will interpret Theorem~\ref{thm:q_mult_dim_B} in terms of lozenge tilings.
In this case for $\Sp_{2k}$, we want to impose a horizontal reflection symmetry to the half hexagon, which further restricts us to the quarter hexagon.
Indeed, this requires that there are $2n$ steps on the left and $2k+1$ steps along the top and bottom sides and along the middle are all $B$ tiles.
We now consider the middle of the half hexagon to be the position $0$ on the boundary (\textit{i.e.}, height $0$), and so this symmetry and indexing in terms of the corresponding GT patterns is precisely those described for $\Sp_{2k}$.
In particular, tilings of this quater hexagon give the character for the corresponding $\spn_{2k}$ representation by the natural bijection with the type $C_k$ Proctor pattern.

We can translate this symmetry to the NILPs on the rectangular grid as the lattice paths must stay below the antidiagonal.
This means the NILP corresponds to nonintersecting Dyck paths (which do not necessarily have to end on the antidiagonal) and we obtain a determinant of triangle Catalan numbers.
Hence, the combinatorial skew Howe duality is simply taking two different types of lattice paths on these lozenge tiling similar to the type A case.

\begin{ex}
\label{ex:quarter_hexagon}
We consider a tiling of the quarter hexagon for $n = 4$ with $k = 3$ and the corresponding Proctor pattern and NILP from~\cite[Thm.~3.8, Thm.~4.12]{OS19II}:
\[
\begin{tikzpicture}[scale=0.5,baseline=50]
\clip(-7,-0.5) rectangle (1,8.2);  
\foreach \y in {0,1,4,6}
  \draw[fill=blue!30] (0,\y) -- ++(30:1) -- ++(150:1) -- ++ (30:-1) -- cycle;
\foreach \y in {0,2,5}
  \draw[fill=blue!30] (150:1) ++ (0,\y) -- ++(30:1) -- ++(150:1) -- ++ (30:-1) -- cycle;
\foreach \y in {-1,1,3}
  \draw[fill=blue!30] (150:2) ++ (0,\y) -- ++(30:1) -- ++(150:1) -- ++ (30:-1) -- cycle;
\foreach \y in {-1,2}
  \draw[fill=blue!30] (150:3) ++ (0,\y) -- ++(30:1) -- ++(150:1) -- ++ (30:-1) -- cycle;
  \foreach \y in {-2,1}
  \draw[fill=blue!30] (150:4) ++ (0,\y) -- ++(30:1) -- ++(150:1) -- ++ (30:-1) -- cycle;
\draw[fill=blue!30] (150:5) ++ (0,-1) -- ++(30:1) -- ++(150:1) -- ++ (30:-1) -- cycle;
\draw[fill=blue!30] (150:6) ++ (0,-3) -- ++(30:1) -- ++(150:1) -- ++ (30:-1) -- cycle;
\draw[fill=green!30] (0,3)  -- ++(0,1) -- ++(150:1) -- ++ (0,-1) -- cycle;
\draw[fill=green!30] (150:1) ++ (0,4)  -- ++(0,1) -- ++(150:1) -- ++ (0,-1) -- cycle;
\draw[fill=green!30] (150:2) ++ (0,0)  -- ++(0,1) -- ++(150:1) -- ++ (0,-1) -- cycle;
\draw[fill=green!30] (150:4) ++ (0,0)  -- ++(0,1) -- ++(150:1) -- ++ (0,-1) -- cycle;
\draw[fill=green!30] (150:5) ++ (0,-2)  -- ++(0,1) -- ++(150:1) -- ++ (0,-1) -- cycle;
\foreach \y in {2,5,7}
  \draw[fill=red!30] (0,\y) -- ++(0,1) -- ++ (210:1) -- ++(0,-1) -- cycle;
\foreach \y in {1,3,6}
  \draw[fill=red!30] (150:1) ++ (0,\y) -- ++(0,1) -- ++ (210:1) -- ++(0,-1) -- cycle;
\foreach \y in {2,4,5}
  \draw[fill=red!30] (150:2) ++ (0,\y) -- ++(0,1) -- ++ (210:1) -- ++(0,-1) -- cycle;
\foreach \y in {0,1,3,4}
  \draw[fill=red!30] (150:3) ++ (0,\y) -- ++(0,1) -- ++ (210:1) -- ++(0,-1) -- cycle;
\foreach \y in {-1,2,3}
  \draw[fill=red!30] (150:4) ++ (0,\y) -- ++(0,1) -- ++ (210:1) -- ++(0,-1) -- cycle;
\foreach \y in {0,1,2}
  \draw[fill=red!30] (150:5) ++ (0,\y) -- ++(0,1) -- ++ (210:1) -- ++(0,-1) -- cycle;
\foreach \y in {-2,-1,0,1}
  \draw[fill=red!30] (150:6) ++ (0,\y) -- ++(0,1) -- ++ (210:1) -- ++(0,-1) -- cycle;
\end{tikzpicture}
\longleftrightarrow
\begin{array}{ccccccc}
3 && 2 && 0 \\
& 3 &&1 && 0 \\
&& 2 && 1 \\
&&& 2 && 0 \\
&&&& 2 \\
&&&&& 1
\end{array}
\longleftrightarrow
\begin{tikzpicture}[scale=0.48,baseline=50]
\foreach \x in {0,1,...,7} {
  \draw[-, gray, very thin] (\x,\x) -- (13,\x);
  \draw[-, gray, very thin] (\x,0) -- (\x,\x);
}
\foreach \x in {8,...,13}
  \draw[-, gray, very thin] (\x,0) -- (\x,7);
\draw[-, thick] (0,0) -- (7,7);
\draw[-, dashed] (3,3) -- (6,0);
\draw[-, dashed] (6.5,6.5) -- (13,0);
\foreach \x in {0,1,2,3}
  \draw[very thick, blue, dashed] (6-\x, \x) -- (7-\x, \x);
\draw[very thick, blue, line join=round]  (4,3) -- (4,4) -- (6,4) -- (6,5) -- (8,5);
\draw[very thick, blue, line join=round]  (5,2) -- (6,2) -- (6,3) -- (9,3) -- (9,4);
\draw[very thick, blue, line join=round]  (6,1) -- (10,1) -- (10,2) -- (11,2);
\draw[very thick, blue, line join=round]  (7,0) -- (13,0);
\foreach \i in {0,1,2,3} {
  \draw[very thick, red] (\i,\i) -- (6-\i,\i);
  \draw[fill=purple, color=purple] (6-\i, \i) circle (0.12);
}
\foreach \i in {0,1,2,3} {
  \draw[fill=red, color=red] (\i, \i) circle (0.12);
  \draw (\i, \i) node[anchor=east] {$s_{\i}$};
}
\foreach \i/\x in {3/1,2/2,1/4,0/6} {
  \draw[fill=purple, color=purple] (7+\x, 6-\x) circle (0.12);
  \draw (7+\x, 6-\x) node[anchor=south west] {$t_{\i}$};
}
\end{tikzpicture}
\]
Note that if we think of the lozenge tiling as a stacking of boxes in the corner with the $B$ tiles at height $0$ being the floor, the heights along the diagonals are the diagonals of the Proctor pattern.
It is easy to see this holds in general.
\end{ex}

For Theorem~\ref{thm:q_mult_dim_B} with $V(\fw_n)^{\otimes 2k}$ in type $B_n$, we cannot have full symmetry of the (half) hexagon.
Instead we consider lozenge tilings of the half hexagon that are \defn{almost symmetric}, where they are symmetric up to the middle row of hexagons, which are then forced to be either
\[
\begin{tikzpicture}[scale=0.8]
\draw[fill=blue!30] (0,1) -- ++(30:1) -- ++(150:1) -- ++ (30:-1) -- cycle;
\draw[fill=green!30] (0,0)  -- ++(0,1) -- ++(150:1) -- ++ (0,-1) -- cycle;
\draw[fill=red!30] (0,1) -- ++(30:1) -- ++(0,-1) -- ++(210:1) -- cycle;
\end{tikzpicture}
\hspace{60pt}
\begin{tikzpicture}[scale=0.8]
\draw[fill=blue!30] (0,0) -- ++(-30:1) -- ++(-150:1) -- ++ (-30:-1) -- cycle;
\draw[fill=green!30] (0,0)  -- ++(0,1) -- ++(-30:1) -- ++ (0,-1) -- cycle;
\draw[fill=red!30] (0,1) -- ++(210:1) -- ++(0,-1) -- ++(30:1) -- cycle;
\end{tikzpicture}
\]
There are $2^k$ such possible choices, where we take $k-1$ of them to correspond to the sign vector and the last to correspond to the parity (that is, taking $\mu + \omega_k$ or $\mu + \omega_{k-1}$).
The remaining part of the lozenge tiling corresponds to the King tableaux as in the case of a symmetric lozenge tiling.
This choice is also the difference between the type $C_k$ and $B_k$ Proctor patterns allowing the rightmost entries to be in $\frac{1}{2}\ZZ_{\geq 0}$.
So we can realize our combinatorial skew Howe duality as a full hexagon as for the $\gl_k$ case with as much symmetry as possible.

\begin{ex}
We consider the quarter hexagon from Example~\ref{ex:quarter_hexagon}, reflect it vertically to a type $C_4$ symmetric half hexagon, and then adjoin a type $B_3$ almost symmetric half hexagon:
\[
\iftikz
\begin{tikzpicture}[scale=0.5,baseline=50]
\foreach \y in {-6,-4,-1,0,1,4,6}
  \draw[fill=black!50] (0,\y) -- ++(150:1) -- ++(30:1) -- cycle;
\foreach \y in {-6,-3,-1,0,2,5}
  \draw[fill=blue!30] (150:1) ++ (0,\y) -- ++(30:1) -- ++(150:1) -- ++ (30:-1) -- cycle;
\foreach \y in {-5,-3,-1,1,3}
  \draw[fill=blue!30] (150:2) ++ (0,\y) -- ++(30:1) -- ++(150:1) -- ++ (30:-1) -- cycle;
\foreach \y in {-5,-2,-1,2}
  \draw[fill=blue!30] (150:3) ++ (0,\y) -- ++(30:1) -- ++(150:1) -- ++ (30:-1) -- cycle;
  \foreach \y in {-5,-2,1}
  \draw[fill=blue!30] (150:4) ++ (0,\y) -- ++(30:1) -- ++(150:1) -- ++ (30:-1) -- cycle;
\foreach \y in {-4,-1}
  \draw[fill=blue!30] (150:5) ++ (0,\y) -- ++(30:1) -- ++(150:1) -- ++ (30:-1) -- cycle;
\draw[fill=blue!30] (150:6) ++ (0,-3) -- ++(30:1) -- ++(150:1) -- ++ (30:-1) -- cycle;
\foreach \y in {-7,-5,-2,3}
  \draw[fill=green!30] (0,\y)  -- ++(0,1) -- ++(150:1) -- ++ (0,-1) -- cycle;
\foreach \y in {-7,-4,-2,4}
  \draw[fill=green!30] (150:1) ++ (0,\y)  -- ++(0,1) -- ++(150:1) -- ++ (0,-1) -- cycle;
\foreach \y in {-7,-6,-4,0}
  \draw[fill=green!30] (150:2) ++ (0,\y)  -- ++(0,1) -- ++(150:1) -- ++ (0,-1) -- cycle;
\foreach \y in {-7,-6,-4,-3}
  \draw[fill=green!30] (150:3) ++ (0,\y)  -- ++(0,1) -- ++(150:1) -- ++ (0,-1) -- cycle;
\foreach \y in {-7,-6,-3,0}
  \draw[fill=green!30] (150:4) ++ (0,\y)  -- ++(0,1) -- ++(150:1) -- ++ (0,-1) -- cycle;
\foreach \y in {-7,-6,-5,-2}
  \draw[fill=green!30] (150:5) ++ (0,\y)  -- ++(0,1) -- ++(150:1) -- ++ (0,-1) -- cycle;
\foreach \y in {-7,-6,-5,-4}
  \draw[fill=green!30] (150:6) ++ (0,\y)  -- ++(0,1) -- ++(150:1) -- ++ (0,-1) -- cycle;
\foreach \y in {-3,2,5,7}
  \draw[fill=red!30] (0,\y) -- ++(0,1) -- ++ (210:1) -- ++(0,-1) -- cycle;
\foreach \y in {-5,1,3,6}
  \draw[fill=red!30] (150:1) ++ (0,\y) -- ++(0,1) -- ++ (210:1) -- ++(0,-1) -- cycle;
\foreach \y in {-2,2,4,5}
  \draw[fill=red!30] (150:2) ++ (0,\y) -- ++(0,1) -- ++ (210:1) -- ++(0,-1) -- cycle;
\foreach \y in {0,1,3,4}
  \draw[fill=red!30] (150:3) ++ (0,\y) -- ++(0,1) -- ++ (210:1) -- ++(0,-1) -- cycle;
\foreach \y in {-4,-1,2,3}
  \draw[fill=red!30] (150:4) ++ (0,\y) -- ++(0,1) -- ++ (210:1) -- ++(0,-1) -- cycle;
\foreach \y in {-3,0,1,2}
  \draw[fill=red!30] (150:5) ++ (0,\y) -- ++(0,1) -- ++ (210:1) -- ++(0,-1) -- cycle;
\foreach \y in {-2,-1,0,1}
  \draw[fill=red!30] (150:6) ++ (0,\y) -- ++(0,1) -- ++ (210:1) -- ++(0,-1) -- cycle;
\foreach \y in {-7,-5,-3,-2,2,3,5,7}
  \draw[fill=black!50] (0,\y) -- ++(30:1) -- ++(150:1) -- cycle;
\foreach \y in {-5,-4,-2,1,3,5,6}
  \draw[fill=blue!30] (30:1) ++ (0,\y-1) -- ++(30:1) -- ++(150:1) -- ++ (30:-1) -- cycle;
\foreach \y in {-5,-4,-1,1,4,5}
  \draw[fill=blue!30] (30:2) ++ (0,\y-1) -- ++(30:1) -- ++(150:1) -- ++ (30:-1) -- cycle;
\foreach \y in {-4,-2,1,3,5}
  \draw[fill=blue!30] (30:3) ++ (0,\y-2) -- ++(30:1) -- ++(150:1) -- ++ (30:-1) -- cycle;
\foreach \y in {-4,-1,1,4}
  \draw[fill=blue!30] (30:4) ++ (0,\y-2) -- ++(30:1) -- ++(150:1) -- ++ (30:-1) -- cycle;
\foreach \y in {-3,0,4}
  \draw[fill=blue!30] (30:5) ++ (0,\y-3) -- ++(30:1) -- ++(150:1) -- ++ (30:-1) -- cycle;
\foreach \y in {-2,2}
  \draw[fill=blue!30] (30:6) ++ (0,\y-3) -- ++(30:1) -- ++(150:1) -- ++ (30:-1) -- cycle;
\draw[fill=blue!30] (30:7) ++ (0,-3) -- ++(30:1) -- ++(150:1) -- ++ (30:-1) -- cycle;
\foreach \y in {-6,-1,0,4}
  \draw[fill=green!30] (30:1) ++ (0,\y-1)  -- ++(0,1) -- ++(150:1) -- ++ (0,-1) -- cycle;
\foreach \y in {-2,3,6}
  \draw[fill=green!30] (30:2) ++ (0,\y-1)  -- ++(0,1) -- ++(150:1) -- ++ (0,-1) -- cycle;
\foreach \y in {-3,0,2,6}
  \draw[fill=green!30] (30:3) ++ (0,\y-2)  -- ++(0,1) -- ++(150:1) -- ++ (0,-1) -- cycle;
\foreach \y in {-2,3,5}
  \draw[fill=green!30] (30:4) ++ (0,\y-2)  -- ++(0,1) -- ++(150:1) -- ++ (0,-1) -- cycle;
\foreach \y in {2,3,5}
  \draw[fill=green!30] (30:5) ++ (0,\y-3)  -- ++(0,1) -- ++(150:1) -- ++ (0,-1) -- cycle;
\foreach \y in {-3,0,1,4}
  \draw[fill=green!30] (30:6) ++ (0,\y-3)  -- ++(0,1) -- ++(150:1) -- ++ (0,-1) -- cycle;
\foreach \y in {-1,0,3,4}
  \draw[fill=green!30] (30:7) ++ (0,\y-4)  -- ++(0,1) -- ++(150:1) -- ++ (0,-1) -- cycle;
\foreach \y in {1,2,3}
  \draw[fill=green!30] (30:8) ++ (0,\y-4)  -- ++(0,1) -- ++(150:1) -- ++ (0,-1) -- cycle;
\foreach \y in {-3,2,7}
  \draw[fill=red!30] (30:1) ++ (0,\y-1) -- ++(0,1) -- ++ (210:1) -- ++(0,-1) -- cycle;
\foreach \y in {-6,-3,0,2}
  \draw[fill=red!30] (30:2) ++ (0,\y-1) -- ++(0,1) -- ++ (210:1) -- ++(0,-1) -- cycle;
\foreach \y in {-5,-1,4}
  \draw[fill=red!30] (30:3) ++ (0,\y-2) -- ++(0,1) -- ++ (210:1) -- ++(0,-1) -- cycle;
\foreach \y in {-5,-3,0,2}
  \draw[fill=red!30] (30:4) ++ (0,\y-2) -- ++(0,1) -- ++ (210:1) -- ++(0,-1) -- cycle;
\foreach \y in {-4,-2,-1,1}
  \draw[fill=red!30] (30:5) ++ (0,\y-3) -- ++(0,1) -- ++ (210:1) -- ++(0,-1) -- cycle;
\foreach \y in {-4,-1,3}
  \draw[fill=red!30] (30:6) ++ (0,\y-3) -- ++(0,1) -- ++ (210:1) -- ++(0,-1) -- cycle;
\foreach \y in {-3,-2,2}
  \draw[fill=red!30] (30:7) ++ (0,\y-4) -- ++(0,1) -- ++ (210:1) -- ++(0,-1) -- cycle;
\foreach \y in {-3,-2,-1,0}
  \draw[fill=red!30] (30:8) ++ (0,\y-4) -- ++(0,1) -- ++ (210:1) -- ++(0,-1) -- cycle;
\end{tikzpicture}
\fi
\]
\end{ex}

Finally, we interpret Theorem~\ref{thm:mult_type_D} in terms of lozenge tilings similar to the previous case.
For the type $D_n$ case, we have symmetry in the $B$ tiles except for the middle $B$ tile, which is a direct translation of the Proctor pattern condition~\cite[Thm.~7.2]{Proctor94}.
Note that this allows for a greater amount of asymmetry than in a type $B_n$ lozenge tiling.
The paths along such tilings avoiding the $B$ tiles precisely correspond to the \emph{unfolded} NILPs.
In fact, this shows that Corollary~\ref{cor:q_mult_D_even} holds at $q = 1$.


\section{Limit shapes of Young diagrams}
\label{sec:limit-shapes-young}

In this section, we demonstrate how the limit shapes of the random Young diagrams with respect to the probability measure from skew Howe duality, see~\eqref{eq:GL_prob_measure} below, can be derived using a variational formulation of the limit shape.

Here we demonstrate that the limit shapes of partitions using the skew Howe duality measures of classical Lie groups are described by the same function, which is computed explicitly.
Therefore we need to use the parameters that are related in a certain way.
Furthermore, let $l$ be such that $n = 2l$ or $n = 2l + 1$ depending on if $n$ is even or odd.
Recall that skew Howe duality acts on the space $\bigwedge\left(\CC^{n}\otimes \CC^{k}\right)$, which has a natural action of the group $\GL_{n} \times \GL_{k}$ and a multiplicity-free decomposition given by~\eqref{eq:gl_gl_skew_Howe}.
This space also has an action of the Clifford algebra as discussed in Section~\ref{sec:skew_howe} with an invariant subspace $\bigwedge\left( \CC^{n}\otimes \CC^k \right)$.
The invariant subspace also has actions of
\begin{itemize}
\item $\SO_{2l+1}\times \Pin_{2k}$ for $n = 2l+1$,
\item $\Or_{2l}\times \SO_{2k}$ for $n=2l$, and
\item $\Sp_{2l}\times \Sp_{2k}$ for $n=2l$,
\end{itemize}
with multiplicity-free decompositions given by~\eqref{eq:so2np1},~\eqref{eq:so2nso2l}, and~\eqref{eq:spsp}, respectively, using (generalized) Young diagrams for the corresponding Lie groups.
Let ${\bf 1}_{[a,b]}(x)$ denote the indicator function, which is $1$ if $x \in [a,b]$ and $0$ otherwise.

\begin{thm}
  \label{thm:limit_shape_gl}
  The decomposition of $\bigwedge\left(\CC^{n}\otimes \CC^{k}\right)$ gives rise to the probability measure
\begin{subequations}
\label{eq:prob_measures}
\begin{equation}
  \label{eq:GL_prob_measure}
  \mu_{n,k}(\lambda)=\frac{\dim V_{\GL_{n}}(\lambda)\cdot\dim V_{\GL_{k}}(\overline{\lambda}')}{2^{nk}}, 
\end{equation}
for the action of $GL_{n}\times GL_{k}$ and, for $k$ even, $\bigwedge\left(\CC^{n}\otimes \CC^{k/2}\right)$ gives rise to
\begin{equation}
  \label{eq:other_prob_measure}
  \mu_{n,k/2}(\lambda)=\frac{\dim V_{G_{1}}(\lambda)\cdot\dim V_{G_{2}}(\overline{\lambda}')}{2^{nk/2}},
\end{equation}
\end{subequations}
for the actions of $\SO_{2l+1}\times \Pin_{k}$ for $n=2l+1$, $\Or_{2l} \times \SO_{k}$ for $n=2l$, and $\Sp_{2l} \times \Sp_{k}$ for $n=2l$.

  Let $f_{n}$ denote the upper boundary of a Young diagram in a decomposition, rotated and scaled by $\frac{1}{n}$ as in Figure~\ref{fig:gl-rotated-diagram} and regarded as a function $f_{n}(x)$ of $x\in [0,c+1]$.
  As  $n\to\infty,\; k\to\infty$ in such a way that $c=\lim_{n,k\to\infty}\frac{k}{n}=\mathrm{const}$ and $\frac{k}{n}=c+\mathcal{O}\left(\frac{1}{n}\right)$, the functions $f_n$ converge in probability with respect to the probability measure~\eqref{eq:prob_measures} in the supremum norm $\Abs{\cdot}_{\infty}$ to the limiting shape given by the formula
  \begin{equation}
    \label{eq:limit-shape-for-f}
    f(x) = \begin{cases}
    \displaystyle 1 + \int_{0}^{x} \bigl( 1-2\rho(t) \bigr) \dt & \text{if } c \geq 1,
    \\[10pt]
    \displaystyle 1 + \int_{0}^{x} \bigl( 2\rho(t)-1 \bigr) \dt & \text{if } c < 1,
    \end{cases}
  \end{equation}
 where the limit density $\rho(x)$ is written explicitly as
\begin{equation}
  \label{eq:limit-shape-gl}
  \rho(x)=
      \frac{{\bf 1}_{[-\sqrt{c},\sqrt{c}]}(\widetilde{x})}{2\pi} \left[
      \arctan \left(\frac{-(c+1) \widetilde{x}+2c}{(c-1) \sqrt{c-\widetilde{x}^2}}\right)+
      \arctan\left(\frac{ (c+1) \widetilde{x} +2c}{(c-1)
          \sqrt{c-\widetilde{x}^2}}\right)\right],
  \end{equation}
  where $\widetilde{x} = x-\frac{c+1}{2}$,
 for~\eqref{eq:GL_prob_measure} and with a shifted argument $\rho\bigl(x+\frac{c+1}{2}\bigr)$ such that $x \in [0, \frac{c+1}{2}]$ for~\eqref{eq:other_prob_measure}.
\end{thm}

Note that the limit shape of the diagrams for the special orthogonal and symplectic groups is a ``half'' of the limit shape for the general linear groups (see Figures~\ref{fig:soo-over-gl} and~\ref{fig:sp-over-gl}).
The case $c=1$ corresponds to a constant solution $\rho(x)=\frac{1}{2}$, for $x\in[-1,1]$ and a triangular limit shape.
A heuristic for this is we are considering a charged line without any external field with an equal number of holes and charges, so the charges will want to spread out uniformly.
Another way to see why this is expected is that the measure has rotational symmetry and it wants to avoid partitions with nonconstant slopes as it behaves very roughly like $x(1-x)$ on $x \in [0,1]$ (\textit{cf.}\ Figure~\ref{fig:rho} below).

The proof of the theorem is presented case by case in the following Sections \ref{sec:limit-shape-gl_n}, \ref{sec:limit-shape-series-B}, \ref{sec:limit-shape-series-C}, \ref{sec:limit-shape-series-D}.
We demonstrate the derivation of the limit shape in the $\GL_{n}\times \GL_{k}$ Section~\ref{sec:limit-shape-gl_n}.

\subsection{Limit shape for \texorpdfstring{$(\GL_{n},\GL_{k})$}{(GL(n), GL(k))} skew Howe duality}
\label{sec:limit-shape-gl_n}

For a partition $\lambda$, we will use the coordinates $a_{i} := \lambda_{i}+n-i$, which correspond to the rotated diagram as demonstrated in the Figure~\ref{fig:gl-rotated-diagram}.
To derive the limit shape~\eqref{eq:limit-shape-gl}, we write the probability measure in the form
\begin{equation}
  \label{eq:measure_factorized}
  \mu_{n,k}(\{a_{i}\}) = C_{n,k} \prod_{i< j} (a_{i}-a_{j})^{2} \times \prod_{l} W(a_{l}),
\end{equation}
as the product of the square of the Vandermonde determinant and the product of single variable dependent weights:
\begin{align}
  \mu_{n,k}(\lambda) & = \frac{\dim V_{n}(\lambda)\cdot\dim V_{k}(\overline{\lambda}')}{2^{nk}}=\frac{M^k(\lambda)\dim V_{n}(\lambda)}{2^{nk}}= \nonumber \\
  & = \displaystyle\prod_{m=0}^{n-1} \frac{ (k+m)!}{2^{k} m!(k+n-1)!}\times\displaystyle\prod_{1 \leq i < j \leq n} (a_{i}-a_{j})^{2}\times \displaystyle \prod_{i=1}^{n}\frac{(k+n-1)!}{ a_{i}! (k+n-1  - a_{i})!},  \nonumber \\
  & = \displaystyle\prod_{m=0}^{n-1} \frac{ (k+m)!}{2^{k} m!(k+n-1)!}\times\displaystyle\prod_{1 \leq i < j \leq n} (a_{i}-a_{j})^{2}\times \displaystyle \prod_{i=1}^{n}\binom{k+n-1}{a_{i}},
  \label{eq:gln-measure-krawtchouk}
\end{align}
where we have used the Weyl dimension formula
\begin{equation}
\label{eq:weyl_dim}
\dim V_{n}(\lambda)=\frac{\prod_{i<j}(\lambda_{i}-\lambda_{j}+j-i)}{\prod_{m=0}^{n-1}m!}.
\end{equation}
In this form the probability measure $\mu_{n,k}(a_{1}, \dotsc, a_{n})$ is the measure for the configurations of the Krawtchouk polynomial ensemble (\textit{cf.}~\cite[Lemma~5.1]{borodin2007asymptotics}), since the weights are given by the binomial coefficients and the asymptotic results of \cite{johansson2002non} can be applied.
Nevertheless, we will see that the measures $\mu_{n,k/2}(\lambda)$ for Lie groups of series $\SO_{2l+1}$, $\Sp_{2l}$, $\SO_{2l}$ do not exactly coincide with the Krawtchouk ensemble (see~\eqref{eq:so2p1_coord_measure},~\eqref{eq:sp_coord_measure}, and~\eqref{eq:so2n_coord_measure}).
Therefore, in this subsection we present a method for derivation of the limit shape for $(\GL_{n},\GL_{k})$ that can then be applied with slight modifications to other classical series of simple Lie groups in Sections~\ref{sec:limit-shape-series-B}, \ref{sec:limit-shape-series-C}, \ref{sec:limit-shape-series-D}.

We are interested in the limit $n,k\to\infty$ such that $\frac{k}{n} = c+\mathcal{O}\left(\frac{1}{n}\right)$.
In this case $\GL_{n}$ and $\GL_{k}$ appear in the same way, so without loss of generality we can assume that $k > n$.
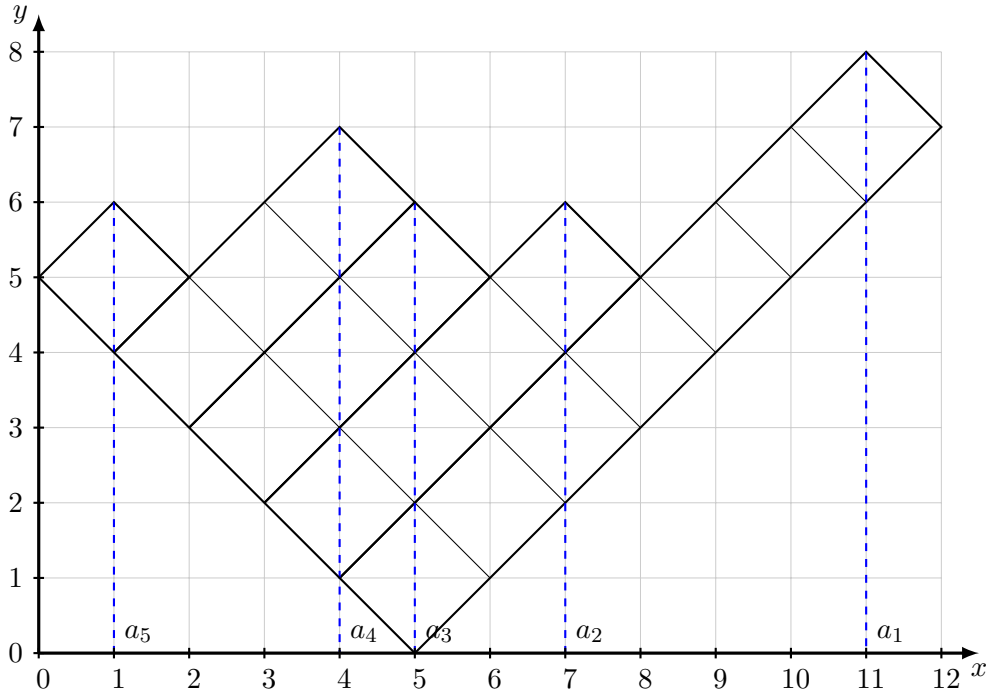
\begin{figure}[t]
\[
\iftikz
  \begin{tikzpicture}[baseline=10, scale=1.0]
    \newcount\n
  \n=5;
\draw[gray!40, very thin] (0,0) grid (12,8);  
  
\foreach \y [count=\i from 0] in {7,4,3,3,1} {
  \draw[black, thick] (\n-\i,\i) -- (\n+\y-\i,\i+\y)--(\n-1+\y-\i,\i+1+\y)--
  (\n-1-\i,\i+1);
  \tikzmath{
    integer \k;
    \k=\i+1;
  }
   \draw[blue, thick, dashed] (\n-1+\y-\i,\i+1+\y) -- (\n-1+\y-\i,0) node[anchor=south west, black] {$a_{\k}$};
  \foreach \len in {0,...,\y} {
    \draw[black,thin] (\n-\i+\len,\i+\len) -- (\n-1-\i+\len,\i+1+\len);
  }
}
\draw[black, thick] (\n,0)--(0,\n);
\tkzInit[xmin=0,xmax=12,ymin=0,ymax=8]
\tkzAxeXY[very thick]
\end{tikzpicture}
\fi
\]
\caption{Rotated diagram for $\GL_{5}$, coordinates $a_{i}=\lambda_{i}+n-i$
  correspond to the left boundaries of the intervals, where the upper boundary
  as a function $f_{n}$ is decreasing.}
\label{fig:gl-rotated-diagram}
\end{figure}
Rescale the coordinates as $x_{i}=\frac{a_{i}}{n}=\frac{\lambda_{i}+n-i}{n}$ and regard the upper boundary of the rotated diagram as a piecewise-linear function $f_{n}(x)$, so $f_{n}'(x) = \pm 1$ for $x\neq \frac{j}{n}, j\in \ZZ$.
To derive the limit shape it is convenient to consider the diagram as a particle configuration with particle coordinates $\{x_{i}\}_{i=1}^{n}$.
Introduce the piecewise constant function $\rho_{n}(x)=\frac{1}{2}(1-f_{n}'(x))$ that is equal to zero on an interval of the length $\frac{1}{n}$ if there is no particle in the left boundary of the interval and is equal to $1$ if there is a particle.
Then $\rho_{n}(x)$ can be called particle density.
The convergence of the diagrams to the limit shape leads to the convergence of particle density functions $\rho_{n}$ to a limit particle density $\rho(x)$, where the limit density $\rho(x)$ is connected to a derivative of limit function $f(x)$ of the diagrams by the formula
\[
  f'(x)=1-2\rho(x).
\]
The limit shape can be recovered from the explicit expression for
$\rho(x)$ by the formula
\begin{equation}
 \label{eq:limit_density}
  f(x)=1+\int_{0}^{x}(1-2\rho(t))\dt.
\end{equation}
It is more convenient to solve the variational problem for the limit density $\rho(x)$.

The probability of a configuration $\{x_{i}\}_{i=1}^{n}$ can be written as an exponent of a functional $J[\rho_{n}]$:
\[
  \mu_{n,k}(\{x_{i}\}) = \frac{1}{Z_{n}} \exp\bigl( -n^{2} J[\rho_{n}] + \bigO(n\ln n) \bigr),
\]
where
\begin{equation}
  \label{eq:limit-shape-functional}
  J[\rho_{n}] = \int_{0}^{c+1}\int_{0}^{c+1} \rho_{n}(x)\rho_{n}(y) \ln\abs{x-y}^{-1}\dx\dy+\int_{0}^{c+1}\rho_{n}(x)\; V(x)\dx
\end{equation}
and the normalization constant $Z_{n}$ does not depend on $\{x_{i}\}$.

We omit the computation of the normalization constant and the estimate of the next order term $\mathcal{O}(n\ln n)$, which are straightforward and completely parallel to the computations for $\SO_{2n+1}$ presented in~\cite[Lemmas 1,2]{NNP20}. 
The potential $V(x)$ appears from the use of Stirling formula for the factorials in Equation~\eqref{eq:gln-measure-krawtchouk} and has the form
\[
  V(x)=x\ln x+(c+1-x)\ln (c+1-x).
\]
By~\cite[Thm.~2.1]{Johansson-1998} and arguments similar to~\cite[Thm.~6.27]{deift1999orthogonal} the functional is strictly convex and the minimizer is unique by an analogue of~\cite[Thm.~6.132]{deift1999orthogonal} (see also \cite[Thm.~1.16]{romik2015surprising}).
The minimizer is constructed explicitly in the following lemma.

\begin{lemma}
\label{lemma:limit-shape-function}
The minimizer of the functional~\eqref{eq:limit-shape-functional} is given by the formula~\eqref{eq:limit-shape-gl}.
\end{lemma}

\begin{proof}
If we shift the coordinates as $\widetilde{x}=x-\frac{c+1}{2}$ and introduce the function $\widetilde{\rho}_{n}(\widetilde{x})=\rho_{n}(x)$,  we can make the functional invariant with respect to the sign flip $\widetilde{x}\to -\widetilde{x}$:
\begin{equation}
\label{eq:J_functional}
\begin{aligned}
  J[\widetilde{\rho}_{n}] &= \int_{-\frac{c+1}{2}}^{\frac{c+1}{2}}\int_{-\frac{c+1}{2}}^{\frac{c+1}{2}}\widetilde{\rho}_{n}(x)\widetilde{\rho}_{n}(y)\ln\abs{x-y}^{-1}\dx\dy \\
  & \hspace{20pt} +\int_{-\frac{c+1}{2}}^{\frac{c+1}{2}}\widetilde{\rho}_{n}(x)\left[\left(\frac{c+1}{2}-x\right)\ln\left(\frac{c+1}{2}-x\right)+\left(x+\frac{c+1}{2}\right)\ln\left(x+\frac{c+1}{2}\right)\right]\dx.
\end{aligned}
\end{equation}
Now we need to find a minimizer in the class of the even functions $\widetilde{\rho}(x)$ such that $\abs{\widetilde{\rho}(x)}< 1$ for any $x$ with the normalization condition
\begin{equation}
  \label{eq:12}
  \int_{-\frac{c+1}{2}}^{\frac{c+1}{2}}\widetilde{\rho}(x)\dx=1.
\end{equation}
We remark that this is a consequence that we have $n = \ell(\lambda) = \lambda_1'$ particles.

Assume that the minimizer $\widetilde{\rho}$ is supported on an interval $[-a,a]$.
Taking the variation by $\widetilde{\rho}$ and redefining the potential as $\widetilde{V}(\widetilde{x})=\frac{1}{2}V(x)$, we obtain an Euler--Lagrange equation for $x \in \supp \widetilde{\rho}$:
  \begin{equation}
    \label{eq:euler_lagrange_x}
    \int_{-a}^{a}\ln|x-y|^{-1}\widetilde{\rho}(y)\dy +\widetilde{V}(x)=\mathrm{const}.
  \end{equation}
To write the solution we take the derivative of Equation~\eqref{eq:euler_lagrange_x} and arrive at the electrostatic equilibrium condition
\begin{equation}
  \label{eq:electrostatic_equilibrium}
  -\int_{-a}^{a}\frac{\widetilde{\rho}(y)\dy}{y-x}+\widetilde{V}'(x)=0.
\end{equation}
Then we denote the Hilbert transform of $\widetilde{\rho}(x)$ by
\[
   G(z) := -i\int_{-a}^{a}\frac{\widetilde{\rho}(y)}{y-z}\dy,
\]
which can be defined on any complex number $z \in \CC$.
In the sequel, we will have $z$ denoting a complex number and $x$ being a real number.
Note that $G(z)$ is analytic on $\CC \setminus [-a,a]$ with limit values given by
\begin{align*}
  G_{\pm}(x) & = \lim_{\varepsilon\to 0}\frac{1}{i}\int\frac{\widetilde{\rho}(y)\dy}{y-(x\pm i\varepsilon)}
  = \lim_{\varepsilon\to 0}\frac{1}{i}\int\frac{y-x\pm i\varepsilon}{(y-x)^2+\varepsilon}\widetilde{\rho}(y)\dy
  \\ & =-i \pv \int \frac{\widetilde{\rho}(y)\dy}{y-x}\pm\pi\widetilde{\rho}(x),
\end{align*}
where we have used $\dfrac{\varepsilon}{\pi(x^2+\varepsilon^2)}\to\delta(x)$.
Thus we arrive at
\[
  G_{\pm}(x)=\pm\pi\widetilde{\rho}(x)+i\widetilde{V}'(x),
\]
so on the support of $\widetilde{\rho}(x)$ we have
\begin{equation}
  \label{eq:limit_points_sum}
  G_+(x)+G_-(x)=2i\widetilde{V}'(x),\qquad x\in[-a,a],
\end{equation}
and outside of $[-a,a]$ the following conditions appear
\begin{subequations}
\label{eq:hilbert_transform_cond}
\begin{align}
    & G_+(x)-G_-(x)=0, \qquad x \notin [-a,a],\\
    & G(z) \to 0, \qquad \text{as } z \to \infty.
  \end{align}
\end{subequations}
Now we have a Riemann--Hilbert problem for $G(z)$, but the condition~\eqref{eq:limit_points_sum} is in a non-standard form with the sum instead of a difference.
We need to redefine $G$ in such a way  as to obtain a standard problem that can be solved by the Plemelj formula~\cite{deift1999orthogonal}:
\[
  \widetilde{G}(z) := \frac{G(z)}{\sqrt{z^2-a^2}}.
\]
Then we get
\begin{align*}
  \widetilde{G}_+(x)-\widetilde{G}_-(x) & = \frac{G_+(x)}{\left(\sqrt{x^2-a^2}\right)_+}-\frac{G_-(x)}{\left(\sqrt{x^2-a^2}\right)_-} 
  = \frac{G_+(x)+G_-(x)}{\left(\sqrt{x^2-a^2}\right)_+}=\frac{2i\widetilde{V}'(x)}{\left(\sqrt{x^2-a^2}\right)_+},
\end{align*}
where the branch of the square root changes the sign crossing the real line
\[
\left(\sqrt{x^2-a^2}\right)_+=-\left(\sqrt{x^2-a^2}\right)_-,\qquad x\in[-a,a].
\]
The conditions~\eqref{eq:hilbert_transform_cond} are preserved for $\widetilde{G}$:
\begin{align*}
  & \widetilde{G}_+(x)-\widetilde{G}_-(x)=0, \qquad x \notin[-a,a], \\
  & \widetilde{G}(z)\to 0 \qquad \text{as } z\to\infty.
\end{align*}
Then $\widetilde{G}(z)$ is a solution of the standard Riemann--Hilbert problem and is given by the Plemelj formula
\begin{align*}
  \widetilde{G}(z) & = \frac{1}{2\pi i}\int_{-a}^a\frac{2i\widetilde{V}'(s)\ds}{\left(\sqrt{s^2-a^2}\right)_+(s-z)}, \\
  G(z) & = \frac{\sqrt{z^2-a^2}}{\pi }\int_{-a}^a\frac{\widetilde{V}'(s)\ds}{\left(\sqrt{s^2-a^2}\right)_+(s-z)}.
\end{align*}

To find the support of $\widetilde{\rho}$, we need to consider the asymptotics of $G(z)$ as $z\to\infty$.
We expand the above expression into series:
\begin{equation}
  \label{eq:G_expansion}
  G(z) = \frac{z+\cdots}{\pi} \left(-\frac{1}{z}\right) \int_{-a}^a\frac{\widetilde{V}'(s)}{\left(\sqrt{s^2-a^2}\right)_+}\left(1+\frac{s}{z}+\cdots\right) \ds.
\end{equation}
Consider the first term in the series.
For $G(z) \to 0$ as $z\to\infty$ we need to have
\[
  \int_{-a}^a\frac{\widetilde{V}'(s)}{\left(\sqrt{s^2-a^2}\right)_+} \ds = 0,
\]
which is automatically satisfied since $\widetilde{V}(x)$ is an even function and $\widetilde{V}'(s)$ is an odd function.
At the same time 
\[
  G(z)=-i\int\frac{\widetilde{\rho}(y) \dy}{y-z}\simeq\frac{i}{z}\int\widetilde{\rho}(y)\dy+ \bigO\left(\frac{1}{z^2}\right),
\]
and comparing it to the second term in the series~\eqref{eq:G_expansion} we arrive at
\begin{equation}
\label{eq:series_second_term}
  -\frac{1}{\pi} \int_{-a}^{a}\frac{\widetilde{V}'(s)s}{\left(\sqrt{s^2-a^2}\right)_+z}\ds = \frac{i}{z}.
\end{equation}
Taking the derivative of the potential $\widetilde{V}(x)$ and substituting it into Equation~\eqref{eq:series_second_term}, we get
\begin{equation}
\label{eq:ln_integral_equals_i}
 \frac{1}{2}\int_{-a}^{a}\frac{s}{\sqrt{s^2-a^2}}\cdot\frac{-1}{\pi}\ln \absval{\frac{s+(c+1)/2}{s-(c+1)/2}} \ds = i.
\end{equation}
By taking a derivative, we can check that
\begin{align*}
  \int\frac{s}{\sqrt{s^2-a^2}} & \ln\absval{\frac{s+(c+1)/2}{s-(c+1)/2}} \ds
  \\ & = \frac{1}{2} \left(\left(2 \sqrt{s^2-a^2}-\sqrt{(c+1)^2-4 a^2}\right) \log (c+1-2 s) \right.
  \\ & \hspace{30pt} \left. + \left(\sqrt{(c+1)^2-4 a^2}-2 \sqrt{s^2-a^2}\right) \log (c+1+2s) \right.
  \\ & \hspace{30pt} \left. - \sqrt{(c+1)^2-4 a^2} \log \left(\sqrt{(c+1)^2-4 a^2} \sqrt{s^2-a^2}-2 a^2-(c+1) s\right) \right.
  \\ & \hspace{30pt} \left. + \sqrt{(c+1)^2-4 a^2} \log \left(\sqrt{(c+1)^2-4 a^2} \sqrt{s^2-a^2}-2 a^2+(c+1) s\right) \right.
  \\ & \hspace{30pt} \left. -2 (c+1) \log \left(\sqrt{s^2-a^2}+s\right)\right)+\mathrm{const}.
\end{align*}
Substituting the integration limits we obtain the equation
\begin{equation}
\label{eq:relating_a_to_c}
  \frac{c+1}{2}\left[1-\sqrt{1-\left(\frac{2a}{c+1}\right)^{2}}\right] = 1,
\end{equation}
which can be solved for $c\geq 1$, and we obtain
\begin{equation}
\label{eq:a_sqrt_c}
  a=\sqrt{c}.
\end{equation}
We see that indeed $a<\frac{c+1}{2}$ for $c>1$ and the solution  $\widetilde{\rho}$ of the variational problem~\eqref{eq:J_functional} is given by the formula
\[
 \widetilde{\rho}(x)=\frac{1}{\pi}\Re[ G_+(x)] = \frac{1}{\pi^{2}}\Re\left[\sqrt{x^2-c}\int_{-\sqrt{c}}^{\sqrt{c}}\frac{\frac{1}{2}\left(\ln\left(\frac{c+1}{2}+s\right)-\ln\left(\frac{c+1}{2}-s\right)\right)}{\left(\sqrt{s^2-c}\right)_+ (s-x)} \ds \right].
\]

To compute the integral, we combine the logarithms the same way as we did in Equation~\eqref{eq:ln_integral_equals_i}:
\[
    \frac{1}{\pi^{2}}\int_{-\sqrt{c}}^{\sqrt{c}}\frac{\left(\ln\left(\frac{c+1}{2}+s\right)-\ln\left(\frac{c+1}{2}-s\right)\right)}{\sqrt{c-s^{2}}(s-x)} \ds =
  \frac{1}{\pi}\int_{-\sqrt{c}}^{\sqrt{c}}\frac{1}{\sqrt{c-s^2}(s-x)}\cdot\frac{1}{\pi}
  \ln\absval{\frac{s-\frac{c+1}{2}}{s+\frac{c+1}{2}}} \ds.
\]
Notice that the function
\[
\frac{1}{\pi}\ln\absval{\frac{s-(c+1)/2}{s+(c+1)/2}}
\]
is the Hilbert transform of the 
indicator function ${\bf 1}_{[-(c+1)/2,(c+1)/2]}$.
By using the following well-known relation (see, for example,~\cite{giang2010finite})
\[
  \int_{-\infty}^{\infty}f(s)\widetilde{g}(s)\ds=-\int_{-\infty}^{\infty}\widetilde{f}(s)g(s)\ds,
\]
where $\widetilde{f}$ is a Hilbert transform of $f$ and $f\in L^{p}(\RR)$, $g\in L^{q}(\RR)$ with $\frac{1}{p}+\frac{1}{q}=1$, and taking $g = {\bf 1}_{[-(c+1)/2,(c+1)/2]}$, we obtain
\[
  \frac{1}{\pi}\int_{-\infty}^{\infty}f(s)\ln\left|\frac{s-(c+1)/2}{s+(c+1)/2}\right|\ds=-\int_{-(c+1)/2}^{(c+1)/2}\widetilde{f}(s)\ds.
\]

Thus, we need to compute the Hilbert transform for the function
\[
f(y) = \begin{cases}
\displaystyle \frac{1}{\pi}\frac{1}{(y-x) \sqrt{y^{2}-c}} & \text{if } y \in[-\sqrt{c},\sqrt{c}], \\
0 & \text{otherwise.}
\end{cases}
\]
and then integrate it from $-(c+1)/2$ to $(c+1)/2$.
In order to compute the integral in the Hilbert transform $\widetilde{f}$, we take the change of variables
\begin{equation}
\label{eq:hilbert_intergal_cov}
y=\sqrt{c}\frac{c-t^{2}}{c+t^{2}},
\qquad
\frac{\dy}{\sqrt{c-y^{2}}}=-\frac{2\sqrt{c}\dt}{c+t^{2}},
\end{equation}
and hence, we obtain
\[
  \widetilde{f}(z) = \frac{1}{\pi^{2}}\int_{-\sqrt{c}}^{\sqrt{c}}\frac{ds}{\sqrt{c-s^{2}}(s-x)(s-z)} = \frac{1}{\pi}
  \frac{\displaystyle \left(\frac{1}{\sqrt{z^2-c}}-\frac{1}{\sqrt{x^2-c}}\right)}{x-z}.
\]
At last, we compute the integral
\[
\widetilde{\rho}(x) = \frac{1}{\pi}\Re\left[\sqrt{c-x^{2}}\int_{-(c+1)/2}^{(c+1)/2}\frac{1}{2}
    \left(\frac{1}{(x-z)\sqrt{z^2-c}}-\frac{1}{(x-z)\sqrt{x^2-c}}\right)
  \dz\right].
\]
Here again we can use the substitution~\eqref{eq:hilbert_intergal_cov} or find the indefinite integral in a reference table of integrals such as~\cite{gradshteyn2014table} and obtain
\begin{equation}
\label{eq:rho_tilde_answer}
\begin{aligned}
\widetilde{\rho}(x) & = -\frac{1}{2\pi}
  \left[\Im\left(\log \left(\sqrt{(c-1)^{2}} \sqrt{x^2- c}-(c+1)x+2c\right) \right.\right.
   \\ & \left.\left. \hspace{60pt} +\log \left(\sqrt{(c-1)^{2}} \sqrt{x^2-c}+(c+1) x+2c\right)\right)-\pi \right].
\end{aligned}
\end{equation}
This answer~\eqref{eq:rho_tilde_answer} is easily rewritten in terms of the inverse trigonometric functions for $c\geq 1$ and $\abs{x}\leq \sqrt{c}$ as
\begin{equation}
\label{eq:41}
\widetilde{\rho}(x)=\frac{1}{2\pi}\left[
    \arctan \left(\frac{-(c+1)x+2c}{(c-1) \sqrt{c-x^2}}\right)+
    \arctan\left(\frac{ (c+1)x+2c}{(c-1) \sqrt{c-x^2}}\right)\right].
\end{equation}
The typical graph of the function $\widetilde{\rho}(x)$ for $c>1$ is presented in
Figure~\ref{fig:rho}. The limit shape of the diagram is then obtained
using the formula~\eqref{eq:limit_density}. An example for $c=9$ and a diagram
with $n=10, k=90$ is presented in Figure
\ref{fig:n-10-k-90-diagram-and-limit-shape}.

\begin{figure}
  \includegraphics[width=12cm]{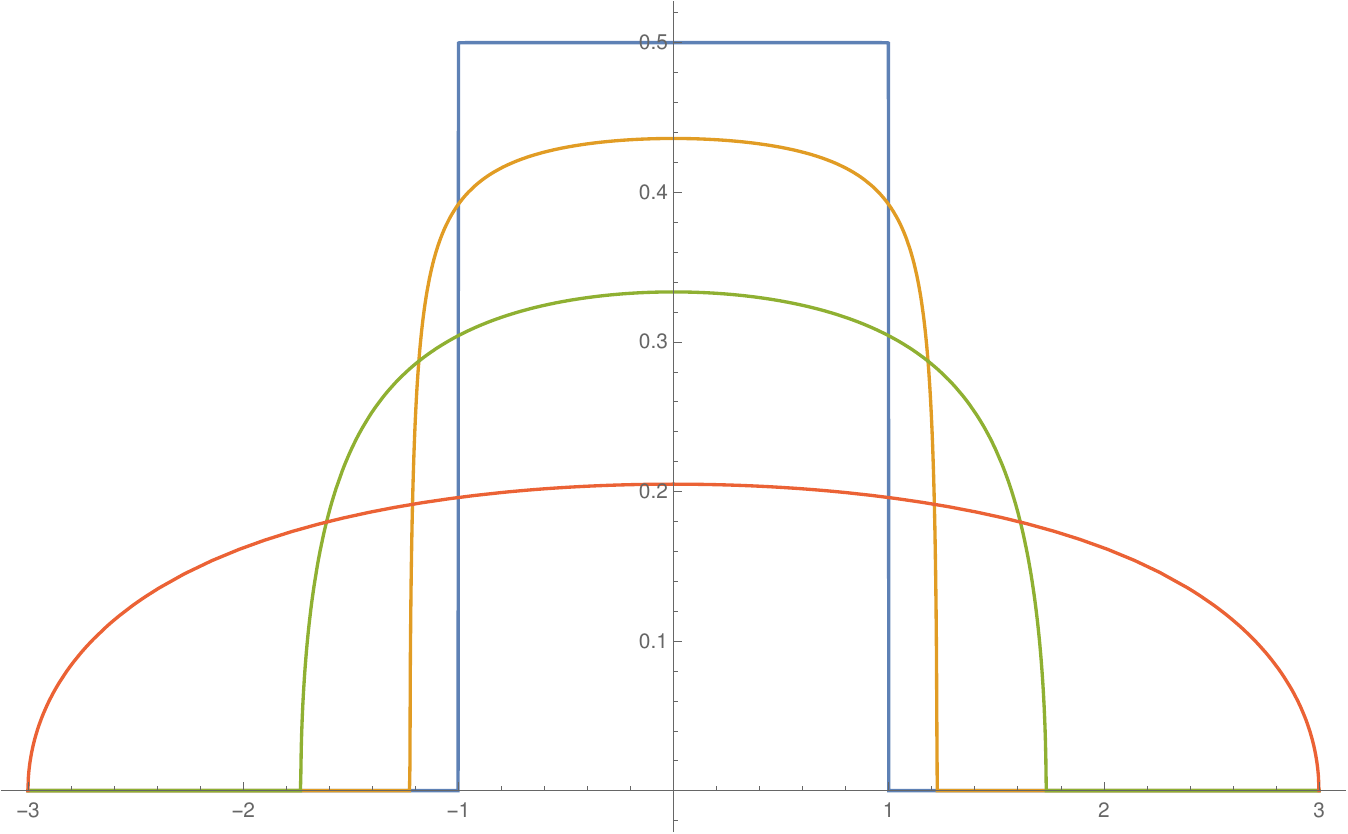}
\caption{The function $\widetilde{\rho}(x)$ for $c=1$ (blue), $c=\frac{3}{2}$ (orange), $c=3$ (green) and $c=9$ (red).}
\label{fig:rho}

\end{figure}
\begin{figure}[t]
  \includegraphics[width=10cm]{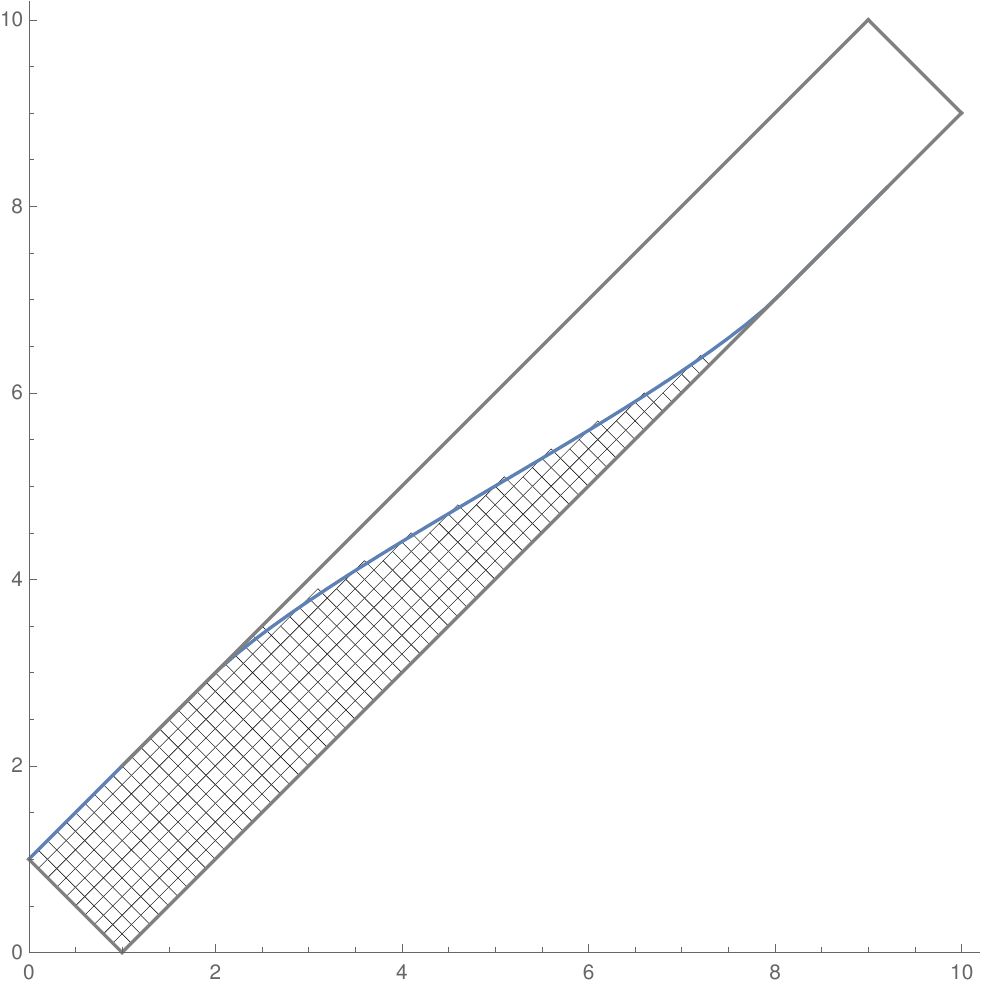}
\caption{The most probable $(\GL_{n},\GL_{k})$ diagram  for $n=10, k=90$ and the limit shape for $c=9$.}
\label{fig:n-10-k-90-diagram-and-limit-shape}
\end{figure}

For $c<1$ it is no longer possible to find the minimizer such that $\widetilde{\rho}(x) < 1$ for all $x$.
The potential $\widetilde{V}(x)$ becomes weaker as $c$ tends to $1$, and when $c=1$ we have a ``phase transition.''
In this case the particles are not confined strictly inside the interval $[-1,1]$ anymore, and instead we have a constant density $\widetilde{\rho}(x)\equiv 1/2$ on the whole interval.
We have an obvious restriction $\rho(x) \leq 1$, therefore for $c < 1$ it is reasonable to expect
\[
\widetilde{\rho}(x)=1 - \rho_{1}(x),  
\]
where $\supp \rho_1 \subset \left[-\frac{c+1}{2}, \frac{c+1}{2}\right]$.
Note that $\widetilde{\rho}(x)\equiv\frac{1}{2}$ is a constant solution to Equation~\eqref{eq:electrostatic_equilibrium} for $a=1$.
Then
\[
  \int_{-(c+1)/2}^{(c+1)/2}\frac{\widetilde{\rho}(y)\dy}{x-y}=\int_{-(c+1)/2}^{(c+1)/2}\frac{(1-\rho_{1}(y))\dy}{x-y}=-2\widetilde{V}'(x)+\int_{-(c+1)/2}^{(c+1)/2}\frac{\rho_{1}(y)\dy}{x-y}=-\widetilde{V}'(x),
\]
and the function $\rho_{1}(x)$ should also be a solution of~\eqref{eq:electrostatic_equilibrium}, but with a different normalization condition
\[
  \int_{-(c+1)/2}^{(c+1)/2}\rho_{1}(x)\dx=-\int_{-(c+1)/2}^{(c+1)/2}\widetilde{\rho}(x)\dx+\int_{-(c+1)/2}^{(c+1)/2}1\dx=c.
\]
The integral representation of $\rho_{1}(x)$ is obtained 
in the same way as for the case $c>1$, but Equation~\eqref{eq:relating_a_to_c} becomes
\[
  \frac{c+1}{2}\left[1-\sqrt{1-\left(\frac{2a}{c+1}\right)^{2}}\right]=c,
\]
and we again get $a=\sqrt{c}$. The function $\rho_{1}$ is derived in exactly the same way as in the case $c>1$ and the final formula is
\begin{equation}
  \label{eq:47}
  \widetilde{\rho}(x)=1-\frac{1}{2\pi}\left[
    \arctan \left(\frac{-(c+1)x+2c}{(1-c) \sqrt{c-x^2}}\right)+
    \arctan\left(\frac{ (c+1)x+2c}{(1-c) \sqrt{c-x^2}}\right)\right],
\end{equation}
which leads to the formula~\eqref{eq:limit-shape-for-f}.
This formula can be also obtained by interchanging $n$ and $k$ for $\GL_{n}\times \GL_{k}$ case.\footnote{This does not hold for the $\SO$ and $\Sp$ cases we consider in the sequel.}
\end{proof}

The most probable diagram (in our measure) for $n=20$, $k=10$ and the corresponding limit shape for $c=0.5$ as well as the most probable diagram for $n=10, k=20$ and the limit shape for $c=2$ are presented in Figure~\ref{fig:n-20-k-10-diagram-and-limit-shape}.
We obtained these most probable diagrams, as well as those below, by using (a discrete) gradient descent.

\begin{figure}[t]
  \centering
  \includegraphics[width=7cm]{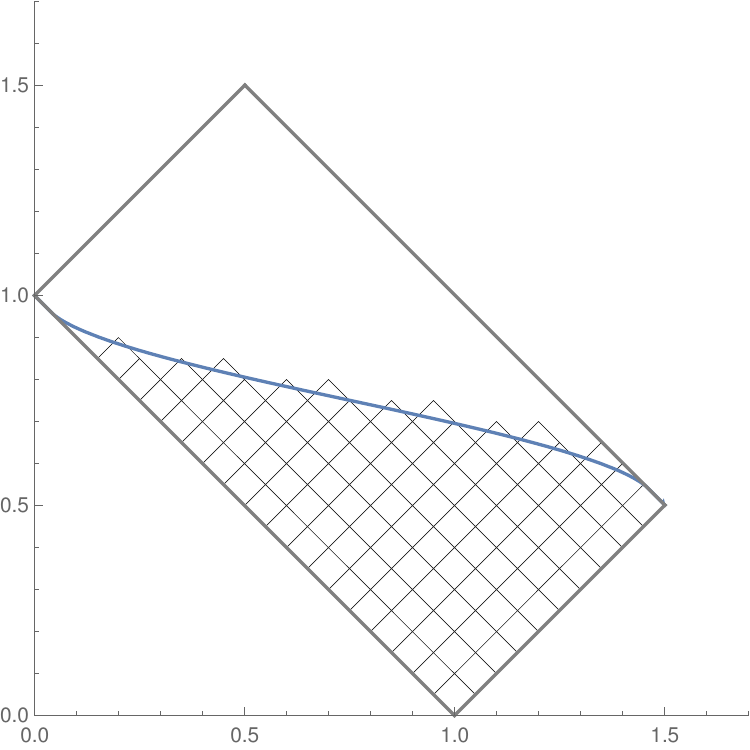}
  \qquad
  \includegraphics[width=7cm]{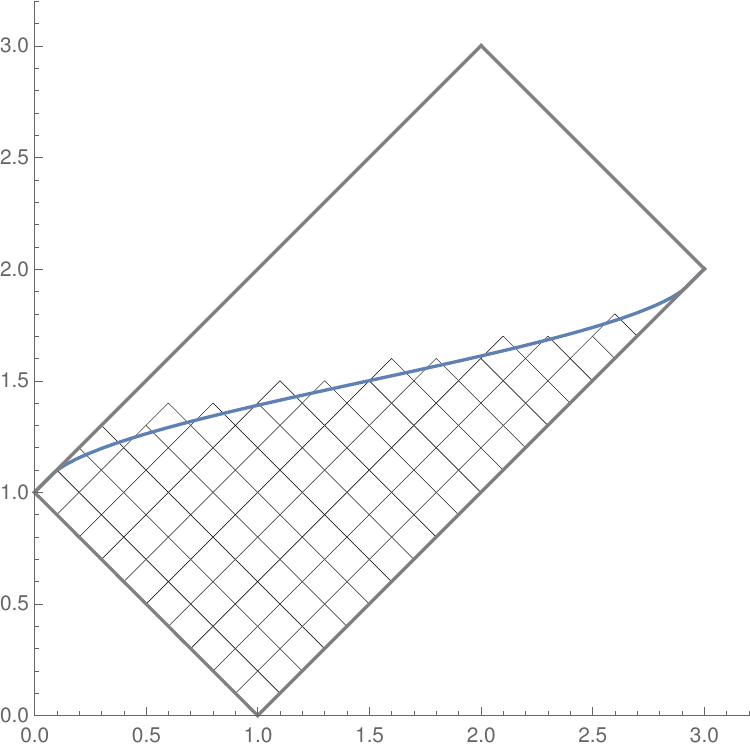}
  \caption{One of the most probable $(\GL_{n},\GL_{k})$ diagrams for $n=20, k=10$ and the limit
    shape for $c=0.5$ on the left and one of the most probable
    diagrams for $n=10, k=20$ and the limit shape for $c=2$ on the right.}
  \label{fig:n-20-k-10-diagram-and-limit-shape}
\end{figure}

\begin{proof}[Proof of Theorem \ref{thm:limit_shape_gl} for
  $\GL_{n}\times \GL_{k}$]
  The proof of the convergence to the limit shape is completely analogous to the proof for $\SO_{2n+1}$ presented in~\cite{NNP20}.

  The proof proceeds as follows. First the functional $J$ is written in terms of the upper boundary $f_{n}$ as $J[f_{n}] = Q[f_{n}]+C$, where $Q$ is quadratic in the derivative $f_{n}'$:
\[
    J[f_{n}]=Q[f_{n}]+C,
    \qquad
    Q[f_{n}]=\frac{1}{2}\int_{0}^{(c+1)/2}\int_{0}^{(c+1)/2} f_{n}'(x)
    f_{n}'(y) \ln \abs{x-y}^{-1}\dx\;\dy.
\]
Since our definition of $Q$ is similar to a definition in the book~\cite{romik2015surprising}, we can use~\cite[Prop.~1.15]{romik2015surprising} and see that $Q$ is positive-definite on compactly-supported Lipschitz functions.
Then for a compactly supported Lipschitz function $f \colon \RR \to [0,\infty)$, the quadratic part $Q$ of the functional $J$ is used to introduce a norm
\[
  \Abs{f}_{Q} := Q[f]^{1/2}.
\]
Consider a space of $1$-Lipschitz functions $f_{1}$ and $f_{2}$ such that the derivative $f_{1,2}'(x) = \mathrm{sgn}\; x$ for $\abs{x} > \frac{c+1}{2}$.
Then the difference $f_{1}-f_{2}$ is a compactly supported Lipschitz function and we can use its norm to introduce a metric
\begin{equation}
\label{eq:Q-distance}
  d_{Q}(f_{1},f_{2}) := \Abs{f_{1}-f_{2}}_{Q}.
\end{equation}
We can use~\cite[Lemma 1.21]{romik2015surprising} to obtain an estimate on the supremum norm for a Lipschitz function $f$ with a compact support:
\begin{equation}
  \label{eq:supremum-norm}
  \Abs{f}_{\infty} = \sup_{x} \abs{f(x)} \leq C_{1} Q[f]^{1/4},
\end{equation}
where $C_{1}$ is some constant.

Then we estimate the probability of the diagram that differs from the limit shape by $\varepsilon$.
For a highest weight $\lambda$ with the boundary of rotated Young diagram given by a function $f_{n}(x)$ such that $d(f_{n},f)=\varepsilon$, the probability is bounded by
\[
    \mu_{n,k}(\lambda)\leq C_{2} e^{-n^{2}\varepsilon^{2}+\mathcal{O}(n\ln n)}.
\]
After that we need only to estimate total number of diagrams in the $n\times k$ box as at most $\tilde{C}e^{cn}$ in order to have the convergence in probability in the metric $d_{Q}$ to the limiting shape given by the formula~\eqref{eq:limit-shape-gl}.
This estimate is easily obtained from the Hardy--Ramanujan formula, since total number of boxes in the diagram is not greater than $cn^{2}$.
That is, for all $\varepsilon>0$ we have
  \begin{equation}
    \label{eq:67}
    \mathbb{P}\left(\Abs{f_{n}-f}_{Q}>\varepsilon\right)\xrightarrow[n\to\infty]{} 0,
  \end{equation}
since the probability of each highest weight $\lambda$ with a rotated Young diagram with boundary $f_{n}$ such that $\Abs{f_{n}-f}_{Q}>\varepsilon$ is bounded by $e^{-n^{2}\varepsilon^{2}+\mathcal{O}(n\ln n)}$.

  At last we apply the relation~\eqref{eq:supremum-norm} to complete the
  proof of the theorem. 
\end{proof}
\subsection{Limit shape for \texorpdfstring{$(\SO_{2l+1},\Pin_{2k})$}{(SO(2l+1), Pin(2k))} skew Howe duality}
\label{sec:limit-shape-series-B}

Now we will assume that $n=2l+1$ is odd.
Then, as was discussed in Section~\ref{sec:skew_howe},  it has a multiplicity-free decomposition into the direct sum of $\SO_{2l+1}\times \Pin_{2k}$ irreducible representations that are parametrized by generalized Young diagrams given in Section~\ref{sec:mult_type_BC}.
Regarding this decomposition as a $\SO_{2l+1}$ representation, we obtain the formula for the multiplicities of the irreducible representations in the tensor power decomposition of the exterior algebra of the defining representation $\bigwedge V(\fw_{1})$:
\[
  \left(\bigwedge V(\fw_{1}) \right)^{\otimes k} = \bigoplus_{\lambda} \widetilde{M}^k(\lambda) V(\lambda).
\]
Since $\bigwedge V(\fw_{1}) \iso V(\Lambda_{l})^{\otimes 2}\otimes \bigwedge V(0)$, it is equivalent to compute the multiplicity of $V(\lambda)$ in the tensor power decomposition
\[
V(\fw_{l})^{\otimes 2k} = \bigoplus_{\lambda} M^{2k}(\lambda) V(\lambda)
\]
since the multiplicities are related by $M^{2k}(\lambda) = 2^{-k} \widetilde{M}^k(\lambda)$.
Thus we recover the multiplicity formula obtained in~\cite{kulish2012tensor}:
\[
  M^{2k}(\lambda)=
   \prod_{m=1}^{l}\frac{\left(2k+2m-2\right) !}{2^{2m-2}\left( \frac{2k+a_{m}+2l-1}{2}\right) !\left( \frac{2k-a_{m}+2l-1}{2}\right) !}
\times \prod_{s=1}^{l}a_{s}
\times \prod_{i < j} \left( a_{i}^{2}-a_{j}^{2}\right),
\]
where the coordinates $\{a_{i}\}$ we are related to the values $\lambda = \sum_{i=1}^{n}\ell_{i} \fw_{i}$ by the formula
\begin{equation}
\label{eq:a_coords_so2np1}
  a_{i}=2\sum_{j=i}^{l-1}\ell_{j}+\ell_{l}+2(l-i)+1 = 2(\lambda_i + l - i) + 1
\end{equation}
and correspond to the rotated Young diagram, as demonstrated in Figure~\ref{fig:young-rotated-ai-geom-meaning}.
We will use the coordinates~\eqref{eq:a_coords_so2np1} for the remainder of this section.

The limit shape for this case was completely derived and presented with all the proofs in~\cite{NNP20}.
Here we will present the limit shape in a special normalization so that the connection between limit shapes for the diagrams of $\SO_{2l+1}$ and $\GL_{n}$ becomes apparent.

\begin{figure}[t]
  \[
  \begin{tikzpicture}[baseline=0,scale=0.7]
  \draw[->, thick] (-5,0) -- (-5,8);
  \draw[->, thick] (-5,0) -- (8,0);
  \foreach \x in {-10,...,12}
    \draw[-] (\x/2,-0.15) -- ++(0,.3);
  \foreach \x in {0,5,10}
    \draw[-,thick] (\x-5,-0.25) -- ++(0,.5);
    \draw(-5,-0.15) node[anchor=north] {$0$};
    \draw(0,-0.25) node[anchor=north] {$10$};
    \draw(5,-0.15) node[anchor=north] {$20$};        
  \foreach \i/\x in {0/5, .5/5, 1/5, 2/4, 3/2, 4/2, 5/1, 6/1}
    \draw[-] (\i,\i) -- ++(-\x,\x);
  \draw[-] (0,0) -- ++(6,6);
  \foreach \x [count=\i from 1] in {6, 4, 2, 2, 1} {
    \draw[-] (-\i,\i) -- ++(\x,\x);
    \draw[dashed, blue] (\x-\i+.5, \x+\i-0.5) -- (\x-\i+0.5, 0);
    \draw[blue, fill=blue] (\x-\i+.5, \x+\i-0.5) circle (0.08);
    \draw[blue, fill=blue] (\x-\i+.5, 0) circle (0.07) node[anchor=north, color=black] {$a_{\i}$};
    }
  \draw[-] (-5,5) -- ++(1,1);
  \end{tikzpicture}
  \]
  \caption{Rotated generalized Young diagram for $\SO_{2l+1}$ and the geometrical meaning of
    the coordinates $\{a_i\}_{i=1}^{l}$.}
  \label{fig:young-rotated-ai-geom-meaning}
\end{figure}
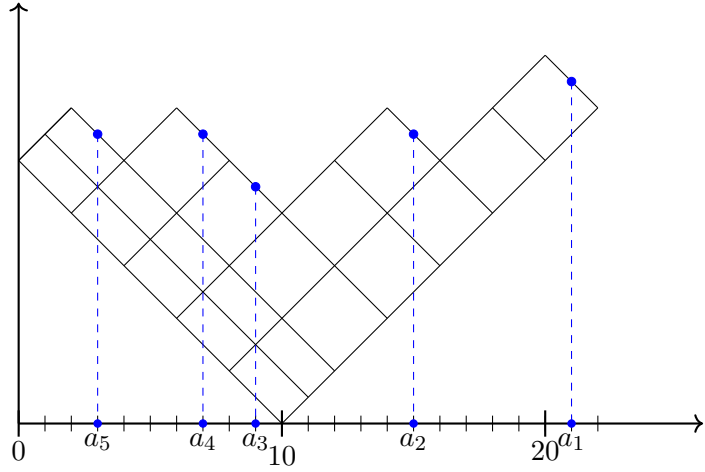

Using the Weyl dimension formula, the probability measure is written as
\begin{equation}
  \label{eq:so2p1_coord_measure}
\mu_{n,k}(\lambda) 
 = \frac{2^{-l^2+2l-lk}l!}{{(2l)!(2l-2)!\dots 2!}}\times
   \prod_{m=1}^{l}\frac{\left(2k+2m-2\right) !}{2^{2m-2}\left(\frac{2k+a_{m}+2l-1}{2}\right) !\left(\frac{2k-a_{m}+2l-1}{2}\right) !}\times\prod_{s=1}^{l}a_{s}^{2} \times \prod_{ i<j} \left( a_{i}^{2}-a_{j}^{2}\right)^{2}.
\end{equation}
Now, we consider the limit $n,k\to\infty$ such that $\frac{2k}{n} = c+\mathcal{O}\left(\frac{1}{n}\right)$.
Here the notation is different from what was used in the paper~\cite{NNP20}.
The coordinates $\{a_{i}\}$ are taking integer values in the domain $[0,n(c+1)]$.

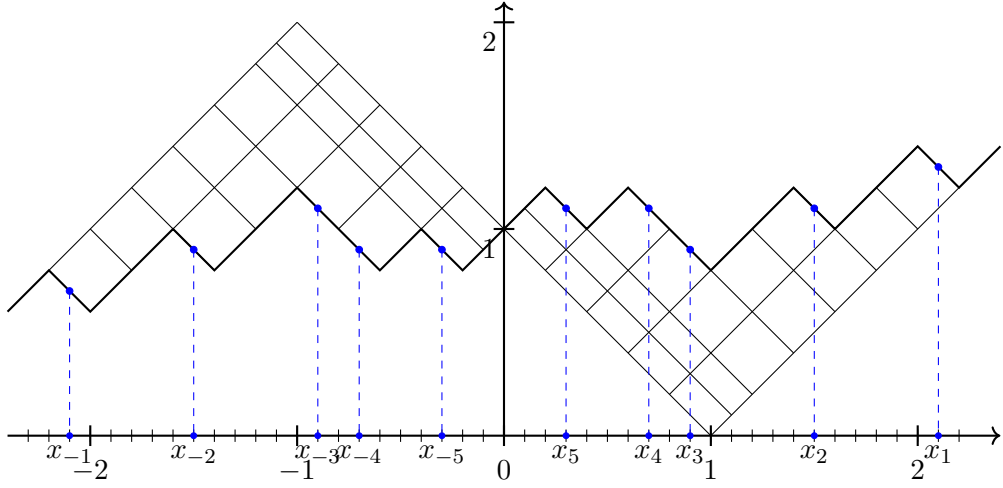
\begin{figure}[t]
  \[
  \begin{tikzpicture}[baseline=0,scale=0.55]
  \draw[->, thick] (-5,0) -- (-5,10.5);
  \draw[->, thick] (-17,0) -- (7,0);
  \draw[-, thick] (-17,3) -- (-16,4) -- (-15,3) -- (-13,5) -- (-12,4) -- (-10,6) -- (-8,4) -- (-7,5) -- (-6,4) -- (-4,6) -- (-3,5) -- (-2,6) -- (0,4) -- (2,6) -- (3,5) -- (5,7) -- (6,6) -- (7,7);
  \foreach \x in {-33,...,12}
    \draw[-] (\x/2,-0.15) -- ++(0,.3);
  \foreach \x in {-2,-1,0,1,2} {
    \draw[-,thick] (\x*10/2-5,-0.25) -- ++(0,.5);
    \draw(\x*10/2-5,-0.35) node[anchor=north] {$\x$};
  }
  \foreach \x in {1,2} {  
    \draw[-,thick] (-5-0.25,\x*10/2) -- ++(.5,0);
    \draw(-5-0.35,\x*10/2) node[anchor=north] {$\x$};
  }
  \foreach \i/\x in {0/5, .5/5, 1/5, 2/4, 3/2, 4/2, 5/1, 6/1} {
    \draw[-] (\i,\i) -- ++(-\x,\x);
    \draw[-] (-10-\i,10-\i) -- ++(\x,-\x);
  }
  \draw[-] (0,0) -- ++(6,6);
  \foreach \x [count=\i from 1] in {6, 4, 2, 2, 1} {
    \draw[-] (-\i,\i) -- ++(\x,\x);
    \draw[dashed, blue] (\x-\i+.5, \x+\i-0.5) -- (\x-\i+0.5, 0);
    \draw[blue, fill=blue] (\x-\i+.5, \x+\i-0.5) circle (0.08);
    \draw[blue, fill=blue] (\x-\i+.5, 0) circle (0.07) node[anchor=north, color=black] {$x_{\i}$};
    \draw[-] (-11+\i,11-\i) -- ++(-\x,-\x);
    \draw[dashed, blue] ({-10-(\x-\i+.5)}, {10-(\x+\i-0.5)}) -- ({-10-(\x-\i+.5)}, 0);
    \draw[blue, fill=blue] ({-10-(\x-\i+.5)}, {10-(\x+\i-0.5)}) circle (0.08);
    \draw[blue, fill=blue] ({-10-(\x-\i+.5)}, 0) circle (0.07) node[anchor=north, color=black] {$x_{-\i}$};
    }
  \end{tikzpicture}
  \]
  \caption{Rotated and scaled diagram for $\SO_{2l+1}$ with $l=5$ and its continuation to negative values of coordinate $x$.
    The function $f_{l}(x)$ is shown in solid black, and the points $x_{i}=\frac{a_{i}}{2l}$ are the midpoints of intervals, where $f_{l}'(x)=-1$. }
  \label{fig:young-rotated-and-continued}
\end{figure}

To bring the expression~\eqref{eq:so2p1_coord_measure} to the form~\eqref{eq:measure_factorized}, we denote by $a_{2l+1-i},\;i>0,\;i<l$ the ``mirror image'' of $a_{i}$:
\begin{equation}
  \label{eq:reflected_a_coords}
  a_{2l+1-i}\equiv -a_{i}.
\end{equation}
These points correspond to a continued diagram (that is, with the diagram rotated 180 degrees around its intersection along the $x=0$ line), as illustrated in Figure~\ref{fig:young-rotated-and-continued}.
Then we use Stirling formula to rewrite the measure~\eqref{eq:so2p1_coord_measure} in the form:
\[
  \mu_{n,k}(\left\{a_{i}\right\}_{i=1}^{2l})=\frac{1}{Z_{l}}\prod_{\substack{i< j\\i,j=1}}^{2l}\abs{a_{i}-a_{j}}
  \cdot \prod_{s=1}^{2l}\exp\left[-(4l)V\left(\frac{a_{s}}{4l}\right)-e_{l}\left(a_{s}\right)\right],
\]
where
\begin{align}
  V(u) & = \frac{1}{4} \left[\left(\frac{c+1}{2}+u\right)\ln\left(\frac{c+1}{2}+u\right)+\left(\frac{c+1}{2}-u\right)\ln\left(\frac{c+1}{2}-u\right)\right], \label{eq:so2np1_potential_explicit}
  \\ e_{l}(u) & = \frac{1}{4}\ln\left(\bigl( (c+2)l \bigr)^{2}-u^{2}\right)+\frac{1}{2}\ln \abs{u} + \bigO\left(\frac{1}{l}\right),
\end{align}
and $Z_{l}$ does not depend on $a_{l}$ and the additional conditions~\eqref{eq:reflected_a_coords} are satisfied.

Introducing the coordinates $\left\{x_{i}=\frac{a_{i}}{4l}\right\}_{i=1}^{2l}$, we arrive at the same variational problem~\eqref{eq:J_functional}.
Yet now we are interested only in values of $\widetilde{\rho}(x)$ for $x>0$.
The solution is given by the formula~\eqref{eq:41} for $c>1$ and by the formula~\eqref{eq:47} for $c < 1$.
This coincidence of density $\rho$ with $\GL_{n}$ case leads to a peculiar effect for the limit shapes of Young diagrams: for large $n,k$ typical Young diagram of $\SO_{2l+1}$  looks as a part of a typical diagram for $\GL_{n}$.
This is demonstrated in Figure~\ref{fig:soo-over-gl}.

\begin{figure}[t]
\centering
\begin{minipage}{.49\linewidth}
  \centering
\[
\iflimitshapes
\begin{tikzpicture}[baseline=10, scale=0.05]

    \newcount\n
    \newcount\k
  \n=41;
  \k=101;

\foreach \y [count=\i from 0] in  {79, 73, 67, 63, 59, 55, 51, 47, 43, 39, 35, 31, 27, 25, 21, 17, 13, 11, 7, 3} {
  \draw[blue,fill=blue] (\n+\k/2-\i,\k/2+\i) -- (\n+\k/2+\y/2-\i,\k/2+\i+\y/2)--(\n+\k/2-1+\y/2-\i,\k/2+\i+1+\y/2)--
  (\n+\k/2-1-\i,\k/2+\i+1);
}

\draw[gray!40, very thin] (\n,0) -- (\n+\k,\k) -- (\k,\k+\n) -- (0,\n);
\draw[->, gray!40, thin] (0,0) -- (\n+\k,0);
\draw[->, gray!40, thin] (0,0) -- (0,\n+\k);
  
\foreach \y [count=\i from 0] in  {89, 86, 83, 81, 79, 76, 74, 72, 70, 68, 67, 65, 63, 61, 59, 57, 55, 54, 52, 50, 48, 46, 44, 42, 40, 38, 36, 34, 32, 30, 28, 26, 24, 22, 20, 18, 16, 14, 12, 10, 8 } {
  \draw[black, thick] (\n-\i,\i) -- (\n+\y-\i,\i+\y)--(\n-1+\y-\i,\i+1+\y)--
  (\n-1-\i,\i+1);
  \foreach \len in {0,...,\y} {
    \draw[black,thin] (\n-\i+\len,\i+\len) -- (\n-1-\i+\len,\i+1+\len);
  }
}
\draw[black, thick] (\n,0)--(0,\n);
\end{tikzpicture}
\fi
\]
  \caption{One of the most probable Young diagrams for $\GL_{40}$ and $k = 101$ (white background). We superimposed one of the most probable diagrams for
    $\SO_{41}$ and tensor power 50 (shaded blue background).}
  \label{fig:soo-over-gl}
\end{minipage}
\begin{minipage}{.49\linewidth}
  \centering
\[
\iflimitshapes
  \begin{tikzpicture}[baseline=10, scale=0.05]
    \newcount\n
    \newcount\k
  \n=40;
  \k=100;

\foreach \y [count=\i from 0] in  {38, 35, 33, 30, 28, 26, 24, 22, 20, 18, 17, 15, 13, 11, 10, 8, 6, 4, 3, 1 } {
  \draw[fill=blue] (\n+\k/2-\i,\k/2+\i) -- (\n+\k/2+\y-\i,\k/2+\i+\y)--(\n+\k/2-1+\y-\i,\k/2+\i+1+\y)--
  (\n+\k/2-1-\i,\k/2+\i+1);
}
\draw[blue, thick] (\n+\k/2,\k/2)--(\n/2+\k/2,\n/2+\k/2);

\draw[gray!40, very thin] (\n,0) -- (\n+\k,\k) -- (\k,\k+\n) -- (0,\n);
\draw[->, gray!40, thin] (0,0) -- (\n+\k,0);
\draw[->, gray!40, thin] (0,0) -- (0,\n+\k);

  
\foreach \y [count=\i from 0] in  {88, 85, 82, 80, 78, 75, 73, 71, 69, 67, 66, 64, 62, 60, 58, 56, 54, 53, 51, 49, 47, 45, 43, 41, 39, 37, 35, 33, 31, 29, 27, 25, 23, 21, 19, 17, 15, 13, 11, 8 } {
  \draw[black, thick] (\n-\i,\i) -- (\n+\y-\i,\i+\y)--(\n-1+\y-\i,\i+1+\y)--
  (\n-1-\i,\i+1);
  \foreach \len in {0,...,\y} {
    \draw[black,thin] (\n-\i+\len,\i+\len) -- (\n-1-\i+\len,\i+1+\len);
  }
}
\draw[black, thick] (\n,0)--(0,\n);
\end{tikzpicture}
\fi
\]
  \caption{One of the most probable Young diagrams for $\GL_{40}$ and $k=100$ (white background). We superimposed one of the most probable diagrams for $\Sp_{40}$ and tensor power $50$ (shaded blue background).}
  \label{fig:sp-over-gl}
\end{minipage}
\end{figure}

\subsection{Limit shape for \texorpdfstring{$(\Sp_{2l}, \Sp_{2k})$}{(Sp(2l), Sp(2k))} skew Howe duality}
\label{sec:limit-shape-series-C}

This case is very similar to the $\SO_{2l+1}$ case.
We can consider the exterior algebra $\bigwedge\left(\CC^{2l}\otimes \CC^k \right)$ as the $k$-th tensor power of the exterior algebra of the defining representation $V = \bigwedge V(\fw_1)$ since $\dim V = 2^{2l}$.
The multiplicity of $V(\lambda)$ in the decomposition of $V^{\otimes k}$ can be written as a product formula
\[
  M^k(\lambda)=2^{l}\prod_{i=1}^{l}\frac{(2k-1+2i)!}{(k+l+a_{i})!(k+l-a_{i})!}\times\prod_{s=1}^{l}a_{s} \times \prod_{i<j}(a_{i}^{2}-a_{j}^{2}),
\]
where we use the coordinates
\[
a_{i}=\lambda_{i} + l - i + 1.
\]
Using the Weyl dimension formula, we can write
the probability measure as
\begin{equation}
  \label{eq:sp_coord_measure}
  \mu_{n,k}(\left\{a_{i}\right\})=\frac{2^{2l(1-k)}}{\prod_{i<j}(j-i)(2l+2-i-j)} \prod_{i=1}^{l}\frac{(2k-1+2i)!}{(k+l+a_{i})!(k+l-a_{i})!} \prod_{s=1}^{l}a_{s}^{2} \prod_{i<j}(a_{i}^{2}-a_{j}^{2})^{2}.
\end{equation}
We are again interested in the limit $n,k\to\infty$ such that $\frac{2k}{n}=\frac{2k}{2l}=c+\mathcal{O}\left(\frac{1}{n}\right)$.
To bring the expression~\eqref{eq:sp_coord_measure} to the form~\eqref{eq:measure_factorized} we again denote by $a_{2l-i}\equiv -a_{i}$ ($i>0$ or $i<l$) the ``mirror image'' of $a_{i}$.
Then we use Stirling formula to rewrite the measure~\eqref{eq:sp_coord_measure} in the form:
\[
  \mu_{n,k}(\left\{a_{i}\right\}_{i=1}^{2l})=\frac{1}{Z_{l}}\prod_{ \substack{i <  j\\ i,j=1}}^{2l} \abs{a_{i}-a_{j}}
  \times \prod_{s=1}^{2l}\exp\left[-(2l)V\left(\frac{a_{s}}{2l}\right)-e_{l}(a_{s})\right],
\]
where $V(u)$ is the same as in Equation~\eqref{eq:so2np1_potential_explicit},
but the expressions for the correction term $e_{l}(u)$ and the normalization constant $Z_{l}$ are different.

Introducing the coordinates $\left\{x_{i}=\frac{a_{i}}{2l}\right\}_{i=1}^{2l}$, we arrive at the same variational problem~\eqref{eq:J_functional}, and thus we obtain the same limit shape as in the $\SO_{2l+1}$-case.
This limit shape again coincides with a half of non-linear part of limit shape for $GL$ case.
We illustrate this coincidence  with a diagram for $\Sp_{n}$ with $n=20, k=50$  and  $\GL_{n}$-diagram for $n=40,k=50$, presented in Figure~\ref{fig:sp-over-gl}.
Since both cases correspond to $c=5$ and $n$ is large enough, we see a good but not a perfect coincidence of the shapes of the most probable diagrams.


\subsection{Limit shape for \texorpdfstring{$(\Or_{2l}, \SO_{2k})$}{(O(2l), SO(2k))} skew Howe duality}
\label{sec:limit-shape-series-D}

As before, consider the exterior algebra $\bigwedge \left(\CC^{2l}\otimes \CC^{k}\right)$.
Then this space can be seen as $\left(\bigwedge \CC^{2l}\right)^{\otimes k}$, the $k$-th tensor power of the exterior algebra of the first fundamental representation of $\Or_{2l}$.
On the other hand it can be seen as the $2k$-th tensor power of the sum of the last two fundamental representations $\bigl( V(\Lambda_{l-1}) \oplus  V(\Lambda_{l}) \bigr)^{\otimes 2k}$ for $\SO_{2l}$ (recall that as an $\Or_{2l}$-representation, it is irreducible).

The tensor product decomposition coefficient is obtained in Theorem~\ref{thm:mult_type_D}.
Similarly to all previous cases, the coordinates
\[
a_i = 2 \lambda_i + 2 (l-i)
\]
correspond to a rotated Young diagram.
The probability measure is given by the formula
\begin{equation}
  \label{eq:so2n_coord_measure}
  \mu_{n,k}(\{a_{i}\}) = \frac{\displaystyle 2^{-4lk-2l(l-1)} \prod_{i=1}^l (2k+2l-2i)! \prod_{1 \leq i < j \leq l} (a_i^2 - a_j^2)^{2}}{\displaystyle \prod_{i<j}(j-i)(2l-i-j) \prod_{i=1}^l \left( \frac{2k+2l-2-a_i}{2} \right)! \left( \frac{2k+2l-2+a_i}{2} \right)!}
\end{equation}
by applying the Weyl dimension formula.
Similarly to $\SO_{2l+1}$ case we consider the limit $n,k\to\infty$ such that $\frac{2k}{n}=\frac{2k}{2l}=c+\mathcal{O}\left(\frac{1}{n}\right)$. 
Again we bring the expression~\eqref{eq:so2n_coord_measure} to the form~\eqref{eq:measure_factorized} denoting by $a_{2l-i} \equiv -a_{i}$ ($i>0$ or $i<l$) the ``mirror image'' of $a_{i}$.
The only difference here is that there are no columns of the half-width.
Using the Stirling formula to rewrite the measure~\eqref{eq:so2n_coord_measure} in the form
\[
  \mu_{n,k}(\left\{a_{i}\right\}_{i=1}^{2l}) = \frac{1}{Z_{l}}\prod_{\substack{i < j\\ i,j=1}}^{2l} \abs{a_{i}-a_{j}}
  \times \prod_{s=1}^{2l}\exp\left[-(4l) V\left(\frac{a_{s}}{4l}\right)-e_{l}(a_{s})\right],
\]
we again obtain $V(u)$ as in Equation~\eqref{eq:so2np1_potential_explicit},
but the expression for the correction term $e_{l}(u)$ and the normalization constant $Z_{l}$ is different from the $\SO_{2l+1}$ case.
We do not write these expressions here since the limit shape does not depend upon them. 

Introducing the coordinates $\left\{x_{i}=\frac{a_{i}}{4l}\right\}_{i=1}^{2l}$, we again arrive at the same variational problem~\eqref{eq:J_functional}.
Thus we obtain the same limit shape as in the $\SO_{2l+1}$-case.

\subsection{Limit shapes and the insertion algorithms}
\label{sec:limit-shap-insert}

All skew Howe dualities considered above can be seen as tensor power decompositions.
The tensor product decompositions we consider here can all be represented by an insertion algorithm for the corresponding generalized Young diagrams:
\begin{itemize}
\item[$\GL_{n}$:] Schensted insertion (or dual RSK)~\cite{schensted1961longest,knuth1970permutations},
\item[$\Sp_{2l}$:] Berele insertion~\cite{berele1986schensted,Terada93},
\item[$\SO_{2l+1}$:] Benkart--Stroomer insertion~\cite{benkart1991tableaux},
\item[$\SO_{2l}$:] Okada insertion~\cite{okada1993robinson}. 
\end{itemize}
Hence, by pushing forward the uniform distribution on matrices, these insertion algorithms give the same probability measure as~\eqref{eq:prob_measures} on partitions.
Therefore, our results provide the limit shape for these insertion schemes and gives an algorithm to efficiently sample the random diagrams with respect to this measure.

Let us discuss the $(\GL_k, \GL_n)$ case in more detail, where the sampling algorithm proceeds as follows.
First, we generate a uniform random $n\times k$ matrix $M$ with matrix elements taking values $0$ and $1$ with the probability $\frac{1}{2}$.
This matrix $M$ encodes the random basis element of $\bigwedge\left(\CC^{n}\otimes \CC^{k}\right)$
\[
e_M := \bigwedge_{(i,j) : M_{ij} = 1} (e_{i}\otimes e_{j}),
\]
where we go through the pairs $(i,j)$ is some fixed order, such as lexicographic order (the sign does not matter).
Similarly, we consider a sequence of the pairs $(i,j)$ such that $M_{ij} = 1$ ordered lexicographically, which is called a generalized permutation or a biword by Stanley~\cite{ECII}.
We then apply Schensted insertion using the second value $j$ in each pair, where an equal element is bumped downwards in a row the insertion tableau $P$ or added to the end~\cite{knuth1970permutations}.
The new box added to $P$ has the first value $i$ added to the corresponding position in the recording tableau $Q$.
The shapes of $P$ and $Q$ are the same and $Q$ and the transposed insertion tableau $P'$ are semistandard.
Then the shape of $Q$ is conjugate to the shape of $P'$, and tableaux $P'$ and $Q$ encode the basis elements of the decomposition~\eqref{eq:gl_gl_skew_Howe}, as demonstrated in~\cite{knuth1970permutations}.
The shape of the tableau $P'$ is sampled from the distribution~\eqref{eq:GL_prob_measure}.
In Figure~\ref{fig:gl-rsk-sample}, we present a diagram, sampled by the dual RSK algorithm for $n=50, k=150$, as well as the corresponding limit shape for $c = 3 = k/n$.
The limit shape~\eqref{eq:limit-shape-for-f} can be used to deduce the asymptotics of first row length of the random diagram as we obtain $\lambda_{1} \approx \sqrt{kn} + \frac{k-n}{2}$ as $n,k\to\infty$ from~\eqref{eq:a_sqrt_c}.

For the other series, we have analogous sampling algorithms by using the corresponding insertion algorithm.
We also present a diagram in Figure~\ref{fig:gl-rsk-sample} sampled using Benkart--Stroomer insertion for $\SO_{51}$ and $2k = 150$ from the distribution~\eqref{eq:other_prob_measure} since $2k / (2l+1) \approx 3 = c$.
We also obtain the asymptotic of first row length as $\lambda_{1}\approx \sqrt{2kl}$ as $l,k\to\infty, 2k/l \to\mathrm{const}$ from Theorem~\ref{thm:limit_shape_gl}.

\begin{figure}[htb]
  \centering
  \includegraphics[width=12cm]{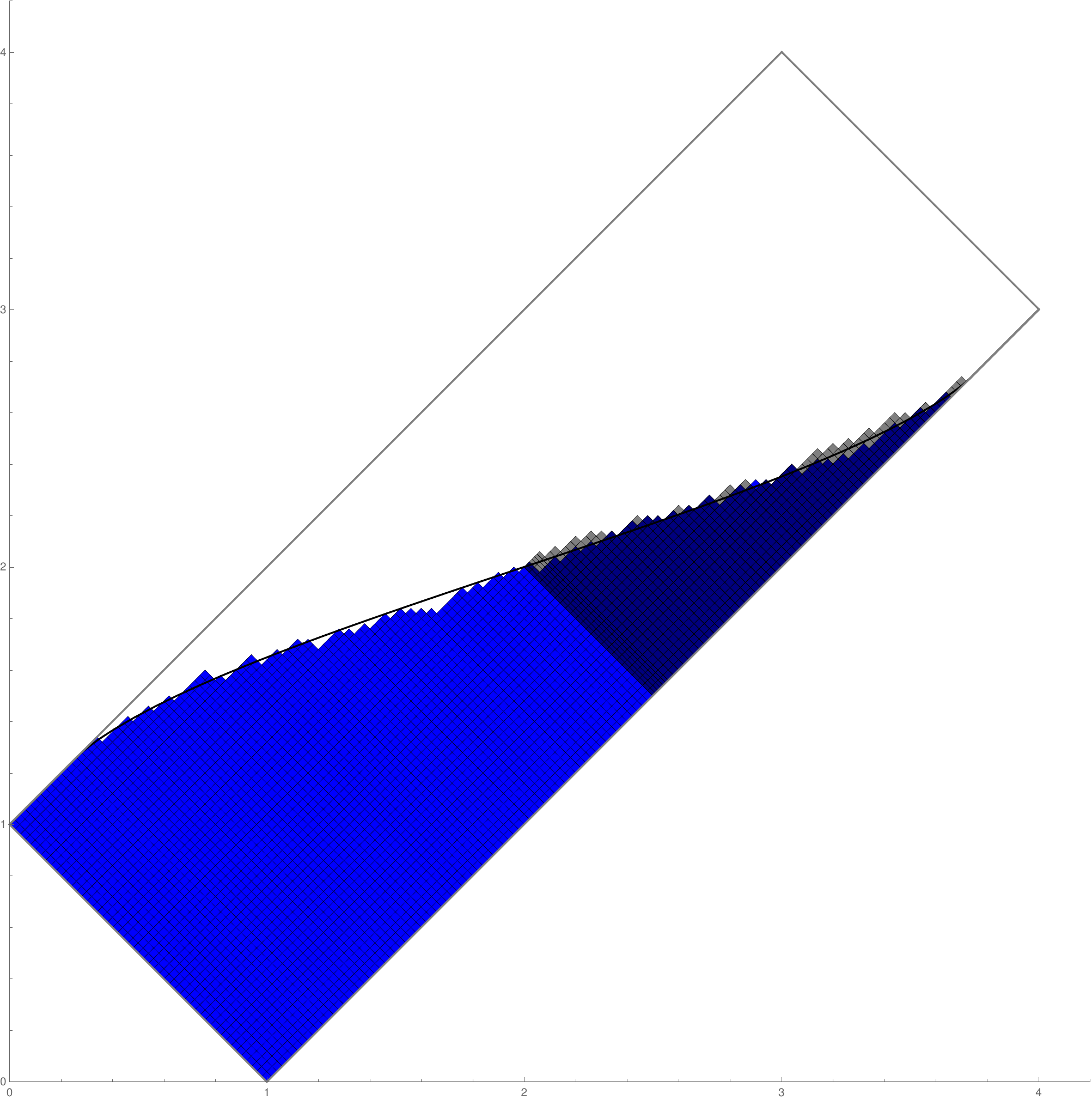}

  \caption{{\it Blue:} Random Young diagram sampled using dual RSK algorithm for $\GL_{50}$ and $k=150$ with the limit shape for $c=3$.
   {\it Shaded:} Random Young diagram sampled using the Benkart--Stroomer insertion algorithm for $\SO_{51}$ and $2k=150$.}
  \label{fig:gl-rsk-sample}
\end{figure}

\subsection{Analytic continuation and orthogonal polynomial ensembles}
\label{sec:limit_shapes_poly_ensembles}

We will discuss the relation between our measures and limit shapes with other results in the literature.
We begin with the $(\GL_n, \GL_k)$ skew Howe duality, discussing its relation with~\cite{pittel2007limit,panova2018skew} and the Krawtchouk ensemble.
We then briefly survey papers of Borodin, Johansson, Okounkov, and Olshanski to connect it with the Meixner ensemble through the ``analytic continuation'' of the parameters $n,k$, and the $z$-measure.
There is another related measure that we discuss called the $zw$-measure, which comes from harmonic analysis of the infinite unitary group $U_{\infty} := \bigcup_{m \geq 1} U_m$.
We then sketch a possible unification of these as a manifestation of the super Howe duality for $(\GL_k, \gl(m|n))$~\cite{howe1989remarks}.
We conclude with showing that other skew Howe dual pairs from our paper are specializations of the type $BC$ $z$-measure introduced by Cuenca~\cite{Cuenca18}.
We discuss the relationship with orthogonal polynomials and possible related super Howe dualities.

\subsubsection{Type A}

The decomposition of the exterior power
\begin{equation}
  \label{eq:exterior_power_decomp}
  \bigwedge\nolimits^{m}\left(\CC^{n}\otimes\CC^{k}\right)=\bigoplus_{\abs{\lambda}=m}V_{\GL_{n}}(\lambda)\otimes
  V_{\GL_{k}}(\lambda') 
\end{equation}
was considered by P.~Sniady and G.~Panova~\cite{panova2018skew}.
They proved the equality
\begin{equation}
  \label{eq:gl-dim-to-sm-dim}
  \frac{\dim V_{\GL_{n}}(\lambda)\dim V_{\GL_{k}}(\lambda')}{\dim \bigwedge^{m}\left(\CC^{n}\otimes\CC^{k}\right)}=
  \frac{f^{\lambda} f^{\overline{\lambda}}}{f^{k^n}},
\end{equation}
recalling $f^{\nu}$ is the dimension of the irreducible representation of the permutation group $S_{\abs{\nu}}$ (which equals the number of standard Young tableaux of shape $\nu$) and $k^n$ denotes a rectangular Young diagram with $n$ rows and $k$ columns.
In~\cite[Thm.~1.4]{panova2018skew}, it was shown the random irreducible component of~\eqref{eq:exterior_power_decomp} corresponds to a pair of Young diagrams $(\lambda,\lambda')$, where $\lambda$ has the same distribution as the Young diagram formed by taking the boxes with its entry $<m$ of a uniformly random Young tableau of rectangular shape $k^n$.
Thus, the limit shape for Young diagrams with the probability measure
\[
\mu_{n,k}^{\langle m \rangle}(\lambda) = \dfrac{\dim V_{GL_{n}}(\lambda)\cdot\dim V_{GL_{k}}(\overline{\lambda}')}{\binom{nk}{m}}
\]
in the limit $n,k,m\to\infty$, $\frac{k}{n}\to\mathrm{const}, \frac{m}{nk}\to\mathrm{const}$ is the same as the level lines of the limit shape for plane partitions presented in~\cite{pittel2007limit}.

Since
\[
\bigwedge\left(\CC^{n}\otimes\CC^{k}\right) = \bigoplus_{m=0}^{nk} \bigwedge\nolimits^{m}\left(\CC^{n}\otimes\CC^{k}\right),
\]
the measure $\mu_{n,k}(\lambda)$ can be written as
\begin{equation}
  \label{eq:binomialization}
  \mu_{n,k}(\lambda)=\sum_{m=0}^{nk}\frac{\mu_{n,k}^{\langle m \rangle}(\lambda) \binom{nk}{m}}{2^{nk}},  
\end{equation}
for the finite values of $n,k,m$. In the limit $n,k\to\infty$, the binomial distribution concentrates on the point $m = \frac{nk}{2}$.
Therefore, the limit shape~\eqref{eq:limit-shape-gl} coincides with the limit shape for $\mu_{n,k}^{\langle \frac{nk}{2} \rangle}(\lambda)$ that was obtained in~\cite{panova2018skew} and is the same as the corresponding level line of the
plane partitions in the box from the paper~\cite{pittel2007limit}.

In the paper~\cite{borodin2007asymptotics}, it was demonstrated that the ``binomialization'' of the measure $\mu_{n,k}^{\langle m\rangle}(\lambda)$ given by~\eqref{eq:binomialization} is the Krawtchouk ensemble and its limit shape was related to the $m=\frac{nk}{2}$ level line of plane partitions in the box.
In particular, compare the following:
\begin{itemize}
\item Equation~\eqref{eq:measure_factorized} recalling $W(a_i) = \binom{k+n-1}{a_i}$ with~\cite[Eq.~(5.2)]{borodin2007asymptotics} (or~\cite[Eq.~(2.4)]{johansson2002non}) at $p = 1/2$;
\item Equation~\eqref{eq:gl-dim-to-sm-dim} with the probability measure denoted $M_{n,N,M}$ in~\cite[Sec.~5]{borodin2007asymptotics}; and
\item Equation~\eqref{eq:binomialization} with~\cite[Eq.~(5.1)]{borodin2007asymptotics}.
\end{itemize}
The relation of the Krawtchouk ensemble to the skew $(GL_{n},GL_{k})$-duality does not seem to have been noticed in~\cite{borodin2007asymptotics,panova2018skew}.
However, it does appear indirectly in~\cite[Prop.~5.1]{Johansson01} through the use of dual RSK and the result~\cite[Thm.~7.1]{BR01}.
From the Krawtchouk ensemble perspective, we obtain Equation~\eqref{eq:gl-dim-to-sm-dim} from~\cite[Prop.~4.3]{BO06}.

We describe the appearance of the skew Howe duality in~\cite{johansson2002non} and connect the lozenge tilings with domino tilings of Aztec diamonds.
We note that the function denoted $w[h]$ in the proof of~\cite[Thm.~2.2]{johansson2002non}, where using our notation $h = (a_n, \dotsc, a_2, a_1)$ and depends on a parameter $\omega$, can be described in terms of lozenge tilings as
\begin{equation}
\label{eq:weighted_lozenge_sum}
w[h] = (1 + \omega^2)^{\binom{n}{2}} \times \prod_{i < j} (a_i - a_j) \times \prod_{i=1}^n \frac{\omega^{\lambda_i}}{(i-1)!}
= \sum_{L} (1 + \omega^2)^{\#B} \omega^{\#R},
\end{equation}
where we sum over all half hexagon lozenge tilings $L$ giving $V(\lambda)$ and $\#X$ denotes the number of tiles $X$ in $L$.
We note that these formulas agree from the Weyl dimension formula~\eqref{eq:weyl_dim} and the fact that the number of $B$ tiles and $R$ tiles is fixed for any given $\lambda$.
Similarly, the formula for $w[h']$ is the same sum over the tilings for $V(\overline{\lambda}')$.
Therefore, we have~\cite[Eq.~(2.11)]{johansson2002non} at $\omega = 1$ is our probability measure $\mu_{n,k}(\lambda)$.

We can then think of the factor $1 + \omega^2$ as choosing between a pair of horizontal or vertical domino tiles in the Aztec diamond, and therefore there exists a $2^{\binom{n}{2}}$-to-$1$ mapping of Aztec diamond tilings to lozenge tilings for $V(\lambda)$.
As a consequence, we have that there are
\[
2^{\binom{n}{2}} 2^{\binom{k}{2}} 2^{nk} = 2^{n(n-1)/2 + k(k-1)/2 + nk} = 2^{(n+k)(n+k-1)/2} = 2^{\binom{n+k}{2}}
\]
tilings of the Aztec diamond of order $n+k-1$, first shown in~\cite{EKLP92I}.\footnote{The bijection between NILPs consisting of large Schr\"oder paths and Aztec diamond tilings given by the DR paths in~\cite{johansson2002non} was rediscovered in~\cite{BK05,EF05}.}
However, we are unable to find such an explicit mapping to give a combinatorial proof of Equation~\eqref{eq:weighted_lozenge_sum}.

\begin{figure}[t]
  \centering
  \includegraphics[width=8cm]{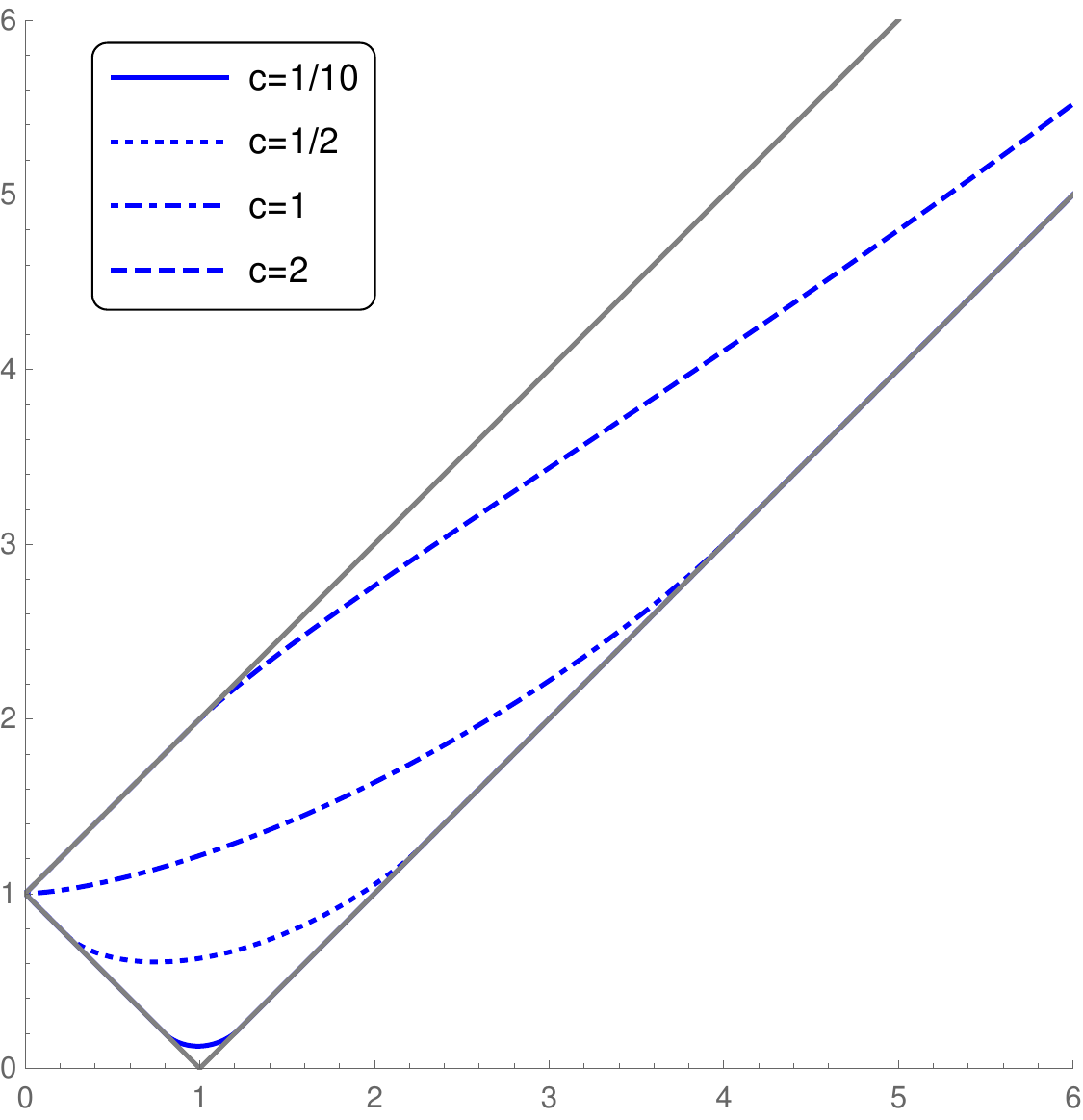}

  \caption{A rescaled version of the limit shapes from Schur--Weyl duality.}
  \label{fig:rescaled_SW_limit}
\end{figure}

We also note similarities with the boundary point fluctuations in~\cite{Biane01} using the measure $\mu_{n,k}^{SW}$ coming from Schur--Weyl duality under the limit $\sqrt{k}/n \to c$ as $n,k \to \infty$.
When $c < 1$, the length of the partition determines the left boundary and will have Tract--Widom GUE fluctuations.
When $c > 1$, the length of the partition is deterministic and the left boundary point will have Tracy--Widom GUE fluctuations.
Lastly, when $c = 1$, this is the critical case that was studied in detail in~\cite{borodin2007asymptotics}, where the fluctuations are described using the discrete Hermite kernel.
These also correspond to the three different regimes of~\cite{GTW01} (while that corresponds to the right boundary of our limit shapes, it is equivalent to the left by the $n \leftrightarrow k$ symmetry).
Compare the left boundary for Figure~\ref{fig:rescaled_SW_limit} with the limit shapes obtained from $\mu_{n,k}$ with $k / n \to c$ such as in Figures~\ref{fig:n-10-k-90-diagram-and-limit-shape} and \ref{fig:n-20-k-10-diagram-and-limit-shape}.

We bring in another character into our ensemble cast, the \defn{$z$-measure}, that comes from harmonic analysis on the infinite symmetric group~\cite{KOV93}.
To do so, we begin by looking at the Schur measure~\cite{Okounkov00,Okounkov01}, which is the measure on partitions $\lambda$ from the Cauchy identity~\eqref{eq:cauchy} renormalized so the sum is $1$.
By specializing $x_i = 1$ and $y_j = \xi$, we obtain a $\xi$-deformed version of the measure $\mu^{\square \infty}_{n,k}(\lambda)$ from the introduction
\begin{equation}
\label{eq:zeta_cauchy_measure}
\mu^{\square \infty}_{n,k;\xi}(\lambda) = (1 - \xi)^{nk} \xi^{\abs{\lambda}} \dim V_{\GL_n}(\lambda) \dim V_{\GL_k}(\lambda).
\end{equation}
Note this does not make sense when $\xi = 1$ as the sum from the Cauchy identity would be infinite.
By using the hook-content formula~\cite[Thm.~15.3]{Stanley71} for $\dim V_{\GL_n}(\lambda)$ and $\dim V_{\GL_k}(\lambda)$ and replacing $n,k \in \ZZ_{>0}$ with $z,z' \in \CC$, we have the $z$-measure~\cite[Eq.~(2.4)]{Okounkov01} (\textit{cf.}~\cite[Eq.~(1.3)]{BO06}).
The connection with the Meixner ensemble was given in~\cite{BO06,Johansson01}\footnote{The Meixner ensemble, Howe duality, and a last passage percolation model have been linked in~\cite{Johansson10,MS20,RS05}.} (see also~\cite[Ex.~1.5]{BO05gamma} for an explicit statement) and by taking $z = n$ and $z' = -k$, we obtain the Krawtchouk ensemble with $\omega = \xi/(\xi - 1)$ with $\xi < 0$~\cite[Prop.~4.1]{BO06}.

We can also describe a relation to the \defn{$zw$-measure}, which now comes from the ``big'' group $U_{\infty}$~\cite{BO05,BO05II}.
An alternative way of describing skew Howe duality is we have a natural left action of $U_n \times U_k$ acting on $n \times k$ matrices by $(a, b) \cdot M = a M b^{-1}$.
For any representation $V$ of $U_n$, we can construct a probability measure $P_V$ by
\[
\widetilde{\chi}(V) = \sum_{\lambda} \widetilde{\chi}(V_{\lambda}) P_V(\lambda),
\]
where $\widetilde{\chi}$ denotes the normalized irreducible character.
Therefore, our measure $\mu_{n,k}$ is this probability measure for the action given above.
In the case $k = n$, skew Howe duality becomes the biregular representation, and additionally we can describe $U_n$ as the symmetric space $(U_n \times U_n) / U_n$, where we use the diagonal embedding $\Delta(U_n)$.
When taking the limit as $n \to \infty$, we need to take a slightly bigger space that still carries the $U_{\infty} \times U_{\infty}$ action.
These give a family of representations parameterized by two complex parameters $z,w$ (with some restrictions), which give probabilities that when restricting down to finite $n$ are the $zw$-measures and have explicit formulas~\cite{BO05}.

Finally, we note that we can take the actions of $U_n$ on the left and the right to be different, which means that the irreducible representations are actually indexed by a pair of partitions.
We can induct a $\gl_n \oplus \gl_m$ representation, which we equate to a $U_n \times U_m$ (polynomial) representation, to a representation of the Lie superalgebra $\gl(m|n)$, where we obtain what is known as a Kac module, and the irreducible representations can be indexed by a pair of partitions with one being the positive part and the other being the negative part.
There is a super analog of Howe duality for $(\GL_k, \gl(m|n))$~\cite{howe1989remarks,Sergeev01,CW01} (see also, \textit{e.g.},~\cite[Ch.~5.2]{CW12}) on the supersymmetric algebra of $\CC^k \otimes \CC^{m|n}$, denoted $\mathcal{S}(\CC^k \otimes \CC^{m|n})$ (see, \textit{e.g.},~\cite[Sec.~5.1.1]{CW12} for a precise definition) giving a super Cauchy identity that ``interpolates'' between the usual Cauchy identity and the dual Cauchy identity:
\[
\sum_{\lambda} s_{\lambda}(\xx) s_{\lambda}(\yy/\mathbf{w}) = \prod_{i,j} \frac{1 + x_i w_j}{1 - x_i y_j},
\]
where $s_{\lambda}(\yy/\mathbf{w})$ denotes the hook Schur or Schur supersymmetric function and are the characters for $\gl(m|n)$ modules~\cite{Kac77} (see also~\cite[Sec.~I.5]{Macdonald98}, where they can be defined in terms of plethystic substitution).
In~\cite{BO05}, they actually use a bigger space $\mathfrak{U}$ that has a $U_{\infty} \times U_{\infty}$ action to describe the extreme characters.
This suggests that $\mathfrak{U}$ corresponds to a Kac representation or the corresponding irreducible representation for $\gl(\infty|\infty)$ as a limit of the Kac representations for $\gl(n|n)$ or the limit of $\mathcal{S}(\CC^k \otimes \CC^{m|n})$ when $k,n,m \to \infty$.

\subsubsection{Types BCD}

By instead considering the infinite orthogonal and symplectic groups and limits of symmetric spaces~\cite{OO06}, we arrive at the \defn{type $BC$ $z$-measure}~\cite{BO05gamma,OO12}, which is defined on partitions with $\ell(\lambda) \leq l$ by
\begin{align*}
\mu^{BC}_{z,z',\alpha,\beta}(\lambda) & = \frac{\displaystyle \prod_{1 \leq i < j \leq l} \left( \left(b_i + \vartheta \right)^2 - \left(b_j + \vartheta \right)^2 \right)^2}{Z_l(z,z',\alpha,\beta)} \prod_{i=1}^l W_{z, z', \alpha, \beta; l}(b_i),
\allowdisplaybreaks \\[5pt]
W_{z, z', \alpha, \beta; l}(x) & = \frac{\displaystyle \left(x + \vartheta \right) \frac{\Gamma(x+2\vartheta) \Gamma(x+\alpha+1)}{\Gamma(x+\beta+1) \Gamma(x+1)}}{
\Gamma(z-x+l) \Gamma(z'-x+l) \Gamma(z+x+l+2\vartheta) \Gamma(z'+x+l+2\vartheta)},
\allowdisplaybreaks \\[5pt]
Z_l(z, z', \alpha, \beta) & = \prod_{i=1}^l \frac{\Gamma(z+z'+\beta+i) \Gamma(\alpha+i) \Gamma(i)}{\Gamma(z+i) \Gamma(z+\beta+i) \Gamma(z'+\beta+i)\Gamma(z+z'+l+\alpha+\beta+i)},
\end{align*}
where $b_i := \lambda_i + l - i$ and $\vartheta = \frac{\alpha+\beta+1}{2}$.
Note that $Z_l(z, z', \alpha, \beta)$ is the normalization constant.
In~\cite{Cuenca18,Cuenca18II}, Cuenca constructed an explicit kernel for the corresponding point process and showed a relation to the one for the $zw$-measure.
Furthermore, as described in~\cite{Cuenca18}, there are special values of the pairs $(\alpha, \beta)$ (there denoted $(a,b)$) that correspond to the limits of symmetric spaces first examined in~\cite{OO06}, where the $BC$ $z$-measure describes an approximation of the spectral measure from a generalization of the biregular representation at finite values.
As given in~\cite[Sec.~8]{BO05}, the $BC$ $z$-measure can be constructed from multivariate Jacobi polynomials, which are $BCD$ analogs of Jack polynomials and are characters for the type $BCD$ irreducible representations when suitably normalized~\cite[Thm.~1.2]{OO98}.
This can also be considered as a type $BC$ Weyl group, the group of signed permutations, analog of $z$-measure from the Plancharel measure.

Our goal is to show that our measure $\mu_{n,k}$ is equal to a specialization of $\mu^{BC}_{z,z',\alpha,\beta}$.
Comparing $\mu^{BC}_{z,z',\alpha,\beta}$ with our measures, we should set $a_i = b_i + \vartheta$, and hence we need to have $\alpha + \beta = -1, 0, 1$ (so $\vartheta = 0, \frac{1}{2}, 1$, respectively) to make our coordinates to agree.
This and the symmetric space description suggests we should take $(\alpha, \beta) = (\pm 1/2, \pm 1/2)$, and these indeed yield our desired specializations.
For the case of $\alpha = \beta = 0$, we will need an ``odd'' measure that would correspond to the skew Howe duality for $(\Or_{2l+1}, \SO_{2k+1})$ (or $(\SO_{2l}, \Pin_{2k+1})$) if it existed:
\begin{align*}
  \mu_{n,k}(\lambda) & = Z \cdot \dim V_{\Or_{2l+1}}(\lambda) \cdot\dim V_{\SO_{2k+1}}(\overline{\lambda}')
= Z \cdot \dim V_{\SO_{2l+1}}(\lambda) \cdot M_1^D(\lambda+\fw_n)
  \\ & = \widetilde{Z} \frac{\displaystyle \prod_{i=1}^l a_i \times \prod_{1 \leq i < j \leq l} (a_i^2 - a_j^2)^2}{\displaystyle \prod_{i=1}^l (k+l-a_i-1/2)! (k+l+a_i-1/2)!}
\end{align*}
for some normalization constant $Z$ and
\[
\widetilde{Z} = Z \frac{\displaystyle \prod_{i=1}^l (2k+2n-2i+1)!}{\displaystyle \prod_{i=1}^l (l-i+1/2) i! \times \prod_{1 \leq i < j \leq l} (2l+1-i-j)}.
\]

\begin{table}
\[
\begin{array}{cccccccc}
\toprule
(\alpha, \beta) & (1/2, 1/2) & (1/2, -1/2) & (-1/2, -1/2) & (0,0)
\\ \midrule
G / K & (\Sp_{2l} \times \Sp_{2l}) / \Sp_{2l} & (\Or_{2l+1} \times \Or_{2l+1}) / \Or_{2l+1} & (\Or_{2l} \times \Or_{2l}) / \Or_{2l} & U_{2n} / (U_n \times U_n)
\\ \midrule
(G_1, G_2) & (\Sp_{2l}, \Sp_{2k}) & (\SO_{2l+1}, \Pin_{2k}) & (\Or_{2l}, \SO_{2k}) & \text{``}(\Or_{2l+1}, \SO_{2k+1})\text{''}
\\ \bottomrule
\end{array}
\]
\caption{The values of the parameters $(\alpha, \beta)$, the corresponding limit of symmetric spaces $G / K$, and the corresponding skew Howe dual pair $(G_1, G_2)$.}
\label{table:BC_values_limits}
\end{table}

\begin{thm}
\label{thm:BC_z_measure_specialization}
Let $k$ be an even positive integer. For $(\alpha, \beta)$ and the corresponding group given by Table~\ref{table:BC_values_limits}, we have
\[
\mu_{n,k}(\lambda) = (-1)^{\sum_{i=1}^l a_i} C_{n,k} \cdot \mu^{BC}_{z,z',\alpha,\beta}(\lambda).
\]
and $C_{n,k}$ is a constant that does not depend on $\lambda$.
\end{thm}

\begin{proof}
We first recall some basic facts about the Gamma function $\Gamma(z)$, where $z \in \CC \setminus \ZZ_{\leq 0}$.
In particular, it satisfies
\[
\Gamma(1) = 1,
\qquad\qquad
\Gamma(z+1) = z \Gamma(z),
\qquad\qquad
\Gamma(1-z)\Gamma(z) = \frac{\pi}{\sin(\pi z)},
\]
where the last identity (Euler's reflection formula) holds if and only if $z \notin \ZZ$.
Note that for any positive integer $m$, we have $\Gamma(m) = (m-1)!$.

As previously mentioned, we let
\[
a_i = b_i + \vartheta = \lambda_i + (n - i) + \frac{\alpha+\beta+1}{2}.
\]
Therefore, we have
\begin{align*}
\mu^{BC}_{z,z',\alpha,\beta}(\lambda) & = \frac{\displaystyle \prod_{1 \leq i < j \leq l} \left( (a_i)^2 - (a_j)^2 \right)^2}{Z_n(z,z',\alpha,\beta)} \prod_{i=1}^l W_{z, z', \alpha, \beta; l}(b_i),
\allowdisplaybreaks \\[5pt]
W_{z, z', \frac{1}{2}, \frac{1}{2}; l}(b_i) & = \frac{\displaystyle a_i \frac{\Gamma(a_i+1)\Gamma(a_i+1/2)}{\Gamma(a_i+1/2)\Gamma(a_i)}}{\Gamma(z-a_i+l+1)\Gamma(z'-a_i+l+1)\Gamma(z+a_i+l+1)\Gamma(z'+a_i+l+1)}
\\ & = \frac{\displaystyle a_i^2}{\Gamma(z-a_i+l+1)\Gamma(z'-a_i+l+1)\Gamma(z+a_i+l+1)\Gamma(z'+a_i+l+1)},
\allowdisplaybreaks \\[5pt]
W_{z, z', \frac{1}{2}, -\frac{1}{2}; l}(b_i) & = \frac{\displaystyle a_i \frac{\Gamma(a_i+1/2)\Gamma(a_i+1)}{\Gamma(a_i)\Gamma(a_i+1/2)}}{\Gamma(z-a_i+l+1/2)\Gamma(z'-a_i+l+1/2)\Gamma(z+a_i+l+1/2)\Gamma(z'+a_i+l+1/2)}
\\ & = \frac{\displaystyle a_i^2}{\Gamma(z-a_i+l+1/2)\Gamma(z'-a_i+l+1/2)\Gamma(z+a_i+l+1/2)\Gamma(z'+a_i+l+1/2)},
\allowdisplaybreaks \\[5pt]
W_{z, z', -\frac{1}{2}, -\frac{1}{2}; l}(b_i) & = \frac{\displaystyle a_i \frac{\Gamma(a_i)\Gamma(a_i+1/2)}{\Gamma(a_i+1/2)\Gamma(a_i+1)}}{\Gamma(z-a_i+l)\Gamma(z'-a_i+l)\Gamma(z+a_i+l)\Gamma(z'+a_i+l)}
\\ & = \frac{\displaystyle 1}{\Gamma(z-a_i+l)\Gamma(z'-a_i+l)\Gamma(z+a_i+l)\Gamma(z'+a_i+l)},
\allowdisplaybreaks \\[5pt]
W_{z, z', 0, 0; l}(b_i) & = \frac{\displaystyle a_i \frac{\Gamma(a_i+1/2)\Gamma(a_i+1/2)}{\Gamma(a_i+1/2)\Gamma(a_i+1/2)}}{\Gamma(z-a_i+l+1/2)\Gamma(z'-a_i+l+1/2)\Gamma(z+a_i+l+1/2)\Gamma(z'+a_i+l+1/2)}
\\ & = \frac{\displaystyle a_i}{\Gamma(z-a_i+l+1/2)\Gamma(z'-a_i+l+1/2)\Gamma(z+a_i+l+1/2)\Gamma(z'+a_i+l+1/2)},
\end{align*}
Next, we set $z = k$ and $z' = 1/2 - l - \vartheta$, and hence every denominator above is equal to
\begin{align*}
& \Gamma(k-a_i+l+\vartheta)\Gamma(1/2-a_i)\Gamma(k+a_i+l+\vartheta)\Gamma(1/2+a_i)
\\ & \hspace{120pt} = (k-a_i+l+\vartheta-1)!(k+a_i+l+\vartheta-1)! \frac{\pi}{\sin \pi(1/2+a_i)}
\\ & \hspace{120pt} = (-1)^{a_i} \pi (k-a_i+l+\vartheta-1)! (k+a_i+l+\vartheta-1)!,
\end{align*}
where we used Euler's reflection formula for the first equality (recall that $a_i \in \ZZ_{>0}$).
The claim follows by comparing these formulas to the corresponding measures.
\end{proof}

From this connection, the relationship between the kernels in~\cite{Cuenca18} could be seen as the reflection of the fact that we essentially get the same limit shapes for type $A$ as for types $BCD$.
The agreement of the limit shapes for types $BCD$ can also be seen as coming from the fact they are all controlled by the $BC$ $z$-measure in the limit as $n \to \infty$.

Additionally, our measure equals the spectral measure $\mu_{l,\alpha,\beta}^{\omega}$ for an extremal character $\omega$ of a ``big'' group restricted down to its rank $l$ subgroup (up to an overall constant).
This is an immediate consequence of~\cite[Prop.~5.1]{OO12}, that the (appropriately normalized) multivariate Jacobi polynomials $P_{\lambda}(\mathbf{1}^l) = \dim V_{G_1}(\lambda)$ with the rank of $G_1$ being $l$, and the definition of $\mu_{n,k}$, where $\mathbf{1}^m = (1, \dotsc, 1)$ be the sequence with every entry $1$ of length $m$.
We note there is a minor typo in~\cite[Eq.~(26)]{OO12}, where the leading factor should be $\prod_{j=1}^K \frac{\beta_j(2-\beta_j)}{2}$.

\begin{cor}
\label{cor:spectral_restriction}
Fix the extremal character $\omega = (0, \mathbf{1}^k, k)$.
For $(\alpha, \beta)$ and the corresponding group given by Table~\ref{table:BC_values_limits}, we have $\mu_{n,k} = \mu_{l, \alpha, \beta}^{\omega}$.
\end{cor}

\begin{remark}
\label{rem:spectral_det}
There is another formula for the restricted spectral measures for $\beta = -\frac{1}{2}$ given in~\cite[Thm.~2.8]{BK10} involving a determinant.
A natural question is if the matrix given there, which we will denote by $\mathcal{M}(\lambda)$, is equal to ours up to some power of $2$ multiplying each entry, but this turns out to not be the case.
If we consider $(\alpha, \beta) = (1/2,-1/2)$ and $M^{BC}_1(\emptyset)$ for $l = 3$ and $k = 4$, then we compute:
\[
M^{BC}_1(\emptyset) =
\begin{bmatrix}
275 & 75 & 20 \\
297 & 90 & 28 \\
132 & 42 & 14
\end{bmatrix}
\qquad\qquad
\mathcal{M}(\emptyset) =
\begin{bmatrix}
\frac{55}{1024} & \frac{35}{512} & \frac{5}{64} \\[3pt]
\frac{49}{1024} & \frac{17}{256} & \frac{7}{64} \\[3pt]
\frac{5}{256} & \frac{7}{256} & \frac{7}{128}
\end{bmatrix}
\]
Note that there is no element in $M^{BC}_1(\emptyset)$ that is a multiple of $17$.
We can also extend the construction for the case $\alpha = \beta = \frac{1}{2}$ by using the normalized Jacobi polynomials $\mathsf{J}^{(a,b)}_j(x) = \frac{J^{(a,b)}_j(x)}{2c_{j+1}}$, where $c_j = \frac{1 \cdot 3 \dotsm (2j-1)}{2 \cdot 4 \dotsm 2j}$ for $j > 0$ and $c_0 = 1$.
\end{remark}


\subsubsection{Relationship with Howe duality}
\label{sec:howe_duality}

For completeness, we discuss how our results in lozenge tilings are related to Howe duality.
Here, we assume $k \in \ZZ$ (not necessarily even).
We start with the classical result of Howe duality for $(\GL_n, \GL_k)$ with restricting the partition $\lambda$ to be inside of an $\min(n,k) \times m$ rectangle from the lozenge tiling description, yielding the measure $\mu_{n,k}^{\square m}$ from the introduction.
Without loss of generality, we assume $k \leq n$.
We can take a half hexagon tiling parameterizing the crystal $B(\lambda)$ for $\GL_n$ and the one for $B(\lambda)$ for $\GL_k$ and join them together at the top point after reflecting the $\GL_k$ across the vertical axis.
We can then ignore the portion of the $\GL_n$ half hexagon that is fixed by $\lambda_{k+1} = \cdots = \lambda_n = 0$.
This gives us a lozenge tiling of a partition inside of an $n \times m \times k$ box.
As an example, consider $n = 4$, $k = 2$, and $m = 1$, one such lozenge tiling for $\lambda = (1,1,0,0) \subseteq 1^2$ is
\[
\iftikz
\begin{tikzpicture}[scale=0.5,baseline=0]
\draw (150:2) -- ++(150:2) -- ++(0,1) -- ++(30:4) -- ++(-30:2) -- ++(0,-1) -- ++(-150:4);
\foreach \y in {3,4}
  \draw[fill=black!50] (0,\y) -- ++(30:1) -- ++(150:1) -- ++ (30:-1) -- cycle;
\foreach \y in {0,1}
  \draw[color=black!20,fill=black!10] (0,\y) -- ++(30:1) -- ++(150:1) -- ++ (30:-1) -- cycle;
\draw[color=black!20,fill=black!10] (150:1) -- ++(30:1) -- ++(150:1) -- ++ (30:-1) -- cycle;
\foreach \y in {1,3}
  \draw[fill=blue!30] (150:1) ++ (0,\y) -- ++(30:1) -- ++(150:1) -- ++ (30:-1) -- cycle;
\foreach \y in {0,2}
  \draw[fill=blue!30] (150:2) ++ (0,\y) -- ++(30:1) -- ++(150:1) -- ++ (30:-1) -- cycle;
\draw[fill=blue!30] (150:3) -- ++(30:1) -- ++(150:1) -- ++ (30:-1) -- cycle;
\draw[fill=blue!30] (30:1) ++ (0,3) -- ++(30:1) -- ++(150:1) -- ++ (30:-1) -- cycle;
\draw[fill=green!30] (0,2)  -- ++(0,1) -- ++(150:1) -- ++ (0,-1) -- cycle;
\draw[fill=green!30] (150:2) ++ (0,1)  -- ++(0,1) -- ++(150:1) -- ++ (0,-1) -- cycle;
\draw[fill=red!30] (150:1) ++ (0,2) -- ++(0,1) -- ++ (210:1) -- ++(0,-1) -- cycle;
\draw[fill=red!30] (150:3) ++ (0,1) -- ++(0,1) -- ++ (210:1) -- ++(0,-1) -- cycle;
\foreach \x in {1,2}
  \draw[fill=red!30] (30:\x) ++ (0,2) -- ++(0,1) -- ++ (210:1) -- ++(0,-1) -- cycle;
\draw[dotted] (0,-0.5) -- (0,5.5);
\end{tikzpicture}
\fi
\]
where we have drawn in the fixed portion in light gray.

Now we describe Howe duality for the pairs $(\SO_n, \spn_{2k})$ or $(\Sp_{2k}, \so_{2l})$ using lozenge tilings with bounding $\lambda$ inside of a rectangle.
We can perform the analogous joining of the quarter hexagons for the $\SO_n$ and $\Sp_{2k}$ representations to form a half hexagon and removing the fixed parts, where we also remove the bottom $B$ tiles and leftmost $R$ tiles from the $\Sp_{2k}$ tiling (they are completely fixed).
The corresponding measure for $(\SO_n, \Sp_{2k})$ was investigated in~\cite[Lemma~2.2]{FN09}, which can be seen as a specialization of the $BC$ $z$-measure $\mu_{k,k+l,\alpha,\beta}^{BC}$ for $\ell(\lambda) \leq \min(k,n)$, where we set the variables in~\cite{FN09} to $N = l + k$ and $p = l$.

\section{Open Problems}
\label{sec:open_problems}

Here we gather some open problems and conjectures from this work.

Since we have $q$-analogs for our multiplicity and dimension formulas (hence our probability measures), a natural question is to determine how the parameter $q$ changes the limit shape.
For the case of $(\GL_k, \GL_n)$, we can compute one such $q$-analog probability measure
\begin{equation}
\label{eq:qprob_ps_A}
\mu_{n,k}^A(\lambda; q) = \frac{q^{\Abs{\lambda}} \dim_q\bigl( V_{\GL_n}(\lambda) \bigr) \cdot q^{\Abs{\overline{\lambda}'}} \dim_q\bigl( V _{\GL_n}(\overline{\lambda}') \bigr)}{N^A_{n,k}(q)},
\end{equation}
where
\[
N^A_{k,n}(q) = q^{P_{k-1} + (n-k) \binom{k}{2}} 2^k \prod_{i=1}^{k-1} (q^i + 1)^{2(k-i)} \times \prod_{j=k+1}^n \prod_{i=1}^k (q^{j-i} + 1)
\quad \text{ with } \quad
P_k = \frac{k(k+1)(2k+1)}{6} 
\]
the square pyramidal numbers~\oeis{A000330}~\cite{OEIS}.
We obtain~\eqref{eq:qprob_ps_A} by taking the principal specialization of the alternative form of the dual Cauchy identity
\begin{equation}
\label{eq:alt_dual_Cauchy}
\sum_{\lambda \subseteq k^n} s_{\lambda}(x_1, \dotsc, x_n) s_{\overline{\lambda}'}(y_1, \dotsc, y_k) = \prod_{i=1}^n \prod_{j=1}^k (x_i + y_j),
\end{equation}
which is constructed from the dual Cauchy identity by using
\[
\ch(V(\overline{\lambda}'))(y_1, \dotsc, y_k)
= \prod_{j=1}^k y_i^{\overline{\lambda}'_1-n} \ch(V(\lambda')^*)(y_1, \dotsc, y_k)
= \prod_{j=1}^k y_i^n \ch(V(\lambda'))(y_1^{-1}, \dotsc, y_k^{-1}).
\]
In particular, we substitute $y_i \mapsto y_i^{-1}$ to account for the $\lambda' \to \overline{\lambda}'$ change, and then we multiply by $y_1^n \cdots y_k^n$ to obtain~\eqref{eq:alt_dual_Cauchy}.
We can similarly construct other $q$-analogs of our probability measures using the principal specializations of the formulas from, \textit{e.g.},~\cite{proctor1993reflection}.

However, the measure $\mu_{n,k}^A(\lambda; q)$ is not a unique $q$-analogue of the measure $\mu_{n,k}(\lambda)$. Other choices are given by the following conjecture based on experimental data.

\begin{conj}
We have probability measures on all partitions $\lambda$ inside of an $n \times k$ rectangle:
\begin{subequations}
\label{eq:prob_conj_A}
\begin{align}
\mu_{n,k}^{A2}(\lambda; q) & = \frac{q^{\Abs{\overline{\lambda}}} \dim_q\bigl( V_{\GL_n}(\lambda) \bigr) \cdot q^{\Abs{\overline{\lambda}'}} \dim_q\bigl( V _{\GL_k}(\overline{\lambda}') \bigr)}{
\displaystyle 2 \prod_{i=1}^{k+1} (q^i + 1)^{k+2-i} \times \prod_{j=k+1}^n \prod_{i=1}^k (q^{j+2-i} + 1)
},
\label{eq:prob_conj_A2} \\
\mu_{n,k}^{A3}(\lambda; q) & = \frac{q^{\Abs{\overline{\lambda}}} \dim_q\bigl( V_{\GL_n}(\lambda) \bigr) \cdot q^{\abs{\overline{\lambda}'}+\Abs{\overline{\lambda}'}} \dim_q\bigl( V _{\GL_k}(\overline{\lambda}') \bigr)}{
\displaystyle \prod_{i=1}^{2k} (q^i + 1)^{k-\abs{k-i}} \times \prod_{j=k+1}^n \prod_{i=1}^k (q^{j+k-i} + 1)
}.
\label{eq:prob_conj_A3}
\end{align}
\end{subequations}
\end{conj}

We note that the numerators of~\eqref{eq:prob_conj_A} become
\begin{align*}
q^{\Abs{\overline{\lambda}}} \dim_q\bigl( V_{\GL_n}(\lambda) \bigr) \cdot q^{\Abs{\overline{\lambda}'}} \dim_q\bigl( V _{\GL_k}(\overline{\lambda}') \bigr)
& = 
q^{\Abs{\overline{\lambda}'}} \dim_q\bigl( V_{\GL_n}(\lambda) \bigr) \cdot M_q^A(\lambda),
\\
q^{\Abs{\overline{\lambda}}} \dim_q\bigl( V_{\GL_n}(\lambda) \bigr) \cdot q^{\abs{\overline{\lambda}'}+\Abs{\overline{\lambda}'}} \dim_q\bigl( V _{\GL_k}(\overline{\lambda}') \bigr)
& = 
q^{\abs{\overline{\lambda}'}+\Abs{\overline{\lambda}'}} \dim_q\bigl( V_{\GL_n}(\lambda) \bigr) \cdot M_q^A(\lambda).
\end{align*}
We also remark that we can rewrite the numerator of $\mu_{n,k}^{A2}(\overline{\lambda}; q)$ as the classical skew Howe duality
\[
q^{\Abs{\lambda}} \dim_q\bigl( V_{\GL_n}(\lambda) \bigr) \cdot q^{\Abs{\lambda'}} \dim_q\bigl( V _{\GL_k}(\lambda') \bigr)
\]
by using the equality $\dim_q V_{\GL_n}(\lambda) = \dim_q V_{\GL_n}(\overline{\lambda})$, which is immediate from the well-known fact $V_{\GL_n}(\lambda)^* \iso V_{\GL_n}(\overline{\lambda})$ up to a shift by the determinant representation (see also Theorem~\ref{thm:q_mult_dim_A}).

We have some initial data  computed by using a discrete steepest descent method that suggests that limit shapes are deformed in the limit $n,k\to\infty$, $q\to 1$ such that $\frac{k}{n}=c+\mathcal{O}\left(\frac{1}{n}\right)$, $q=1-\frac{\gamma}{n}$. 
For the measure $\mu_{n,k}^A(\lambda; q)$ in~\eqref{eq:qprob_ps_A}, we have produced the estimated limit shapes for various values of $\gamma$ in Figure~\ref{fig:q-diagrams-and-limit-shape}. Derivation of the formulas that describe these limit shapes remains an open problem. In the case $q=\mathrm{const}$, which corresponds to $\gamma=\pm\infty$ the limit shape degenerates to one of the straight horizontal lines, which are shown in black in Figure~\ref{fig:q-diagrams-and-limit-shape}. The upper line is for $q>1$ and the lower for $q<1$.

\begin{figure}[t]
  \includegraphics[width=10cm]{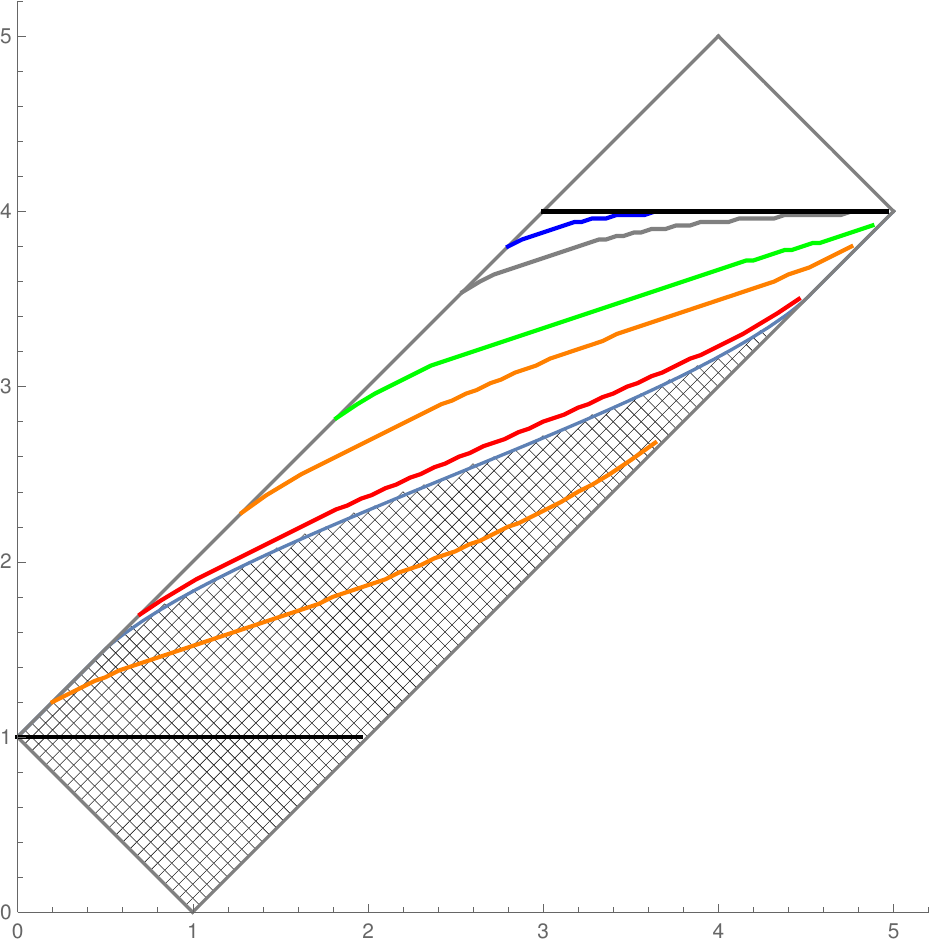}
\caption{The most probable diagram for $n=25, k=100$, the limit shape for $c=4$ and the upper boundaries for the most probable diagrams for the measure $\mu_{n,k}(\lambda;q)$ where $q=1-\gamma/n$ and $\gamma=-\frac{1}{2},\frac{1}{10},\frac{1}{2},2 ,10, 25$. Black horizontal lines correspond to $\gamma=+\infty$ (upper) and $\gamma=-\infty$ (lower). }
\label{fig:q-diagrams-and-limit-shape}
\end{figure}

An even more general problem is to describe the asymptotic behavior of the character measure that can be introduced for  $GL_{n}\times GL_{k}$ as follows:
\[
  \mu_{n,k}(\lambda|\{x_{i}\}_{i=1}^{n},\{y_{j}\}_{j=1}^{k})=\frac{\displaystyle \sum_{\lambda \subseteq k^n} s_{\lambda}(x_1, \dotsc, x_n) s_{\overline{\lambda}'}(y_1, \dotsc, y_k)}{\displaystyle \prod_{i=1}^n \prod_{j=1}^k (x_i + y_j)},
\]
and similarly for other dual pairs of groups. We suggest that the limit shape for $n,k\to\infty$ and $x_{i}=e^{\varphi(i/n)}, y_{j}=e^{\psi(j/n)}$ with smooth $\varphi,\psi$ is described by the Burgers equation.
The asymptotic behavior of the character in the infinite rank limit is related to the asymptotic of the Harish-Chandra--Itzykson--Zuber integral, which is described by Burgers equation as was derived in~\cite{Matytsin94} and proven in~\cite{guionnet2002}.

We now switch to looking at questions from the other skew Howe dual pairs.
To account for the sign difference in Theorem~\ref{thm:BC_z_measure_specialization}, we believe by using the ``Jack parameter'' $\xi < 0$ for the type $BC$ $z$-measure (which is analogous to the $\xi$ from the type $A$ $z$-measure), will yield a positive formula.
We note this extra parameter comes from the multivariate Jacobi polynomials~\cite{OO06,OO12}.
Furthermore, by defining $\eta = \xi/(1-\xi)$, we will obtain a parameter $0 < \eta < 1$ for a polynomial ensemble.
In particular, as a consequence of Corollary~\ref{cor:spectral_restriction}, we should obtain that these processes are specializations of a ``dual'' version of specialized Jacobi polynomial ensembles (see, \textit{e.g.},~\cite[Prop.~8.1]{BO05gamma} and~\cite[Thm.~5.7]{BK10}) in parallel to the case with $(\GL_n, \GL_k)$ with the Meixner and Krawtchouk ensembles.

\begin{conj}
The specialized $BC$ $z$-measure $\mu^{BC}_{k,1/2-l-\vartheta,\alpha,\beta}$ with the extra parameter $\xi < 0$ is equal to a discrete orthogonal polynomial ensemble with parameter $\eta = \xi/(1-\xi)$.
Moreover, this measure equals $\mu_{n,k}$ for a particular value of $\xi$.
\end{conj}

We note the the cases of $\alpha = \beta = \pm \frac{1}{2}$ correspond to kernels from Chebyshev polynomials (of the first or second kind depending on the sign), which are playing the role of the Meixner polynomials for the orthogonal and symplectic groups.
We also have the case when $\alpha = \beta = 0$ in the $BC$ $z$-measure, where the Jacobi polynomials specialize to Legendre polynomials.
However, this does not seem to correspond to a known skew Howe duality or character identity.
We remark that~\cite[($\mathsf{B}_x\mathsf{B}_y$)]{proctor1993reflection} (with all variables specialized to $1$) does not yield this identity because of the extra spin contribution, which means we use $M_1^{BC}(\lambda)$ rather than $M_1^D(\lambda)$ since $\dim V_{\SO_{2k+1}}(\mu + \fw_k) = 2^k \dim V_{\Sp_{2k}}(\mu)$.
Hence, using~\cite[($\mathsf{B}_x\mathsf{B}_y$)]{proctor1993reflection} produces the same measure as for $(\Sp_{2n}, \Sp_{2k})$ (see~\cite[Sec.~3]{proctor1993reflection} for the character identity).

\begin{problem}
Determine if there is a (skew) Howe duality for the case when $\alpha = \beta = 0$.
\end{problem}

Let us expand on the potential relationship between the extremal characters of ``big'' groups and super Howe duality.
Now we note that the skew Howe duality for the other pairs is a special case of the super Howe dualities $(\Sp_{2k}, \mathfrak{osp}(2l|2m))$ and $(\Or(k), \mathfrak{spo}(2m|2l))$ duality (see, \textit{e.g.},~\cite[Sec.~5.3]{CW12}).
In parallel to the $(\GL_n, \GL_k)$ case with $(U_{\infty} \times U_{\infty}) / U_{\infty}$, it is natural to suppose the corresponding infinite symmetric spaces in question (see Table~\ref{table:BC_values_limits}) are related to one or both of these super Howe dualities.
There is also a relationship between infinite rank Lie algebras and Lie superalgebras discussed in~\cite[Ch.~6]{CW12}.

\begin{problem}
Describe the relationship between harmonic analysis on ``big'' groups and representations of Lie superalgebras.
\end{problem}

We can push this parallel even further.
In Section~\ref{sec:howe_duality}, we noted a connection with the probability measure from~\cite[Lemma~2.2]{FN09} and Howe duality via lozenge tilings.
This yields Howe duality versus skew Howe duality based on the sign choice of $z'$ for the $BC$ $z$-measure as in the (type $A$) $z$-measure.
Furthermore, a similar picture for a half hexagon lozenge tiling is given in~\cite[Fig.~1]{BK10} in the special case of $b = c + 1$.\footnote{When we take the infinite limit $b,c \to \infty$ of the~\cite{BK10} quarter hexagon, we obtain our quarter hexagon with $k \to \infty$ after reflecting over the line $y = x$.}
Examining the probability measure from~\cite[Thm.~2.8]{BK10}, we see one factor being the dimension of a representation and the other being a determinant in the finite case.
We remark that a similar process for the symplectic group was constructed by Warren and Windridge~\cite{WW09}, which should correspond to $\alpha = \beta = \frac{1}{2}$.

From Corollary~\ref{cor:spectral_restriction}, another natural problem is to see if there is an extremal character so the spectral measure for positive $s$ from~\cite[Eq.~(16)]{OO12} corresponds to Howe duality and the orthosymplectic analog of the measure $\mu^{\square m}_{n,k}$ with possibly with $m = \infty$.
This would imply the specialized Jacobi polynomial ensemble can be described by (a restriction of) Howe duality.

\begin{problem}
Determine if there exists an extremal character $\omega$ such that the corresponding spectral measure is equal to the measure induced from Howe duality or restricted to partitions inside of an $\min(n,k) \times m$ rectangle.
\end{problem}

We note that for the extremal characters $(\mathbf{1}^m, 0, m)$ and $(\mathbf{1}^m, \mathbf{1}^k, mk)$ with $(\alpha, \beta) = (\pm\frac{1}{2},\pm\frac{1}{2})$, the matrix from~\cite[Thm.~2.4]{BK10} appears to has rational entries with positive determinants by computational evidence.
Furthermore, for $(\mathbf{1}^m, 0, m)$, the support appears to be limited to $\ell(\lambda) \leq m$.
However, these do not appear to correspond to any Howe duality.
It is possible that using $(\xi \mathbf{1}^m, 0, \xi m)$ will lead to a $\xi$-deformed version of a Cauchy identity from a Howe duality analogous to $\mu_{n,k;\xi}^{\square\infty}$ from~\eqref{eq:zeta_cauchy_measure}.

As noted in Remark~\ref{rem:spectral_det}, the matrix from~\cite[Thm.~2.4]{BK10} is distinct from our determinants for the extremal character $\omega = (0, \mathbf{1}^k, k)$.
Computational evidence suggests that all of the entries are positive integers divided by some power of $2$, which could be considered as contributing to the normalization constant.
Hence, there should be a combinatorial interpretation of these matrices $\mathcal{M}(\lambda)$ through the LGV lemma.

\begin{problem}
Find a nonintersecting lattice path interpretation of the matrices $\mathcal{M}(\lambda)$.
\end{problem}

There is also a Howe duality for $(\Sp_{2l}, \GL_k)$ and $(\GL_k, \Or_n)$~\cite[Ch.~3]{howe1995perspectives} (recall $l = \lfloor n/2 \rfloor$) decomposing the coordinate ring of the null fiber as
\begin{equation}
\label{eq:null_fiber_duality}
R(\mcN) \iso \sum_{\ell(\lambda) \leq \min(l, k)} V_{G_1}(\lambda) \otimes V_{G_2}(\lambda).
\end{equation}
Such measures were the focus in the work of Betea~\cite{Betea18}, where they were shown to be determinantal and explicit correlation kernels were computed.
On the other hand, there are other limits of symmetric spaces considered in~\cite[Table~II]{OO06}, which should produce $z$-measures.
This leads to the following problems.

\begin{problem}
Find a $z$-measure that specializes to the measure induced from~\eqref{eq:null_fiber_duality}.
\end{problem}

\begin{problem}
Determine if there are corresponding $z$-measures for the other spaces in~\cite[Table~II]{OO06} and if they correspond to a (skew) Howe duality.
\end{problem}

\bibliographystyle{alpha}
\bibliography{shapes}{}
\end{document}